\documentclass[reqno, 12pt]{amsart}
\usepackage{amsmath,amssymb,mathrsfs,amsthm,amsfonts,mathtools}
\usepackage[inline]{enumitem} 				% to change enum. items\inddown
\usepackage[usenames,dvipsnames]{xcolor}
\usepackage{graphicx} 
\usepackage[paper=letterpaper,margin=1in]{geometry}
\definecolor{hypercolor}{HTML}{003399}
\usepackage{hyperref}
 	
\hypersetup{colorlinks=true, linkcolor=hypercolor,citecolor=hypercolor}
\usepackage[margin=3pt, font=small]{caption}	%change figure caption width
\newtheorem{thm}{Theorem}[section]
\newtheorem{lem}[thm]{Lemma}
\newtheorem{prop}[thm]{Proposition}
\newtheorem{cor}[thm]{Corollary}
\theoremstyle{definition}

\newtheorem{conv}[thm]{Convention}
\newtheorem*{notation}{Notation for Theorem~\ref{t.matching}}

\theoremstyle{remark}
\newtheorem{assu}[thm]{Assumption}
\newtheorem{rmk}[thm]{Remark}

\numberwithin{equation}{section}
\allowdisplaybreaks

\newcommand{\betaloc}{\beta\text{-}\mathrm{loc}}% beta-localized
\newcommand{\Start}{\mathrm{sc}}			% starting condition
\newcommand{\End}{\mathrm{ec}}				% ending condition
\newcommand{\comple}{\mathrm{c}}			% complement
\newcommand{\old}{\text{\tiny{}old}}			
\newcommand{\new}{\text{\tiny{}new}}	
\newcommand{\sprime}{{\scriptscriptstyle\prime}}

\newcommand{\e}{\varepsilon}
\newcommand{\calD}{\mathcal{D}}
\newcommand{\calE}{\mathcal{E}}
\newcommand{\calF}{\mathcal{F}}

\newcommand{\calU}{\mathcal{U}}
\newcommand{\R}{\mathbb{R}}
\newcommand{\Z}{\mathbb{Z}}

\newcommand{\Csp}{\mathscr{C}}
\newcommand{\Cbsp}{\mathscr{C}_\mathrm{b}}
\newcommand{\Ccsp}{\mathscr{C}_\mathrm{c}}
\newcommand{\Lsp}{\mathcal{L}}
\newcommand{\Psp}{\mathscr{P}}				% the space of probability measures

\newcommand{\rate}{\mathcal{I}}				% for attractive BM
\newcommand{\rateM}{\Lambda}				% the log moment generating function for \rate
\newcommand{\rateq}{\mathbb{I}}				% quantile representation for attractive BM
\newcommand{\mom}{\mathbb{L}}				% moment Lyapunov
\newcommand{\momshe}{L_{\scriptscriptstyle\mathrm{SHE}}}% moment Lyapunov, she
\newcommand{\ratekpz}{I_{\scriptscriptstyle\mathrm{KPZ}}}
\newcommand{\ratebm}{I_{\scriptscriptstyle\mathrm{BM}}}
\newcommand{\hf}[1]{\mathsf{f}_{\star#1}}	% terminal height function
				
\newcommand{\E}{\mathbb{E}}
\renewcommand{\P}{\mathbb{P}}
\newcommand{\EE}{\mathbf{E}}				% expectation over noise
\newcommand{\law}{\mathbb{Q}}				% a generic law
\newcommand{\Ebm}{\E_{\scriptscriptstyle\mathrm{BM}}}	
\newcommand{\Pbm}{\P_{\scriptscriptstyle\mathrm{BM}}}	
\newcommand{\termt}{\mathfrak{t}}			% terminal time
\newcommand{\termts}{\termt\hspace{1pt}}	% terminal time + \hspace{1pt}
\newcommand{\intermt}{\mathfrak{t}^{\sprime}}% intermediate time
\newcommand{\xx}{\mathbf{x}}				% points of conditioning
\newcommand{\vecxx}{\vec{\xx}}				% vector \xx
\newcommand{\xxi}{\mathbf{x}^{\sprime}}		% points of conditioning, intermediate time
\newcommand{\vecxxi}{\vec{\mathbf{x}}{}^{\hspace{1pt}\sprime}}%vector \xxi
\newcommand{\mm}{\mathfrak{m}}				% mass
\newcommand{\vecmm}{\vec{\mm}}				% vector \mm
\newcommand{\mmi}{\mm^{\hspace{1pt}\sprime}}% mass, intermediate time
\newcommand{\vecmmi}{\vecmm{}^{\sprime}}	% vector \mmi
\newcommand{\MM}{M}							% cumulative mass

\newcommand{\hv}{\mathbf{r}}				% height values
\newcommand{\vechv}{\vec{\hv}}				% vector \hv
\newcommand{\vechvi}{\vec{\hv}{}^{\hspace{1pt}\sprime}}% vector \hvi
\newcommand{\hvspace}{\mathscr{R}}			% set of \hv
\newcommand{\hvspacec}{\hvspace_\mathrm{conc}}% \hvspace_concave

\newcommand{\sgn}{\mathrm{sgn}}				% the sign function
\newcommand{\Sgn}{\Phi}						% integrating \sgn wrt to a given measure
\newcommand{\quant}{\mathfrak{X}}			% quantile or inverse cdf
\newcommand{\cdf}{F}						% cdf
\newcommand{\pull}{\phi}					% drift

\newcommand{\jj}{j}							% the particle index
\newcommand{\ii}{i}							% the particle index
\newcommand{\cc}{\mathfrak{c}}				% the cluster index
\newcommand{\group}{\mathfrak{I}}			% the group of particles that share a cluster
\newcommand{\velo}{w}						% velocity function in initial and transition segments

\newcommand{\optimal}{\boldsymbol{\xi}}		% optimal deviation
\newcommand{\optimalc}[1]{\boldsymbol{\xi}_{#1}}% optimal cluster
\newcommand{\dotoptimalc}[1]{\dot{\boldsymbol{\xi}}_{#1}}
\newcommand{\inerc}[1]{\boldsymbol{\zeta}_{#1}}% inertia cluster
\newcommand{\bb}{\mathfrak{b}}
\newcommand{\Branch}{\mathfrak{B}}
\newcommand{\diner}{v}						% drift applied to the inertia cluster
\renewcommand{\aa}{\mathfrak{a}}			% the cluster index, intermediate time
\newcommand{\Aranch}{\mathfrak{C}}			% the Branch for \aa
\newcommand{\shock}[1]{\boldsymbol{\sigma}_{#1}}% shock
\newcommand{\dotshock}[1]{\dot{\boldsymbol{\sigma}}_{#1}}% shock
\newcommand{\Branchs}{\mathfrak{B}_\star}	% Branch set defined by cones
\newcommand{\cone}{\mathcal{C}}				% the conic regime
\newcommand{\tangent}{\mathbf{y}}			% tangent points

\newcommand{\scale}{\mathfrak{S}}
\newcommand{\shift}{\mathsf{T}}
\newcommand{\HL}{\mathrm{HL}}
\newcommand{\HLop}{\mathrm{HL}^{\mathrm{bk}}}

\newcommand{\ind}{\mathbf{1}}				% the indicator function
\newcommand{\supp}{\mathrm{supp}}			% support
\renewcommand{\d}{\mathrm{d}}
\newcommand{\gdstate}{\psi_{\scriptscriptstyle\mathrm{gd}}}% ground state 
\newcommand{\noise}{\eta}					% spacetime white noise
\newcommand{\parab}{\mathsf{p}}				% p(t,x) = -x^2/(2t)

\newcommand{\f}{\mathsf{f}}
\newcommand{\g}{\mathsf{g}}
\newcommand{\h}{\mathsf{h}}
\renewcommand{\u}{\mathsf{u}}

\newcommand{\mps}{\h_\mathrm{\star}}		% most probable shape
\newcommand{\mpu}{\mathsf{u}_\mathrm{\star}}% most probable u
\newcommand{\mpuu}[1]{\mathsf{u}_{\star,#1}}				

\newcommand{\hh}{h}							% solution of KPZ
\newcommand{\ZZ}{Z}							% solution of SHE
\newcommand{\EM}{\boldsymbol{\mu}}			% empirical measure
\newcommand{\tEM}{\boldsymbol{\nu}}			% truncated empirical measure
\newcommand{\X}{X}							% the canonical process in Feynman--Kac
\newcommand{\Y}{Y}							% the particles
\newcommand{\bm}{B}							% the BM that drives \X

\newcommand{\tparti}{\Gamma}				% partition in time
\newcommand{\control}{\mathcal{C}}			% the controlling event, linear segment
\newcommand{\controll}{\mathcal{C}'}		% the controlling event, transition segment
\newcommand{\controlll}{\mathcal{C}''}		% the controlling event, s_0 segment
\newcommand{\cost}{\mathrm{Cost}}

\newcommand{\DIST}{\mathfrak{D}\mathrm{ist}}% max of |\X_i - \xi_{\cc}|
\newcommand{\Dist}{\mathrm{Dist}}			% localizing \X
\newcommand{\dist}{\mathrm{dist}}			% metric on \Psp
\newcommand{\wass}{\mathcal{W}}				% Wasserstein metric
\newcommand{\norm}[1]{\Vert #1\Vert}

\newcommand{\ip}[1]{\langle #1\rangle}
\newcommand{\Ip}[1]{\big\langle #1\big\rangle}

\renewcommand{\bar}{\overline}
\newcommand{\und}{\underline}
\newcommand{\til}{\widetilde}

\newcommand*{\Cdot}{{\raisebox{-0.5ex}{\scalebox{1.8}{$\cdot$}}}} % largedot

\usepackage{newtxtext}

\graphicspath{{./figs/}}

\title{High moments of the SHE in the clustering regimes}
\author{Li-Cheng Tsai}

\address[Li-Cheng Tsai]{\hspace{1.5pt} Department of Mathematics, University of Utah}

\subjclass[2020]{}%
\keywords{}

\begin{document}
\begin{abstract}
We analyze the high moments of the Stochastic Heat Equation (SHE) via a transformation to the attractive Brownian Particles (BPs), which are Brownian motions interacting via pairwise attractive drift.
In those scaling regimes where the particles tend to cluster, we prove a Large Deviation Principle (LDP) for the empirical measure of the attractive BPs.
Under the delta(-like) initial condition, we characterize the unique minimizer of the rate function and relate the minimizer to the spacetime limit shapes of the Kardar--Parisi--Zhang (KPZ) equation in the upper tails.
The results of this paper are used in the companion paper \cite{lin23} to prove an $n$-point, upper-tail LDP for the KPZ equation and to characterize the corresponding spacetime limit shape.
\end{abstract}

\maketitle

\section{Introduction}
\label{s.intro}
This paper is motivated by the study of the large deviations and spacetime limit shapes of the Kardar--Parisi--Zhang (KPZ) equation.
Introduced in \cite{kardar86}, the KPZ equation 
\begin{align}
	\label{e.kpz}
	&&
	\partial_t h = \tfrac12 \partial_{xx} h + \tfrac12 (\partial_{x} h)^2 + \noise,
	&&
	h=h(t,x),
	\ \ 
	(t,x)\in (0,\infty)\times\R
\end{align} 
describes the evolution of a randomly growing interface, where $\noise=\noise(t,x)$ denotes the spacetime white noise.
This equation plays a central role in nonequilibrium statistical mechanics and has been widely studied in mathematics and physics; we refer to \cite{quastel2011introduction,corwin2012kardar,quastel2015one,chandra2017stochastic,corwin2019} for reviews on the mathematical literature related to the KPZ equation.

This paper and the companion paper \cite{lin23} use the moments of the Stochastic Heat Equation (SHE) to obtain a Large Deviation Principle (LDP) and the corresponding spacetime limit shape of the KPZ equation.
Recall that the SHE
\begin{align}
	\label{e.she}
	\partial_t Z = \tfrac12 \partial_{xx} Z + \noise Z
\end{align}
gives the solution of \eqref{e.kpz} via $h:=\log Z$, so the moment generating function of the KPZ equation is related to the moments of the SHE.
This relation serves as a gateway to studying the KPZ equation, since the moments of the SHE are accessible from a number of tools, including the delta Bose gas and Feynman--Kac formula.
At the one-point level, this relation has led to fruitful results on the one-point upper tail of the KPZ equation~\cite{chen15,corwin20general,das21,das21iter,lin21half,ghosal20}.

The goal of this paper is to characterize the \emph{multi-point} moment Lyapunov exponents 
\begin{align}
	\label{e.momentL}
	\lim_{N\to\infty}
	\frac{1}{N^3T_N} \log \EE\Big[ \prod_{\cc=1}^n Z(T\termt,NT\xx_\cc)^{N\mm_\cc} \Big]
\end{align}
for all positive powers, $N\mm_\cc\in(0,\infty)$, and for fixed $\xx_1<\ldots<\xx_n\in\R$.
Hereafter $T=T_N$ is the scale of time, and $N$ is the scale of the powers.
The only conditions we impose on $N$ and $T=T_N$ are 
\begin{align}
	\label{e.scaling}
	N\to\infty,
	\qquad
	N^2 T = N^2 T_N \to \infty.	
\end{align}
Our results actually hold for any $T=T_{N,A}:=A/N^2$ with $(N,A)\to(\infty,\infty)$; we take $T=T_N$ to depend only on $N$ for the convenience of notation.
Note that \eqref{e.scaling} allows $T_N\to 0$, $T_N\to 1$, and $T_N\to\infty$.
The second condition in \eqref{e.scaling} underscores what we call the clustering behaviors of the attractive Brownian Particles, as explained in the paragraph after next.
The first condition in \eqref{e.scaling} reduces the task of characterizing \eqref{e.momentL} for positive powers, $N\mm_\cc\in(0,\infty)$, to that of integer powers.
Indeed, $N\mm_\cc\in\Z_{>0}$ translates into $\mm_\cc\in\frac{1}{N}\Z_{>0}$, which becomes a denser and denser subset of $(0,\infty)$ as $N\to\infty$.

Let us briefly describe our methods and results; the full description will be given in Section~\ref{s.results}.
To analyze the moment Lyapunov exponents, we express the integer moments of the SHE in terms of a system of attractive Brownian Particles (BPs), which are Brownian Motions (BMs) interacting via pairwise attractive drift.
The task of analyzing the moment Lyapunov exponents turns into proving the LDP for the attractive BPs.
Theorem~\ref{t.ldp} gives the sample-path LDP for the empirical measure of the attractive BPs.
Next, we specialize the initial condition into the delta-like initial condition (defined in \eqref{e.deltalike.ic}).
Theorem~\ref{t.optimal} explicitly characterizes the unique minimizer of the rate function (of the LDP for the attractive BPs) under the delta-like initial condition.
We call the minimizer the \textbf{optimal deviation}.
Corollary~\ref{c.momentL} gives the limit of \eqref{e.momentL} under the delta-like initial condition.
Finally, in Theorem~\ref{t.matching}, we show how the moment Lyapunov exponents and optimal deviation are related to the corresponding rate function and limit shape of the KPZ equation. 
The results of this paper are used in the companion paper \cite{lin23} to prove an $n$-point, upper-tail LDP for the KPZ equation and to characterize the corresponding spacetime limit shape. 

Under \eqref{e.scaling}, the attractive BPs exhibit what we call the \textbf{clustering behavior}.
Two effects contribute to the evolution of the attractive BPs: the attractive drift and diffusive effect.
The drift pulls the particles together, while the diffusive effect spreads them out.
As will be explained in Section~\ref{s.results.ldp}, under the second condition in \eqref{e.scaling}, the attractive drift dominates the diffusive effect, so the BPs tend to cluster.
Theorem~\ref{t.optimal} (also Theorem~\ref{t.matching}\ref{t.matching.shape}) shows that, under the delta-like initial condition, the optimal deviation consists of a number of clusters and explicitly describes the spacetime trajectories of the clusters.
In the context of the upper-tail LDPs for the KPZ equation, the clustering behavior (of the attractive BPs) corresponds to what is called the noise-corridor effect in \cite[Sect.\ 1.3]{lin23}.
Those noise corridors are exactly the trajectories of the clusters.

Our proof does not rely on integrability or explicit formulas.
Thanks to the integrability of the delta Bose gas, moments of the SHE enjoy explicit formulas.
The formulas offer a potential path to obtain the limit in \eqref{e.momentL}, but they do not seem to provide information for proving a localization result like Corollary~\ref{c.momentL}\ref{c.momentL2}, which is a crucial technical input for the proof in \cite{lin23}. 
Still, it is interesting to see if one can obtain the limit in \eqref{e.momentL} also from the formulas.

Let us compare \eqref{e.scaling} with two commonly-considered scaling regimes.
First, $T_N=A/N^2$, with $A<\infty$ fixed and $N\to\infty$, corresponds to the Freidlin--Wentzell/weak-noise LDP for the KPZ equation.
At the level of the attractive BPs, this is the diffusive regime, and the LDP is proven in \cite{dembo2016large} for a general class of rank-based diffusions that includes the attractive BPs.
The LDP in \cite{dembo2016large} and the LDP proven here are very different in nature, as will be explained in Section~\ref{s.results.ldp}.
Next, $N=1$ and $T\to\infty$ correspond to the hyperbolic scaling regime, in long time, of the KPZ equation, namely $\hh_T(t,x):=T^{-1}\hh(Tt,Tx)$.
This is perhaps the most natural scaling regime in long time, while \eqref{e.scaling} allows us to probe any deviation much larger than those in the hyperbolic regime.
For the hyperbolic scaling regime, since $N^2T=T\to\infty$, we also expect the clustering behaviors to happen.
One may seek to generalize our approach to obtain the multipoint, positive-integer moment Lyapunov exponents in this regime.
This has been achieved by \cite{lin2023multi}, based on exact formulas as well as ideas about the optimal clusters from Section~\ref{s.results.momentL} of this paper.

We end the introduction with a brief literature review.
The moments of the SHE and its variant have been used to study the intermittency property \cite{gartner1990parabolic,gartner2007geometric,khoshnevisan14}, large deviations, and the density function of the SHE (and its variants) in \cite{conus2013chaotic,georgiou2013large,borodin2014moments,chen15,chen2015moments,khoshnevisan17,janjigian15,hu2018asymptotics,corwin20general,chen2021regularity,das21,das21iter,lin21half}.
The connection between the delta Bose gas and attractive BPs has been used in the physics work~\cite{ledoussal2022ranked} to study their relaxation properties and their hydrodynamic limits.
The attractive BPs is a special case of diffusions with rank-based drifts, or rank-based diffusions for short.
We refer to 
\cite{%
fernholz02,%
banner2005atlas,%
pal08,%
fernholz09,%
ichiba2010collisions,%
banner2011hybrid,%
shkolnikov11,%
ichiba2013strong,%
sarantsev2015triple,%
dembo2016large,%
dembo2017equilibrium,%
ichiba2017yet,%
sarantsev2017infinite,%
sarantsev2017two,%
sarantsev2017stationary,%
tsai2018stationary,%
cabezas2019brownian,%
dembo2019infinite,%
sarantsev2019camparison,
banerjee2022domains,
banerjee2023dimension%
}
and the references therein for the literature on rank-based diffusions.
Recently, there has been much interest in the LDPs for the KPZ equation in mathematics and physics.
Several strands of methods produce detailed information on the one-point tail probabilities and the one-point rate function.
This includes the physics works \cite{ledoussal16short,ledoussal16long,krajenbrink17short,sasorov17,corwin18,krajenbrink18half,krajenbrink18simple,krajenbrink18systematic,krajenbrink18simple,krajenbrink19thesis,krajenbrink19four,ledoussal19kp}, the simulation works \cite{hartmann18,hartmann19,hartmann20,hartmann21}, and the mathematics works \cite{corwin20general,corwin20lower,cafasso21airy,das21,das21iter,kim21,lin21half,cafasso22,ghosal20,ganguly22,tsai22exact}.
For the Freidlin--Wentzell/weak-noise LDP, behaviors of the one-point rate function and the corresponding most probable shape(s) for various initial conditions and boundary conditions have been predicted \cite{kolokolov07, kolokolov08, kolokolov09, meerson16, kamenev16, meerson17, meerson18, smith18exact, smith18finitesize, asida19, smith19}, some of which recently proven \cite{lin21,lin22,gaudreaulamarre2023kpz}; an intriguing symmetry breaking and second-order phase transition has been discovered in~\cite{janas16,smith18} via numerical means and analytically derived in \cite{krajenbrink17short,krajenbrink22flat}; a connection to integrable PDEs is recently established and studied in the physics works \cite{krajenbrink20det,krajenbrink21,krajenbrink22flat,krajenbrink23} and the mathematically rigorous work \cite{tsai22}.

\subsection*{Outline}
In Section~\ref{s.results}, we state the results and present some discussions.
In Section~\ref{s.basic}, we introduce some notation, definitions, and tools.
Sections~\ref{s.ratefn}--\ref{s.lwbd} make up the proof of Theorem~\ref{t.ldp}, the LDP for the attractive BPs: We establish properties of the rate function in Section~\ref{s.ratefn}, prove the LDP upper bound in Section~\ref{s.upbd}, and prove the LDP lower bound in Section~\ref{s.lwbd}.
In Section~\ref{s.momentL}, we specialize the setting into the delta-like initial condition and establish results on the moment Lyapunov exponents, stated as Theorem~\ref{t.optimal} and Corollary~\ref{c.momentL}.
Finally, in Section~\ref{s.matching}, we prove Theorem~\ref{t.matching} that relates the moment Lyapunov exponent and optimal deviation to the rate function and spacetime limit shape of the KPZ equation. 

\subsection*{Acknowledgments}
I thank Ivan Corwin and Jeremy Quastel for showing me the derivation described in Remark~\ref{r.conjugation} and thank Ivan Corwin for suggesting the possibility of extracting tail information from the exact formulas of the integer moments of the SHE.
The knowledge has inspired me to pursue this work.
I thank Jeremy Quastel for many useful discussions about the large deviations of KPZ-related models and thank Yau-Yuan Mao for the help with some of the figures in this paper.
I thank the anonymous reviewers for their careful reading of and useful comments on the manuscript.
The research of Tsai was partially supported by the NSF through DMS-2243112 and the Alfred P. Sloan Foundation through the Sloan Research Fellowship FG-2022-19308.

\section{Results and discussions}
\label{s.results}

\subsection{Moment Lyapunov exponents, attractive BPs}
\label{s.results.aBP}
The first goal of this paper is to prove that the $n$-point Lyapunov exponent \eqref{e.momentL} exists and to characterize it, for suitable initial conditions for $\ZZ$.
In \eqref{e.momentL} and hereafter, $\termt\in(0,\infty)$, $\xx_1<\ldots<\xx_n\in\R$, and $\mm_1,\ldots,\mm_n\in(0,\infty)$.
Throughout this paper we will only work with integer moments, so an integer part is implicitly taken whenever needed; for example, $N\mm_\cc := \lceil N\mm_\cc\rceil$ in \eqref{e.momentL}.
As will be shown in Theorem~\ref{t.matching}\ref{t.matching.legendre}, the limit of \eqref{e.momentL} is continuous in $\vecmm\in[0,\infty)^n$.
Hence, once we obtain the limit for $\lceil N\vecmm\rceil\in(\Z_{> 0})^n$, the result automatically extends to $ N\vecmm\in [0,\infty)^n$.

To access \eqref{e.momentL}, let us express the moments in terms of the attractive BPs.
First, by \cite[Theorem~5.3]{hu2009stochastic}, the one-point moment of $\ZZ$ can be expressed as an expectation over independent BMs.
The same proof there works for multipoint (in space) moments, giving, for $y_1,\ldots,y_{N\mm}\in\R$,
\begin{align}
	\label{e.feynmankac}
	\EE\Big[ \prod_{\ii=1}^{N\mm} \ZZ(t,y_{\ii}) \Big]
	= 
	\Ebm\Big[ e^{\sum_{\ii<\jj} \int_0^{t} \d s \, \delta_0(\X_\ii(t-s)-\X_\jj(t-s))} \prod_{\ii=1}^{N\mm} Z(0,\X_\ii(t)) \Big],
\end{align}
where $X_1(s)-y_1$, \ldots, $X_{N\mm}(s)-y_{N\mm}$ are independent standard BMs under $\Ebm$, and the integrals are interpreted as localtimes.
Hereafter, $\mm\in(0,\infty)$ is fixed and is taken to be $\mm=\mm_1+\ldots+\mm_n$ whenever we analyze \eqref{e.momentL}.
The localtimes in \eqref{e.feynmankac} can be removed by transforming $\Ebm$ into a different law $\E$.
Let $\P$ and $\E$ denote the law and expectation under which $\X_1,\ldots,\X_{N\mm}$ evolve as a system of attractive BPs as
\begin{align}
	\label{e.aBP}
	\d \X_{\ii}(s) = \sum_{\jj=1}^{N\mm} \frac{1}{2} \sgn(\X_{\jj}(s)-\X_{\ii}(s)) \, \d s + \d \bm_\ii(s),
	\qquad
	\X_{\ii}(0) = y_{\ii},
\end{align}
where $\bm_1,\ldots,\bm_{N\mm}$ are independent standard BMs, and $\sgn(x):=(x/|x|)\ind_{\{x\neq 0\}}$.
We index the particles so that they are ordered at the start: $\X_1(0)\leq\X_2(0)\leq\ldots$, and the particles can exchange orders as time evolves.
By \cite{ichiba2013strong}, the equations \eqref{e.aBP} have a unique strong (pathwise) solution.
Start from $\Pbm$ (under which $\X_1,\ldots,\X_{N\mm}$ are independent BMs), apply Tanaka's formula to $\X_{\ii}-\X_{\jj}$, sum the result over $\ii<\jj\in\{1,\ldots,N\mm\}$, and exponentiate the result.
Doing so gives
\begin{subequations}
\label{e.tanaka}
\begin{align}
	\label{e.tanaka1}
	\text{under }\Pbm,\quad &
	e^{\sum_{\ii<\jj} \int_0^t \d s \, \delta_0(\X_\ii-\X_\jj)}
	=
	\exp\Big( -\sum_{\ii<\jj} \int_0^t \d (\X_{\ii}-\X_{\jj}) \, \frac12 \sgn(\X_\ii-\X_\jj) \Big)
\\
	\label{e.tanaka2}
	&
	\Cdot\exp\Big( -\sum_{\ii<\jj}\frac{1}{2}|\X_\ii(0)-\X_\jj(0)| +\sum_{\ii<\jj}\frac{1}{2}|\X_\ii(t)-\X_\jj(t)| \Big).
\end{align}
\end{subequations}
Rewrite the sum in \eqref{e.tanaka1} as $-\frac{1}{2}\sum_{\ii=1}^{N\mm} \int_0^t \d\X_{\ii}\, \sum_{\jj:\jj\neq\ii}\sgn(\X_{\ii}-\X_{\jj})$.
As a process in $t$, this term has quadratic variation $tN\mm ((N\mm)^2-1)/12$; see \cite[Eq.~(2.1)]{chen15}.
Given this property, Girsanov's theorem asserts that
\begin{align}
	\label{e.girsanov}
	\frac{\d\Pbm}{\d\P~~}(\vec{\X})
	=
	\big(\text{right hand side of \eqref{e.tanaka1}}\big) \cdot e^{N\mm((N\mm)^2-1)\frac{t}{24}}.
\end{align}
Let $\gdstate(\vec{x}):=\exp(-\frac12\sum_{\ii<\jj}|x_\ii-x_\jj|)=\exp(-\frac14\sum_{\ii,\jj=1}^{N\mm}|x_\ii-x_\jj|)$.
Combining \eqref{e.tanaka}--\eqref{e.girsanov} gives
\begin{align}
\label{e.transformation.aBP}
	\Ebm\big[e^{\sum_{\ii<\jj} \int_0^{t} \d s \, \delta_0(\X_\ii(t-s)-\X_\jj(t-s))} \big(\, \Cdot \, \big)\big]
	=
	e^{N\mm((N\mm)^2-1)\frac{t}{24}}
	\
	\E\big[ \gdstate(\vec{\X}(0))\big(\, \Cdot \, \big)/\gdstate(\vec{\X}(t)) \big].
\end{align}
We view both sides of \eqref{e.transformation.aBP} as operators that act on a measurable function $F(\vec{\X})$ of $(\X_{\ii}(\cdot))_{\ii=1,\ldots,N\mm}$ such that $\gdstate(\vec{\X}(0))F(\vec{\X})/\gdstate(\vec{\X}(t))$ is integrable under $\E$.

\begin{rmk}\label{r.conjugation}
A similar application of Tanaka's formula was used in \cite{chen15} to access the one-point integer moments of the SHE.
The transformation \eqref{e.transformation.aBP} can also be derived by conjugating the Hamiltonian of the delta Bose gas by its ground state.
This is done in \cite[Eq.~(5)--(6)]{ledoussal2022ranked} at a physics level of rigor.
\end{rmk}
 
To analyze the attractive BPs, we will mostly work with the (scaled) empirical measure 
\begin{align}
	\label{e.empirical.measure}
	\EM_N(s) := \frac{1}{N}\sum_{\ii=1}^{N\mm} \delta_{X_{\ii}(Ts)/(N T)}.
\end{align}
Hereafter, we write $T=T_N$, with the understanding that $N\to\infty$ is always taken under \eqref{e.scaling}.
Writing $\ip{\lambda,\f}:=\int_{\R}\lambda(\d x)\, \f(x)$ for the action of a Borel measure $\lambda$ on $\f$, we rewrite \eqref{e.transformation.aBP} as
\begin{subequations}
\label{e.transformation.aBP.}
\begin{align}
	\label{e.transformation.aBP.l}
	\Ebm&\big[e^{\sum_{\ii<\jj} \int_0^{T\termt} \d s \, \delta_0(\X_\ii(T\termt-s)-\X_\jj(T\termt-s))} (\,\Cdot\,) \big]
\\
	\label{e.transformation.aBP.r}
	&=
	\exp\big( N^3T\mm^3\,(1-\tfrac{1}{(N\mm)^2}) \tfrac{\termt}{24} \big)
	\
	\E\Big[ \exp\Big( N^3T\Ip{\EM_N(s)^{\otimes2},\tfrac{1}{4}}|x-x'|\big|_{s=0}^{s=\termt} \Big) \ \big(\ \Cdot\ \big)\Big],
\end{align}
\end{subequations}
where $\lambda^{\otimes2}:=\lambda\otimes\lambda$ stands for the product measure that acts on $\R^2=\{(x,x')\}$.

\begin{conv}
\label{conv.time}
We have and will continue to use $t$ to denote the time variable of the SHE, and use $s$ to denote the time variable of $\X_1,\X_2,\ldots$.
We call $\ZZ(t,\Cdot)|_{t=0}$ and $\ZZ(t,\Cdot)|_{t=N\termt}$ respectively the \textbf{initial} and \textbf{terminal conditions} of the SHE, and call $\X_{\ii}(s)|_{s=0}$ and $\X_{\ii}(s)|_{s=N\termt}$ respectively the \textbf{starting} and \textbf{ending conditions} of $\X_\ii$.
Since $t=N\termt-s$, there is a \emph{time reversal} between $t$ and $s$.
\end{conv}

\subsection{Result: LDP for the attractive BPs}
\label{s.results.ldp}
The task of analyzing the moment Lyapunov exponent~\eqref{e.momentL} turns into studying the large deviations of the attractive BPs.
Indeed, the quantities in \eqref{e.transformation.aBP.} are of exponential scales, so the expectation is controlled by the large deviations of the attractive BPs.

Our result, Theorem~\ref{t.ldp}, gives the sample-path LDP for $\EM_N$.

Let us prepare some notation for Theorem~\ref{t.ldp}.
Let $(\mm\Psp(\R))$ denote the space of positive Borel measures on $\R$ with total mass $\mm$, and endow the space with the weak* topology, namely the topology induced by convergence in distribution, which is denoted by $\Rightarrow$.
Consider $\Csp([0,\termts],\mm\Psp(\R))$.
A sequence $\mu_k$ converging to $\mu$ in this space means $\sup_{s\in[0,\termts]}|\ip{\mu_k(s)-\mu(s),\varphi}|\to 0$ for any $\varphi\in\Csp_\mathrm{b}(\R)$.
Endow $\Csp([0,\termts],\mm\Psp(\R))$ with the topology induced by the convergence sequences.
Namely, a set $A$ in $\Csp([0,\termts],\mm\Psp(\R))$ is closed if and only if $A$ contains the limit of any sequence $\{\mu_k\}_{k=1,2,\ldots}\subset A$ that converges.
In Section~\ref{s.basic.topo}, we will introduce metrics for these topologies.
Given a topological space $\mathscr{X}$, a sequence of $\mathscr{X}$-valued random variables $R_N$ \textbf{satisfying the LDP with the speed $s_N$ and rate function $I$} means that
$
	-\inf_{O} I 
	\leq 
	\liminf_{N\to\infty}\frac{1}{s_N}\log\P[R_N\in O]
$
and
$
	\limsup_{N\to\infty}\frac{1}{s_N}\log\P[R_N\in C]
	\leq
	-\inf_C I	
$
for every open $O\subset\mathscr{X}$ and closed $C\subset\mathscr{X}$.
A function $I:\mathscr{X}\to[0,\infty]$ being a \textbf{good rate function} means that its is lower-semicontinuous and $\{I\leq r\}$ is compact for any $r<\infty$.

More notation for Theorem~\ref{t.ldp} is in place.
For $\lambda\in\mm\Psp(\R)$, set
\begin{align}
	\label{e.Sgn}
	\Sgn[\lambda](x) 
	:= 
	\ip{\lambda, \tfrac12 \sgn(\Cdot-x)}
	= 
	- \tfrac12 \ip{\lambda,\ind_{(-\infty,x)}} + \tfrac12 \ip{\lambda,\ind_{(x,\infty)}}.
\end{align}
Recall that $\sgn(0):=0$, so the right hand side excludes the mass of $\lambda$ at $x$.
This function gives the amount of drift a particle feels in the system of attractive BPs.
More explicitly, the drift in \eqref{e.aBP} is equal to $N\Sgn[\lambda](x)$ with $\lambda=\EM_N(s)$ and $x=X_\ii(Ts)/(NT)$.
Next, let $ {\Cbsp}^{1,1}([0,\termts],\R) $ be the space of functions $\h=\h(s,x)$ with continuous and bounded first derivatives in time and space.
Let
\begin{align}
	\label{e.rateM}
	\rateM(\mu,\h)
	&:=
	\ip{ \mu(s), \h(s) }\big|^{s=\termt}_{s=0}
	-
	\int_{0}^{\termt} \d s \, \ip{ \mu, \partial_s \h} 
	-
	\int_{0}^{\termt} \d s \, \ip{ \mu, \Sgn[\mu] \partial_x \h},
\\
	\label{e.rate}
	\rate(\mu)
	&:=
	\sup\Big\{ \rateM(\mu,\h) - \int_{0}^{\termt} \d s \, \frac12 \Ip{ \mu, (\partial_x \h)^2 } \ : \ \h \in \Cbsp^{1,1}([0,\termts],\R) \Big\}.
\end{align}
Hereafter, we adopt shorthand notation such as $\h(s):=\h(s,\Cdot)$, $\ip{ \mu, \partial_s \h} := \ip{ \mu(s), \partial_s \h(s)} $, and $\Sgn[\mu]:=\Sgn[\mu(s)](x)$.
Expressions like \eqref{e.rate} arise in the martingale method \cite{kipnis89} for proving LDP upper bounds.
To prove the LDP lower bound and to analyze $\rate$, it will be convenient to have an expression more explicit than \eqref{e.rate}.
Given any $\lambda\in\mm\Psp(\R)$, let
\begin{align}
	\label{e.cdf.quant}
	\cdf[\lambda](x) := \ip{ \lambda, \ind_{(-\infty,x]} },
	\qquad
	\quant[\lambda](a) := \inf\{ x\in\R : a\leq \cdf[\lambda](x) \},
\end{align}
be the Cumulative Distribution Function (CDF) and quantile function, aka inverse CDF.
We refer to $a\in[0,\mm]$ as the \textbf{quantile coordinate}.
As will be shown in Section~\ref{s.rate.quantile}, the rate function $\rate$ permits a \textbf{quantile representation} $\rateq$.
Namely, $\rate=\rateq$, where
\begin{align}
	\label{e.rateq}
	\rateq(\mu)
	:=
	\int_0^{\termt} \d s \int_0^\mm \d a \,
	\frac12 \big( \partial_s (\quant[\mu]) - \Sgn[\mu](\quant[\mu]) \big)^2,
\end{align}
where $\quant[\mu]:=\quant[\mu(s)](a)$.
Given that $|\Sgn[\mu]|\leq \mm/2$, when $\partial_s(\quant[\mu])\notin\Lsp^2([0,\termts]\times[0,\mm])$, we define the integral in \eqref{e.rateq} to be $+\infty$.

\begin{thm}\label{t.ldp}
Fix any $\termt,\mm\in(0,\infty)$ and a starting condition $ \mu_\Start \in \mm\Psp(\R) $. Assume \eqref{e.scaling}.

Started from a deterministic $\EM_N(0)$ such that $ \EM_N(0) \Rightarrow \mu_\Start $, the empirical measure $ \EM_N $ satisfies the LDP on $ \Csp([0,\termts],\mm\Psp(\R)) $ with the speed $N^3 T=N^3T_N$ and the rate function
\begin{align}
	\label{e.rate.ic}
	\rate_\star(\mu)
	:=
	\begin{cases}
		\rate(\mu)=\rateq(\mu)	& \text{when } \mu(0) = \mu_\Start,
	\\
		+\infty& \text{when } \mu(0) \neq \mu_\Start,
	\end{cases}
\end{align}
and $ \rate_\star $ is a good rate function.
\end{thm}
\noindent{}%
We will refer to $\rate_\star$, $\rate$, and $\rateq$ all as rate functions; as shown in Sections~\ref{s.ratefn}, they are lower semicontinuous, nonnegative, and not identically infinite.

%\subsection{Discussions on Theorem~\ref{t.ldp}}
%\label{s.results.discussldp}
As mentioned previously, the condition $N^2T\to\infty$ underscores the clustering behaviors.
In \eqref{e.aBP}, the drift pulls the particles together, while the diffusive effect spreads them out.
For $O(N)$ particles evolving over time $O(T)$, the contribution of the drift is $O(NT)$, while the contribution of the diffusive effect is $O(\sqrt{T})$.
The condition $N^2T\to\infty$ amounts to saying that the drift dominates the diffusive effect.
Under this condition, the attractive BPs tend to cluster, so we refer to the regimes with $N^2T\to\infty$ as the \textbf{clustering regimes}.
In these regimes, deviations $\mu\in\Csp([0,\termts],\mm\Psp(\R))$ with \emph{atoms} are relevant.

We emphasize that Theorem~\ref{t.ldp} differs from the LDP proven in \cite{dembo2016large}.
The work \cite{dembo2016large} considered a general class of rank-based diffusions and proved their LDPs under the scaling regime that corresponds to $N\to\infty$ and $TN^2=A<\infty$, which we refer to as the \textbf{diffusive regime}.
In the diffusive regime, the drift and diffusive effect compete at the same footing.
We hence do not expect deviations with atoms (that persist over time) to be relevant.
This distinction has a significant implication in the proof, which is explained in the next paragraph.

Let us describe the major challenge in proving Theorem~\ref{t.ldp}.
The challenge stems from the discontinuity of the drift, which is common in the study of rank-based diffusions.
With the presence of atomic deviations, however, the issue of discontinuity becomes \emph{much more severe}.
To see why, take any deviation $\mu\in\Csp([0,\termts],\mm\Psp)$. 
At the macroscopic level, when the attractive BPs follow $\mu$, the drift should approximate $\Sgn[\mu(s)](x)$.
For a non-atomic $\mu(s)$, the macroscopic drift $\Sgn[\mu(s)](x)$ is continuous in $x$, but for an atomic $\mu(s)$, the macroscopic drift is not continuous: for example $\Sgn[\delta_0](x)=(\ind_{(-\infty,0)}(x)-\ind_{(0,\infty)}(x))/2$.
Put differently, in the diffusive regime, the effect of the discontinuity in the drift diminishes in the $N\to\infty$ limit, while in the clustering regimes, the effect of the discontinuity \emph{persists}.
This issue renders many commonly-used tools non-applicable and manifests itself in the proof of the LDP lower bound, where we (have to) use a cluster approximation with the help of the quantile representation \eqref{e.rateq} of the rate function. 

\subsection{Result: applications to the moment Lyapunov exponents}
\label{s.results.momentL}
The results in Sections~\ref{s.results.momentL}--\ref{s.results.kpzburgers} are for a particular initial condition, which we now define.
For $\xxi\in\R$ and a small $\alpha>0$, let $\ZZ_\alpha=\ZZ_\alpha(t,x)$ be the solution of the SHE with the \textbf{delta-like initial condition} 
\begin{align}
	\label{e.deltalike.ic}
	\ZZ_\alpha(0,NT\Cdot\,) := \ind_{[-\alpha+\xxi,\xxi+\alpha]}.
\end{align}
We will send $N\to\infty$ first and $\alpha\to 0$ later.
The solution of the SHE with this initial condition, scaled in time by $T$ and in space by $NT$, can be written by the Feynman--Kac formula as 
\begin{align}
	\label{e.Z}
	\ZZ_{N,\alpha}(t,x)
	:=
	\ZZ_{\alpha}(Tt,NTx)
	:=
	\Ebm\big[ e^{\int_{0}^{Tt}\d s\, \noise(Tt-s,\X(s))} \ind_{[-\alpha+\xxi,\xxi+\alpha]}(\tfrac{1}{NT}\X(Tt)) \big],
\end{align} 
where $\X=\text{(standard BM)}+NTx$.
For the purpose of studying the LDPs and limit shapes of the KPZ equation, the delta-like initial condition approximates the true delta initial condition $\delta_{\xxi}$ sufficiently well; this is seen in the analysis of the companion paper \cite{lin23}.

We now explain how to obtain the multipoint moment Lyapunov exponents of $\ZZ_{N,\alpha}$.
To set up the notation, fix any $\xx_1<\ldots<\xx_n\in\R$, $\xxi\in\R$, let $\vecmm=(\mm_1,\ldots,\mm_n)\in[0,\infty)^n$ and $\mm=\mm_1+\ldots+\mm_n$ be as before, and let
\begin{subequations}
\label{e.minii}
\begin{align}
\label{e.minii1}
	&\momshe\big( \xxi \xrightarrow{\termt} (\vecxx,\vecmm) \big)
	:=
	\sup\Big\{ 
	\tfrac{\termt\mm^3}{24} + \Ip{ \mu(s)^{\otimes 2}, \tfrac{1}{4}|x-x'| }\big|^{s=\termt}_{s=0} - \rateq(\mu)
		\ :
\\
		&\qquad\mu\in\Csp([0,\termts],\mm\Psp(\R)),
		\
		\mu(0)=\sum_{\cc=1}^n \mm_\cc \delta_{\xx_\cc}, 
		\
		\mu(\termt)=\mm\delta_{\xxi} 
	\Big\}.
\end{align}
\end{subequations}
Let $ \cc_{N}(\ii):=\min\{ \cc : \mm_1+\ldots+\mm_\cc \geq i/N \} $ and set $y_{\ii}=\xx_{\cc_{N}(\ii)}$.
This way, $\EE[\prod_{\ii=1}^{N\mm}\ZZ(T\termt,NTy_{\ii})]=\EE[\prod_{\cc=1}^{n}\ZZ(T\termt,NT\xx_{\cc})^{N\mm_{\cc}}]$.
Specialize to the delta-like initial condition \eqref{e.deltalike.ic}, use \eqref{e.feynmankac} for $t=T\termt$ and $y_i\mapsto NT y_i$, use \eqref{e.transformation.aBP.}, and apply $\frac{1}{N^3T}\log(\Cdot)$ to the result.
Doing so gives
\begin{align}
\label{e.1}
\begin{split}
	&\frac{1}{N^3T}\log \EE\Big[\prod_{\cc=1}^{n}\ZZ_{N,\alpha}(\termt,\xx_{\cc})^{\mm_{\cc}}\Big]
	=
	\frac{\termt\mm^3}{24}
	+
	\frac{1}{4} \sum_{\cc,\cc'=1}^{n} \mm_{\cc} \, \mm_{\cc'}\,|\xx_{\cc}-\xx_{\cc'}|
	+
	o_N(1)
\\
	&+
	\frac{1}{N^3T} \log\P\Big[ 
		\X_{\ii}(0)=NT \xx_{\cc_N(\ii)},
		|\X_{\ii}(T\termt)-NT\xxi| \leq NT\alpha, \ \forall \ii
	\Big],
\end{split}
\end{align}
where $o_N(1)$ denotes a term that converges to zero as $N\to\infty$.
Rewrite \eqref{e.1} as
\begin{subequations}
\label{e.2}
\begin{align}
	&\frac{1}{N^3T}\log\EE\Big[\prod_{\cc=1}^{n}\ZZ_{N,\alpha}(\termt,\xx_{\cc})^{\mm_{\cc}}\Big]
	=
	\frac{\termt\mm^3}{24} + \Ip{ \EM_N(s)^{\otimes 2}, \tfrac{1}{4}|x-x'| }\big|^{s=\termt}_{s=0} + o(1)
\\
	\label{e.21}
	&+
	\frac{1}{N^3T} \log\P\Big[ \EM_{N}(0)=\sum_{\cc=1}^n\mm_{\cc}\delta_{\xx_\cc}\,,\ \supp(\EM_{N}(\termt))\subset[-\delta+\xxi,\xxi+\delta] \Big].
\end{align}
\end{subequations}
It now suffices to obtain the limit of \eqref{e.21} under $\lim_{\alpha\to0}\lim_{N\to\infty}$.
This will be carried out in Section~\ref{s.momentL.moml} with the aid of Theorem~\ref{t.ldp}.
The result gives
\begin{align}
\label{e.3}
	\limsup_{\alpha\to 0}\limsup_{N\to\infty}\Big|
		\eqref{e.21} + \inf\Big\{ \rateq(\mu) : \mu(0)=\sum_{\cc=1}^n \mm_\cc\delta_{\xx_\cc};\ \mu(\termt)=\delta_{\xxi} \Big\}
	\Big|
	=
	0.
\end{align}
Combining \eqref{e.2}--\eqref{e.3} gives the multipoint moment Lyapunov exponents of $\ZZ_{N,\alpha}$:
\begin{align}
	\label{e.momLyap}
	\limsup_{\alpha\to 0}\limsup_{N\to\infty}
	\Big| 
		\frac{1}{N^3T}\log &\EE\Big[\prod_{\cc=1}^{n}\ZZ_{N,\alpha}(\termt,\xx_{\cc})^{\mm_{\cc}}\Big]
		-
		\momshe\big( \xxi \xrightarrow{\termt} (\vecxx,\vecmm)) \big)
	\Big|
	=
	0.
\end{align}

For our subsequent analysis, the functional in the supremum in \eqref{e.minii1} is not quite convenient, and we need an equivalent expression of it.
Let
\begin{align}
	\label{e.mom}
	\mom_{[s_1,s_2]}(\mu)
	:=
	\int_{s_1}^{s_2} \d s\, \Big(\sum_{x} \frac{1}{24} \ip{ \mu, \ind_{\{x\}} }^3 - \int_0^{\mm} \d a \, \frac12 \big( \partial_s \quant[\mu] \big)^2\Big),
\end{align}
where the sum runs over atoms of $\mu$.
In Appendix~\ref{s.a.basic}, we prove that
\begin{align}
	\label{e.mom.rep}
	\tfrac{\termt\mm^3}{24} + \Ip{ \mu(s)^{\otimes 2}, \tfrac{1}{4}|x-x'| }\big|^{s=\termt}_{s=0} - \rateq(\mu)=\mom_{[0,\termts]}(\mu).
\end{align}
Given this, let us rewrite \eqref{e.minii} as
\begin{subequations}
\label{e.mini}
\begin{align}
	\momshe&\big( \xxi \xrightarrow{\termt} (\vecxx,\vecmm)) \big)
	:=
	\sup\Big\{ \mom_{[0,\termts]}(\mu) 
		\ :
\\
		&\mu\in\Csp([0,\termts],\mm\Psp(\R)),
		\
		\mu(0)=\sum_{\cc=1}^n \mm_\cc \delta_{\xx_\cc}, 
		\
		\mu(\termt)=\mm\delta_{\xxi} 
	\Big\},
\end{align}
\end{subequations}
where $\vecmm=(\mm_1,\ldots,\mm_n)$ and $\mm:=\mm_1+\ldots+\mm_n$.

Our result below, Theorem~\ref{t.optimal}, gives the unique minimizer of \eqref{e.mini}.
To state the result we need to construct a few objects.
We begin by introducing what we call \textbf{inertia clusters}.
The inertia clusters are point masses $\mm_1,\ldots,\mm_n$ that start at $s=0$ from $\xx_1<\ldots<\xx_n$ with velocities $\pull_\cc := \frac12(-\mm_{1}\ldots-\mm_{\cc-1}+\mm_{\cc+1}+\ldots+\mm_{n})$, $\cc=1,\ldots,n$.
They travel at constant velocities until they meet, and when they meet, they merge according to the conservation of momentum.
For example, if $\inerc{\cc}$ and $\inerc{\cc+1}$ meet, they merged into a single cluster with mass $(\mm_\cc+\mm_{\cc+1})$ and velocity $ (\mm_\cc \pull_\cc + \mm_{\cc+1} \pull_{\cc+1})/(\mm_\cc+\mm_{\cc+1})$.
Let $\inerc{\cc}=\inerc{\cc}(s)\in\Csp[0,\termts]$, $\cc=1,\ldots,n$ denote the trajectories of the inertial clusters.
Examine which inertia clusters have merged within $(0,\termt)$ and lump the indices of those clusters together.
Doing so gives a partition of $\{1,\ldots,n\}$ into intervals.
We let $\Branch$ denote this partition, and call an element $\bb$ in $\Branch$ a \textbf{branch}.
Namely, $\cc,\cc'\in\bb$ if and only if $\inerc{\cc}$ and $\inerc{\cc'}$ merged with $s\in(0,\termt)$.
Note that branches depend on $\termt$.
All the inertia clusters within a branch end up at the same position at $s=\termt$, namely $\inerc{\cc}(\termt)=\inerc{\cc'}(\termt)$ for all $\cc,\cc'\in\bb$; call this position $\inerc{\bb}(\termt)$.
In general, $\inerc{\bb}(\termt)\neq\xxi$.
To \emph{bring} the clusters to $\xxi$ at the ending time, apply a constant drift:
\begin{align}
	\label{e.optimal.frominertia}
	\optimalc{\cc}(s) := \inerc{\cc}(s) + \diner_\bb s,
	\qquad
	\diner_\bb := (\xxi-\inerc{\bb}(\termt))/\termt,
	\qquad
	\cc \in \bb, \ \bb\in\Branch.
\end{align}
The resulting deviation $\optimal=\sum_{\cc=1}^n \mm_\cc \delta_{\optimalc{\cc}}$ is the \textbf{optimal deviation} and we call $\optimalc{1},\ldots,\optimalc{n}$ the \textbf{optimal clusters}.
See Figure~\ref{f.optimal}.
By construction, $\optimalc{\cc}$ and $\optimalc{\cc'}$ merge within $(0,\termt)$ if and only if $\cc,\cc'\in\bb$.
\begin{thm}\label{t.optimal}
The optimal deviation $\optimal$ is the unique minimizer of~\eqref{e.mini}.
\end{thm}

\begin{figure}
\begin{minipage}{.45\linewidth}
\fbox{\includegraphics[width=\linewidth]{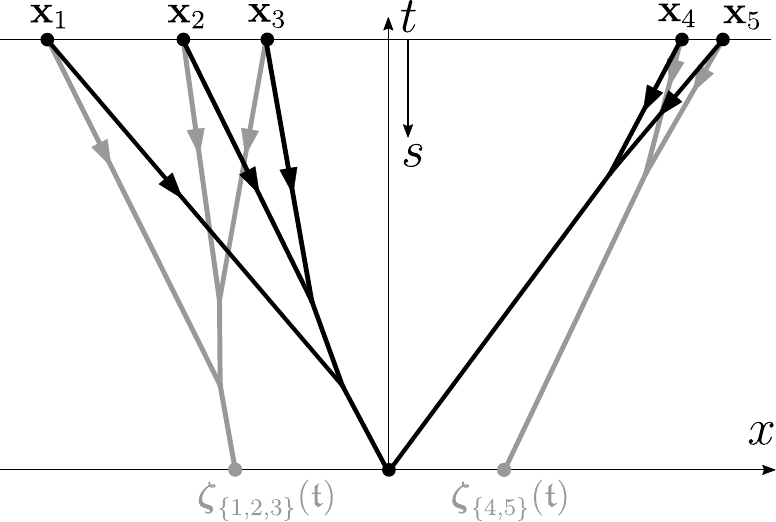}}
\caption{The inertia clusters (gray) and optimal clusters (black).
In this figure, $\Branch=\{\{1,2,3\},\{4,5\}\}$, and $\xxi=0$.
}
\label{f.optimal}
\end{minipage}
\hfill
\begin{minipage}{.45\linewidth}
\fbox{\includegraphics[width=\linewidth]{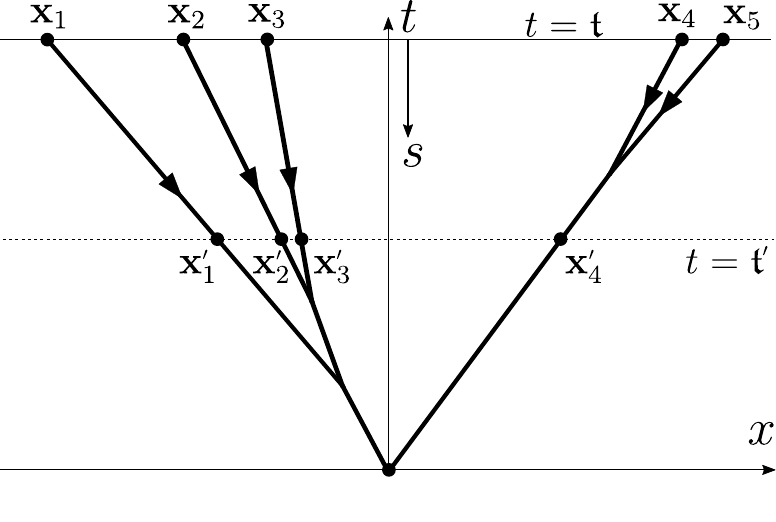}}
\caption{Intermediate-time configuration.
In this figure, $\Aranch(1)=\{1\}$, $\Aranch(2)=\{2\}$, $\Aranch(3)=\{3\}$, $\Aranch(4)=\{4,5\}$.%
}
\label{f.intermediate}
\end{minipage}
\end{figure}

For the analysis in the companion paper \cite{lin23}, the result \eqref{e.momLyap} along does not suffice, and we need a ``localized'' version of it as well.
To set up the notation,
For $f\in\Csp[0,Ts']$, set 
\begin{align}
	\label{e.Dist}
	\Dist_{N,s}(f,\optimal) &:= \min_{\cc=1,\ldots,n}\{ |\tfrac{1}{NT}f(Ts)-\optimal_\cc(s)| \},
\\
	\label{e.Dist[]}
	\Dist_{N,[0,s']}(f,\optimal) &:= \sup_{s\in[0,s']} \Dist_{N,s}(f,\optimal).
\end{align}
Define the localized version of $\ZZ_{N,\alpha}$:
\begin{align}
\begin{split}
\label{e.Zloc}
	\ZZ^{\betaloc}_{N,\alpha}(t,x)
	:=
	\Ebm\big[ 
		e^{\int_{0}^{Tt}\d s\, \noise(Tt-s,\X(s))} 
		&\ind_{[-\alpha+\xxi,\xxi+\alpha]}(\tfrac{1}{NT}\X(Tt))
\\
		&\cdot 
		\ind_{\{ \Dist_{N,[0,t]}(\X,\optimal(\cdot+(\termt-t))) \leq \beta\}} 
	\big],
\end{split}
\end{align} 
where $\X=\text{(standard BM)}+NTx$.
The last indicator in \eqref{e.Zloc} constrains the BM to stay close to $\optimal$, in the post-scale units.
By construction, $\ZZ^{{\betaloc}}_{N,\alpha}\leq \ZZ_{N,\alpha}$.
Below, Corollary~\ref{c.momentL}\ref{c.momentL1} just restates \eqref{e.momLyap}, while Corollary~\ref{c.momentL}\ref{c.momentL2} is the localized version of \eqref{e.momLyap} that will be needed in \cite{lin23}.
\begin{cor}\label{c.momentL}
\begin{enumerate}[leftmargin=20pt,label=(\alph*)]
\item[]
\item \label{c.momentL1}
$
	\displaystyle
	\limsup_{\alpha\to 0}\limsup_{N\to\infty}
	\Big|
		\frac{1}{N^3T} \log \EE\Big[ \prod_{\cc=1}^n \ZZ_{N,\alpha}(\termt,\xx_\cc)^{N\mm_\cc} \Big]
		-
		\momshe\big(\xxi\xrightarrow{\termt}(\vecxx,\vecmm)\big)
	\Big|
	=0.
$
\item \label{c.momentL2}
For any nonempty $ A\subset\{1,\ldots,n\}$ and $\beta>0$,
\begin{align}
\begin{split}
	\limsup_{\alpha\to 0}\limsup_{N\to\infty}
	\frac{1}{N^3T} 
	\log \EE\Big[ 
		\prod_{\cc\in A} \big(\ZZ_{N,\alpha}-\ZZ^{\betaloc}_{N,\alpha}\big)(\termt,\xx_\cc)^{N\mm_\cc} 
		&\cdot \prod_{\cc\notin A} \ZZ_{N,\alpha}(\termt,\xx_\cc)^{N\mm_\cc} 
	\Big]
\\
	&<
	\momshe\big(\xxi\xrightarrow{\termt}(\vecxx,\vecmm)\big).
\end{split}
\end{align}
\end{enumerate}
\end{cor}
\noindent{}%
As mentioned in Section~\ref{s.intro}, the analog of Corollary~\ref{c.momentL}\ref{c.momentL1} in the hyperbolic scaling regime ($T\to\infty$, $N=1$, $\vec{\mm}\in\Z_{>0}^n$) has been proven in \cite{lin2023multi}.

The discussion that leads up to \eqref{e.momLyap} holds for more general initial conditions.
For example, one can consider the multi-delta-like initial condition
\begin{align}
	\label{e.multdeltalike.ic}
	\ZZ'_\alpha(0,NT\Cdot\,) := e^{N^2T\beta_1}\ind_{[-\alpha+\xxi_1,\xxi_1+\alpha]} + \ldots + e^{N^2T\beta_m}\ind_{[-\alpha+\xxi_m,\xxi_m+\alpha]},
\end{align}
for $\xxi_1<\ldots<\xxi_m\in\R$, $\beta_1,\ldots,\beta_m\in\R$, with $N\to\infty$ first and $\alpha\to 0$ later.
On the other hand, the way we characterize the unique minimizer in Theorem~\ref{t.optimal} works only for the delta-like initial condition~\eqref{e.deltalike.ic}.

\subsection{Result: Connection to $\ratekpz$ and the limit shape}
\label{s.results.kpzburgers}
We begin by preparing some notation.
Fix $t\in(0,\termts]$ and $\xx_1<\ldots<\xx_n$, and let $\parab(t)=\parab(t,x):=-x^2/(2t)$ and
\begin{align}
	\label{e.hvspace}
	\hvspace(t,\vecxx) := \{ \vechv=(\hv_\cc)_{\cc=1}^n : \hv_\cc \geq \parab(t,\xx_\cc), \ \cc=1,\ldots,n \}.
\end{align}
Given any $\vechv\in\hvspace(t,\vecxx)$, let $\hf{,t,\vecxx,\vechv}=\hf{}(x)$ be the piecewise $\Csp^1$ function on $\R$ characterized by the properties: $\hf{}(\xx_\cc)=\hv_\cc$, for all $\cc$; $\hf{}\geq \parab(t)$; $\hf{}(x)=\parab(t,x)$ for all $|x|$ large enough; $\partial_x\hf{}$ is constant on $\{\hf{}>\parab(t)\}\setminus\{\xx_1,\ldots,\xx_n\}$; $\hf{}$ is $\Csp^1$ except at $\xx_1,\ldots,\xx_n$.
See Figure~\ref{f.shape.termt} for an illustration.
Set
\begin{align}
	\label{e.ratekpz}
	\ratekpz(t,\vecxx,\vechv)
	&:=
	\int_{\R} \d x \, \big( \tfrac12( \partial_x \hf{,t,\vecxx,\vechv})^2 - \tfrac12(\partial_x\parab(t))^2 \big),
\\
	\label{e.mps}
	\mps(t,x)
	&:=
	\inf\Big\{ \frac{(x-y)^2}{2(\termt-t)} + \hf{}(y) : y\in\R \Big\},
	\
	(t,x)\in(0,1]\times\R.
\end{align}
See Figure~\ref{f.shape} for an illustration of $\mps$.
The space $\hvspace(t,\vecxx)$ is the space of deviations of the $n$-point upper-tail LDP for the KPZ equation at $(t,\xx_1)$,\ldots,$(t,\xx_n)$.
Because our method relies on positive moments, the results in this paper and in \cite{lin23} are restricted to the subspace
\begin{align}
	\label{e.hvspacec}
	\hvspacec(t,\vecxx) := \{ \vechv\in\hvspace(t,\vecxx) : \hf{,t,\vecxx,\vechv} \text{ is concave} \}.
\end{align}
Let us mention a related property.
Recall that the hypograph of a function is $\mathrm{hypo}(\f):=\{(\xx,\hv):\hv\leq \f(\xx), \xx\in\R\}$.
When $\vechv\in\hvspacec(t,\vecxx)$, the function $\hf{,t,\vecxx,\vechv}=\hf{}$ has its hypograph $\mathrm{hypo}(\hf{})$ given by the convex hull of $\mathrm{hypo}(\parab(1))\cup\{(\xx_{\cc},\hv_{\cc}) : \cc = 1,\ldots,n\}$, but this property fails when $\vechv\in\hvspace(t,\vecxx)\setminus \hvspacec(t,\vecxx)$.

The companion paper \cite{lin23} shows that the KPZ equation satisfies the finite-dimensional LDP with the rate function $\ratekpz$, and that $\mps$ gives the corresponding spacetime limit shape, under the same scaling regimes considered in this paper; see Theorem~1.1 in \cite{lin21} for the precise statement.
A similar LDP is proven in \cite{ganguly22} based the Brownian Gibbs resampling property \cite{corwin2014brownian,corwin2016kpz}.
The results in \cite{ganguly22} cover a different set of scaling regimes, work for configurations in $\hvspace(t,\vecxx)$, and give very detailed probability bounds.

Theorem~\ref{t.matching} gives the connection between $\ratekpz$, $\momshe$, $\mps$, which is used in \cite{lin23}.
\begin{notation}
\begin{enumerate}[leftmargin=20pt, label=(\alph*)]
\item[]
\item
Fix $t$ and $\vecxx$; view $\ratekpz(t,\vecxx,\vechv)=:\ratekpz(\vechv)$ and $\momshe(0\xrightarrow{t} (\vecxx,\vecmm))=:\momshe(\vecmm)$ as functions on $\hvspacec(t,\vecxx)$ and on $[0,\infty)^n$, respectively.%
\item \label{t.matching.shape.}
Fix any $(\xx_1<\ldots<\xx_n)$, consider the terminal time $\termt$, take any pair $(\vechv,\vecmm)\in\hvspacec(\termt,\vecxx)^\circ\times(0,\infty)^n$ that satisfies $(\nabla_{\vechv}\ratekpz)(t,\vecxx,\vechv) = \vecmm$, and let $\mps$ and $\optimal$ be the corresponding limit shape and optimal deviation, respectively.
As will be explained in Section~\ref{s.matching.shape}, the limit shape $\mps$ is a weak solution of the Hamilton--Jacobi equation of Burgers' equation and has shocks.

\item 
\label{t.matching.intermediate.}
Let $\vecxx$, $\vechv$, and $\vecmm$ be as in Part~\ref{t.matching.shape.} and take any intermediate time $\intermt\in(0,\termts]$.
Traveling in backward time, some of the optimal clusters may have merged by time $s=\termt-\intermt$.
Let $\{\optimal_\cc(\termt-\intermt)\}_{\cc=1}^n = \{ \xxi_{1}<\ldots<\xxi_{n'} \}$ denote the distinct positions of the clusters at that time.
Accordingly, let $\vechvi := (\mps(\intermt,\xxi_{\aa}))_{\aa=1}^{n'}$, $\Aranch(\aa):=\{\cc:\optimal_\cc(\termt-\intermt)=\xxi_\aa\}$, $\mmi_\aa:=\sum_{\cc\in\Aranch(\aa)} \mm_{\cc}$, and $\vecmmi:=(\mmi_\aa)_{\aa=1}^{n'}$; see Figure~\ref{f.intermediate} for an illustration.
\end{enumerate}
\end{notation}

\begin{thm}\label{t.matching}
Notation as above.
\begin{enumerate}[leftmargin=20pt, label=(\alph*)]
\item 
\label{t.matching.legendre}
The functions $\ratekpz$ and $\momshe$ are continuous, strictly convex, and the Legendre transform of each other.
Further, $\nabla\ratekpz=\nabla_{\vechv}\ratekpz:\hvspacec(t,\vecxx)\to[0,\infty)^n$ is a homeomorphism.
\item 
\label{t.matching.shape}
The trajectories of the shocks in $\mps$ and the trajectories of the optimal clusters in $\optimal$ coincide.
\item 
\label{t.matching.intermediate}
We have
\begin{align}
	\label{e.legendre.dual.intermediate}
	(\nabla_{\vechv}\ratekpz)(\intermt,\vecxxi,\vechvi) & = \vecmmi,
\\
	\label{e.treesg}
	\momshe\big( 0 \xrightarrow{\termt} (\vecxx,\vecmm) \big)
	&=
	\momshe\big( 0 \xrightarrow{\intermt} (\vecxxi,\vecmmi) \big)
	+
	\sum_{\aa=1}^{n'} \momshe\big( \xxi_\aa \xrightarrow{\termt-\intermt} (\xx_\cc,\mm_\cc)_{\cc\in\Aranch(\aa)} \big).
\end{align}
\end{enumerate}
\end{thm}

\begin{figure}
\begin{minipage}{.5\linewidth}
\fbox{\includegraphics[width=\linewidth]{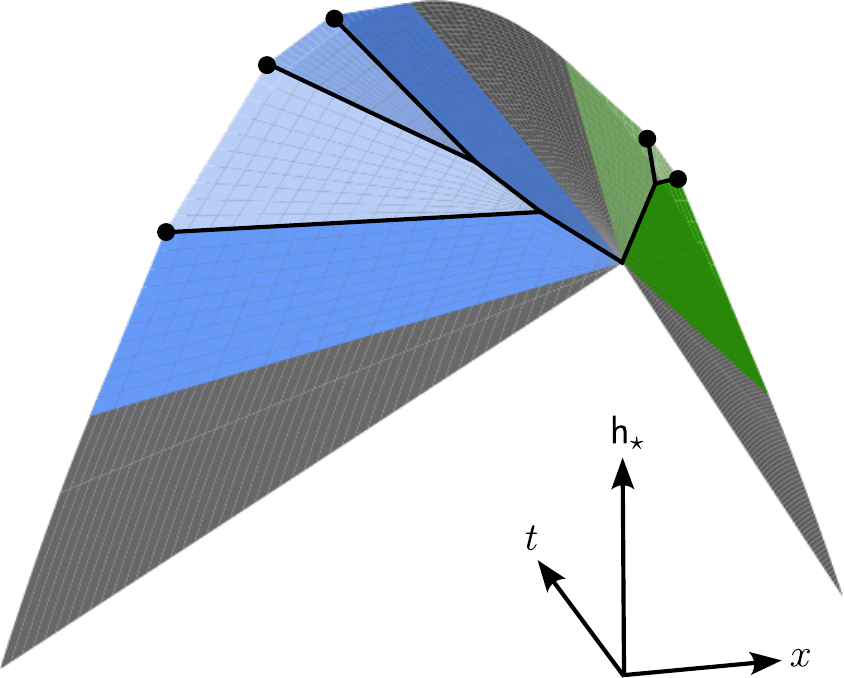}}
\caption{The limit shape $\mps$.
In the gray regions $\mps=\parab$; in the colored regions $\mps$ is piecewise linear.
}
\label{f.shape}
\end{minipage}
\hfill
\begin{minipage}{.48\linewidth}
%	\begin{minipage}{\linewidth}
	\fbox{\includegraphics[width=\linewidth]{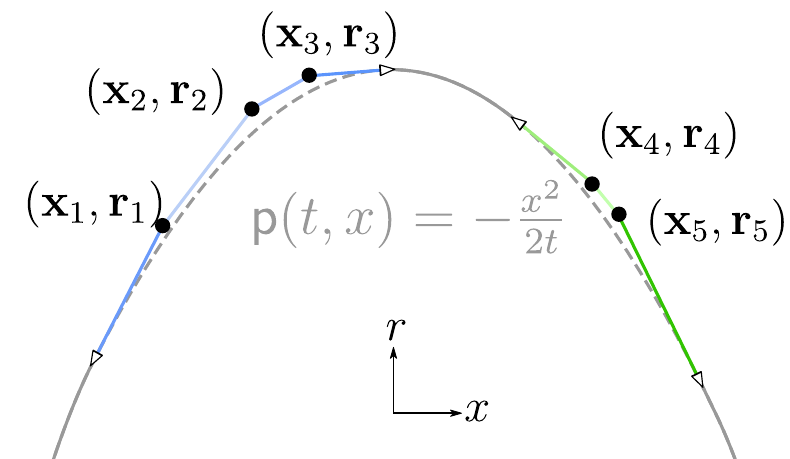}}
	\caption{The function $\hf{}=\hf{,t,\vecxx,\vechv}$.}
	\label{f.shape.termt}
%	\end{minipage}
	%
%	\begin{minipage}{\linewidth}
%	\fbox{\includegraphics[width=\linewidth]{shocks}}
%	\caption{The shocks (thick solid lines) and characteristics (thin solid lines)}
%	\label{f.shocks}
%	\end{minipage}
\end{minipage}
\end{figure}

\section{Notation, definitions, tools}
\label{s.basic}

\subsection{Reduction to $ \mm=1 $}
\label{s.basic.m=1}
We begin by explaining how Theorem~\ref{t.ldp}, Theorem~\ref{t.optimal}, and Corollary~\ref{c.momentL} follow from the special case of $ \mm=1 $.
The key lies in certain scaling relations.
Consider the scaling operator $ \ip{ \scale_\mm\lambda, \f } := \ip{ \lambda, \f(\Cdot/\mm) } $.
The map $ \frac{1}{\mm}\scale_\mm: \mm\Psp(\R) \to \Psp(\R)  $ is a homeomorphism.
Applying $ \frac{1}{\mm}\scale_\mm $ to $\EM_N(s)$ gives $ \frac{1}{\mm}\scale_\mm \EM_N(s) = \frac{1}{N\mm} \sum_{\ii=1}^{N\mm} \delta_{\X_{\ii}(Ts)/(N\mm T)} $.
This can be viewed as a unit-mass empirical measure with $ N':=N\mm $ being the scaling parameter.
We will verify in Appendix~\ref{s.a.basic} the scaling identities
\begin{align}
	\label{e.rate.scaling}
	&
	\rate\big(\tfrac{1}{\mm}\scale_\mm\mu) = \tfrac{1}{\mm^3}\rate(\mu),
	&&
	\rateq\big(\tfrac{1}{\mm}\scale_\mm\mu) = \tfrac{1}{\mm^3}\rateq(\mu),
	&
	\text{for all } \mu\in\Csp([0,\termts],\mm\Psp(\R)).
\end{align}
The $ 1/\mm^3 $ factor absorbs the change in the LDP speed by going from $ N $ to $ N'=N\mm $.
Combining what said above shows that Theorem~\ref{t.ldp}, Theorem~\ref{t.optimal}, and Corollary~\ref{c.momentL} follow from the special case of $ \mm=1 $.

\subsection{The spaces $\Psp(\R)$ and $\Csp([0,\termts],\Psp(\R))$}
\label{s.basic.topo}
By Section~\ref{s.basic.m=1}, we consider $\mm=1$ only.

Let us introduce some metrics on $ \Psp(\R) $ and $ \Csp([0,\termts],\Psp(\R)) $.
Even though we endow $\Psp(\R)$ with the weak* topology, it will be convenient to also consider the $1$-Wasserstein metric:
\begin{align}
	\label{e.wass}
	\wass(\lambda,\lambda')
	:=
	\inf\big\{ \ip{ \pi, |x-x'| } : \pi \in \Psp(\R^2), \, \ip{\pi, \Cdot\otimes 1} = \lambda,  \, \ip{\pi, 1\otimes\Cdot} = \lambda' \big\}.
\end{align}
The 1-Wasserstein metric on $\Psp(\R)$ permits the inverse-CDF formula:
\begin{align}
	\wass(\lambda,\lambda')
	\label{e.wass.inverseCDF}
	&=
	\int_{\R} \d x \, \big| \cdf[\lambda](x) - \cdf[\lambda'](x) \big|
	=
	\int_0^1 \d a \, \big| \quant[\lambda](a) - \quant[\lambda'](a) \big|.
\end{align}
With $\R$ being non-compact, the 1-Wasserstein metric produces a topology \emph{stronger} than the weak* topology on $\Psp(\R)$.
To metrize the weak* topology, we introduce
\begin{align}
	\label{e.dist}
	\dist(\lambda,\lambda')
	:=
	\sum_{k=1}^\infty 2^{-k} \min\Big\{ 1, \, \int_{-k}^{k} \d x \,\big| \cdf[\lambda](x) - \cdf[\lambda'](x) \big| \Big\}. 
\end{align}
It is not hard to check that $\dist$ metrizes the weak* topology on $\Psp(\R)$.
Accordingly, 
\begin{align}
	\label{e.dist.time}
	\dist_{[0,\termts]}(\mu,\mu')
	:=
	\sup\big\{  \dist(\mu(s),\mu'(s)) : s\in[0,\termts] \big\}
\end{align}
metrizes the topology on $ \Csp([0,\termts],\Psp(\R)) $ introduced in Section~\ref{s.results.ldp}.
Let us note two useful inequalities related to these metrics.
First, by \eqref{e.wass.inverseCDF}--\eqref{e.dist},
$
	\dist(\lambda,\lambda') \leq \wass(\lambda,\lambda').
$
Next, for any $v_1,v_2,\ldots\in(0,1]$ that add to $1$ and any $y_1,y_2,\ldots,y'_1,y'_2,\ldots\in\R$,
\begin{align}
	\label{e.wass.coupling}
	\dist\Big(\sum_i v_i\delta_{y_i} , \sum_i v_i\delta_{y'_i} \Big)
	\leq
	\wass\Big(\sum_i v_i\delta_{y_i} , \sum_i v_i\delta_{y'_i} \Big)
	\leq
	\sum_{i} v_i\big| y_\ii - y'_{\ii} \big|,
\end{align}
which holds thanks to the coupling $ \pi(\{y_i,y'_i\}) = v_i $, $ i=1,2,\ldots $.

We will need a criterion for a set $ S \subset \Csp([0,\termts],\Psp(\R)) $ to be precompact.
Recall the topology of $\Csp([0,\termts],\Psp(\R))$ from before Theorem~\ref{t.ldp} and recall that we endow $\Psp(\R)$ with the weak* topology.
First, by a generalized version of the Arzel\'{a}--Ascoli theorem, the set $ S $ is precompact if it is equi-continuous and if $ \{ \mu(s) : \mu\in S, s\in[0,\termts] \} $ is precompact in $ \Psp(\R) $; see \cite[Thm.\ 47.1]{munkres2000topology} for example.
By the Banach–Alaoglu theorem, for any $ b<\infty $, the set $ \{ \lambda \in \Psp(\R) : \supp(\lambda) \subset [-b,b] \} $ is compact.
From this property, it is not hard to show that, for any $ b_1,b_2,\ldots\to\infty $, the set $ \cap_{k=1}^\infty \{ \lambda \in \Psp(\R) : \lambda(\R\setminus[-b_k,b_k]) \leq 1/k \} $ is precompact in $ \Psp(\R) $.
These properties give the following criterion.
\begin{lem}
\label{l.precompact}
A set $ S \subset \Csp([0,\termts],\Psp(\R)) $ is precompact if
\begin{enumerate}[leftmargin=20pt, label=(\roman*)]
\item \label{l.precompact.conti}
the set $ S $ is equicontinuous with respect to $ \dist $, and
\item \label{l.precompact.tailbd}
there exists $ b_k\to\infty $ such that $ \ip{ \mu(s), \ind_{\R\setminus[-b_k,b_k]}} \leq 1/k $, for all $ k\in\Z_{>0} $, $ \mu\in S $, and $ s\in[0,\termts] $.
\end{enumerate}
\end{lem}

Here is a list of useful properties.
\begin{lem}\label{l.useful}
\begin{enumerate}[leftmargin=20pt, label=(\alph*)]
\item[]
\item \label{l.useful.inverseCDF}
For any $ \lambda \in \Psp(\R) $ and $ \f \in \Lsp^1(\R,\lambda) $, 
$
	\ip{ \lambda, \f }
	=
	\int_0^1 \d a \, \f(\quant[\lambda](a)).
$
\item \label{l.useful.tailbd}
For any $\mu\in\Csp([0,\termts],\Psp(\R))$, the tail mass $\sup_{s\in[0,\termts]} \ip{ \mu(s),\ind_{\R\setminus[-b,b]} } $ tends to zero as $b\to\infty$.
\item \label{l.useful.quant.conti}
For any $\mu\in\Csp([0,\termts],\Psp(\R))$, for Lebesgue almost every $a\in[0,1]$, $\quant[\mu(\Cdot)](a)\in\Csp[0,\termts]$.
\end{enumerate}
\end{lem}
\noindent%
Part~\ref{l.useful.inverseCDF} is the standard inverse-CDF formula; Parts~\ref{l.useful.tailbd}--\ref{l.useful.quant.conti} are proven in Appendix~\ref{s.a.basic}.

\subsection{Expressing $\Sgn[\lambda]$ in the quantile coordinate}
\label{s.basic.Sgn}
For any $ \lambda\in\Psp(\R) $ and $ a\in[0,1] $, let
\begin{subequations}
\label{e.a+-}
\begin{align}
	a_- = a_-[\lambda](a) &:= \inf\{ a' \in [0,1] : \quant[\lambda](a')=\quant[\lambda](a) \},
\\
	a_+ = a_+[\lambda](a) &:=  \sup\{ a'\in [0,1] : \quant[\lambda](a')=\quant[\lambda](a) \}.
\end{align}
\end{subequations}
Note that $ a_-<a_+ $ if and only if $ \lambda $ has an atom at $ \quant[\lambda](a) $. 
Refer to the second expression of $ \Sgn[\lambda] $ in \eqref{e.Sgn} and note that $ \lambda((-\infty,\quant[\lambda](a))) = a_- $ and $ \lambda((\quant[\lambda](a),\infty)) = (1-a_+) $.
We have
\begin{align}
	\label{e.Sgn.id}
	\Sgn[\lambda]\big(\quant[\lambda](a)\big)
	=
	\begin{cases}
		\frac12 - a, & \text{ when } a_-=a_+,
	\\
		\frac12- \frac{a_-+a_+}{2} = \frac{1}{a_+-a_-} \int_{a_-}^{a_+} \d a \, (\frac12 - a), & \text{ when } a_-<a_+.
	\end{cases}
\end{align}
Fix any $\f\in\Lsp^1(\R,\lambda)$.
By \eqref{e.Sgn.id}, we have 
$
	\f(\quant[\lambda](a))\cdot \Sgn[\lambda](\quant[\lambda](a))
	=
	\f(\quant[\lambda](a))\cdot (\tfrac12-a)
$
whenever $a_-=a_+$, and
$
	\int_{a_-}^{a^+} \d a \, \f(\quant[\lambda](a))\cdot \Sgn[\lambda](\quant[\lambda](a))
	=
	\int_{a_-}^{a^+} \d a \, \f(\quant[\lambda](a))\cdot (\tfrac12-a)
$
whenever $a_-<a_+$.
Combining these properties gives
\begin{align}
	\label{e.Sgn.id.int}
%	\ip{\lambda,\Sgn[\lambda]\f}
%	=
	\int_0^1 \d a \, \f\big(\quant[\lambda](a)\big) \cdot \Sgn[\lambda]\big(\quant[\lambda](a)\big)
	=
	\int_0^1 \d a \, \f\big(\quant[\lambda](a)\big) \, (\tfrac12-a).
\end{align}

\subsection{Dividing a measure}
\label{s.basic.divide}
Let us introduce a procedure of dividing a given $ \mu\in\Csp([0,\termts],\Psp(\R)) $ into pieces with constant masses.
Fix any $ \mm_1,\ldots,\mm_n \in (0,1] $ that add up to $ \mm=1 $.
For any $ s\in[0,\termts] $ and $ \cc \in \{1,\ldots,n\} $, consider $ \cdf_\cc(s,x) := \max\{\min\{ \cdf[\mu(s)](x)- (\mm_1+\ldots+\mm_{\cc-1}) , \mm_\cc \},0\}  $.
Namely, we consider the graph of $\cdf[\mu(s)](\Cdot)$ between the horizontal levels $(\mm_1+\ldots+\mm_{\cc-1})$ and $(\mm_1+\ldots+\mm_{\cc})$ and shift the graph down so that the lower level is at $0$.
The result $\cdf_\cc(s,x)$ is the CDF of a measure with total mass $\mm_{\cc}$, and we let $\mu_\cc(s)\in \mm_\cc\Psp(\R)$ denote that measure.

Here are a few properties of $\mu_\cc$.
First, $\mu_\cc(s)$ is continuous in $s$.
To see why, note that, by construction, $ |\cdf_\cc(s',x)-\cdf_\cc(s,x)| \leq |\cdf(s',x)-\cdf(s,x)|$.
Using this property in conjunction with \eqref{e.dist} and the continuity of $\mu$ gives the continuity of $\mu_\cc$.
Next, by construction, the measure $ \mu_\cc(s) $ is supported in $ [\quant[\mu(s)](\mm_1+\ldots+\mm_{\cc-1}),\quant[\mu(s)](\mm_1+\ldots+\mm_{\cc})] $.
Finally, note that $ \mu_{\cc}(s) $ and $ \mu_{\cc+1}(s) $ can both have an atom at $ \quant[\mu(s)](\mm_1+\ldots+\mm_{\cc}) $.

\section{LDP for the attractive BPs: properties of the rate functions}
\label{s.ratefn}

As explained in Section~\ref{s.basic.m=1}, we will only consider $ \mm=1 $, so all deviations take value in $ \Psp(\R) $.

\subsection{The quantile representation}
\label{s.rate.quantile}
Here, we show that $ \rate = \rateq $ on $ \Csp([0,\termts],\Psp(\R)) $.
To simplify notation, write $ \quant[\mu(s)](a) = \quant(s,a) $, $ \Sgn[\mu(s)](x) = \Sgn(s,x) $, and $ \cdf[\mu(s)](x) := \cdf(s,x) $.

\medskip
\noindent\textbf{Case~1: $ \boldsymbol{\partial_s\quant[\mu] \in \Lsp^2([0,\termts]\times[0,1])} $.}
We seek to express the terms in \eqref{e.rate} in the quantile coordinate.
Take any $ \h\in{\Cbsp}^{1,1}([0,\termts],\R) $.
Under the assumption of Case~1, $ \quant(s,a) $ is differentiable in $s$ Lebesgue almost everywhere (a.e.) on $ [0,\termts]\times[0,1] $, so $ \frac{\d~}{\d s} \h(s,\quant(s,a)) - \partial_s \h(s,\quant(s,a)) = \partial_x \h(s,\quant(s,a)) \partial_s\quant(s,a) $ Lebesgue a.e.
Apply $ \int\!\!{}_{[0,\termts]} \d s \int_0^1 \d a $ to both sides. 
On the left hand side of the result, use Lemma~\ref{l.useful}\ref{l.useful.inverseCDF} in reverse to turn the result into $\ip{\mu(s),\h(s)}|_{0}^{\termts}-\int\!\!{}_{[0,\termts]} \d s \ip{\mu,\partial_s\h}$, which gives the first two terms in \eqref{e.rateM}.
Next, use Lemma~\ref{l.useful}\ref{l.useful.inverseCDF} to express the last term in \eqref{e.rateM} and the integral term in \eqref{e.rate} in the quantile coordinate.
Collecting the preceding results gives
\begin{align}
	\label{e.p.rate=rateq.0}
	\rateM(\mu,\h) -\int_{0}^{\termt} \d s \, \frac12 \Ip{ \mu, (\partial_x \h)^2 }
	=
	\int_0^{\termt} \d s \int_0^1 \d a \, \Big( (\partial_x \h)_\quant( \partial_s\quant - \Sgn_\quant ) - \frac{1}{2}\big((\partial_x \h)_\quant\big)^2 \Big),
\end{align}
where the subscript means $(\varphi_\quant)(s,a) := \varphi(t,\quant(s,a))$, transformation to the quantile coordinate.
Taking the sumpremum over $\hh\in\Cbsp^{1,1}([0,\termts],\R)$ gives
\begin{align}
	\label{e.p.rate=rateq.1}
	\rate(\mu)
	=
	\sup_{\h} \int_0^{\termt} \d s \int_0^1 \d a \, \Big( (\partial_x \h)_\quant( \partial_s\quant - \Sgn_\quant ) - \frac{1}{2}\big((\partial_x \h)_\quant\big)^2 \Big).
\end{align}

Let us simplify \eqref{e.p.rate=rateq.1}.
Within the integral, write $ \partial_x \h = \g $ and complete the square to get $ -\frac12 (\g_\quant - \partial_s\quant + \Sgn_\quant)^2 + \frac12(\partial_s\quant - \Sgn_\quant)^2 $, and recognize the contribution of $\frac12(\partial_s\quant - \Sgn_\quant)^2$ as $ \rateq(\mu) $.
The supremum is taken over $ \h\in{\Cbsp}^{1,1}([0,\termts],\R) $. 
Via an approximation argument, we can replace this supremum with the supremum over those $\g$s with $ \int\!\!{}_{[0,\termts]} \d s \,\ip{\mu, \g^2} < \infty $.
Doing so gives
\begin{align}
	\label{e.p.rate=rateq.2}
	\rate(\mu)
	=
	\rateq(\mu)
	-
	\inf \Big\{ 
		\int_0^{\termt} \d s \int_0^1 \d a \, \frac12 \big( \g_\quant - \partial_s\quant + \Sgn_\quant \big)^2 
		\, : \, 
		\int_0^{\termt} \d s \, \ip{\mu, \g^2} < \infty 
	\Big\}.
\end{align}

It remains only to show that the infimum in \eqref{e.p.rate=rateq.2} is zero.
%When the CDF $ \cdf(s,x) $ strictly increases in $ x $ for all $s$, choosing $ \g = (\partial_s\quant)_{\cdf} - \Sgn $ completes the proof.
%Otherwise, 
Take any partition $ \tparti = \{0=s_0<s_1<\ldots < s_{|\tparti|}=\termt\} $ of $[0,\termts]$, for any $\varphi\in\Csp[0,\termt]$, let
\begin{align}
	\label{e.discretederivative}
	(\partial_\tparti\varphi)(s) := \sum_{j=1}^{|\tparti|}\frac{\varphi(s_{j})-\varphi(s_{j-1})}{s_j-s_{j-1}}\ind_{[s_{j-1},s_j)}(s)
\end{align} 
denote the discrete-time derivative with respect to $\tparti$, and take $\g(s,x) := (\partial_\tparti\quant(\Cdot,\cdf(\Cdot,x))(s) - \Sgn(s,x)$.
Insert this $\g$ into \eqref{e.p.rate=rateq.2}.
The terms $-\Sgn_\quant$ and $\Sgn_\quant$ cancel each other.
For the discrete-derivative term, note that $ \quant(s_j,\cdf(s_j,x))_{\quant} := \quant(s_j,\cdf(s_j,\quant(s_j,a))) $.
It is not hard to check that, for every $ s_j $, the last expression is equal to $ \quant(s_j,a) $ Lebesgue a.e.\ on $ [0,1] $.
The resulting integral hence reads $ \int\!\!{}_{[0,\termts]} \d s \int_0^1 \d a \, \frac12 (\partial_\tparti\quant - \partial_s\quant)^2 $.
Under the assumption of Case~1, this integral tends to zero as the mesh of $\tparti$ tends to zero.
This completes the proof for Case~1.

\medskip
\noindent\textbf{Case~2: $ \boldsymbol{\partial_s\quant[\mu] \notin \Lsp^2([0,\termts]\times[0,1])} $.}
In this case, by definition, $\rateq(\mu):=\infty$.
Take any partition $ \tparti = \{0=s_0<s_1<\ldots < s_{|\tparti|}=\termt\} $ of $[0,\termts]$ and let $\partial_\tparti$ be as in \eqref{e.discretederivative}.
We will show that
\begin{align}
	\label{e.p.rate=rateq.case2}
	\rate(\mu) +\frac{\termt}{8} \geq \frac{1}{4} \int_0^{\termt} \d s \int_0^1 \d a \, \big( \partial_\tparti \quant \big)^2.
\end{align}
By the assumption of Case~2 and Lemma~\ref{l.useful}\ref{l.useful.quant.conti}, the right hand side of \eqref{e.p.rate=rateq.case2} tends to $\infty$ as the mesh of $\tparti$ tends to zero.
Hence proving \eqref{e.p.rate=rateq.case2} will give the desired result $\rate(\mu)=\rateq(\mu)=\infty$.

The proof will invoke a variant of $\rate$.
Let us introduce this variant and its properties.
\begin{align}
	\label{e.rate0.j}
	\rate^\circ{}_{[s',s'']}(\mu)
	:=
	\sup_{\h \in {\Cbsp}^{1,1}([s',s''],\R)}
	\Big\{ 	
		\ip{ \mu(s), \h(s) }\big|_{s'}^{s''} - \int_{s'}^{s''}\d s \, \ip{ \mu, \partial_s\h } - \int_{s'}^{s''}\d s \, \ip{ \mu, (\partial_x\h)^2 } 	
	\Big\}.
\end{align}
The following properties are not difficult to verify, which we do in Appendix~\ref{s.a.basic}.
\begin{enumerate}[leftmargin=20pt, label=(\alph*),topsep=0pt]
\item \label{eu.rate0.timeadd}
Additivity in time: $\rate^\circ{}_{[s_1,s_2]}+\rate^\circ{}_{[s_2,s_3]}=\rate^\circ{}_{[s_1,s_3]}$.
\item \label{eu.rate0.timemono}
Monotonicity in time: For all $[s_2,s_3]\subset[s_1,s_4]$, $\rate^\circ{}_{[s_2,s_3]}\leq\rate^\circ{}_{[s_1,s_4]}$.
\item \label{eu.rate0.convex}
Convexity: $ \mu\mapsto\rate^\circ{}_{[s',s'']}(\mu) $ is convex.
\item \label{eu.rate0.spacetranslation}
Space translation invariance: $ \rate^\circ{}_{[s',s'']}(\shift_{0,y}\mu)=\rate^\circ{}_{[s',s'']}(\mu) $, where $ \shift_{0,y} $ translates $ \mu $ in space by $ y $, namely $ \ip{\shift_{0,y}\mu(s),\f} := \ip{\mu(s),\f(\Cdot-y)} $ for $ \f\in\Cbsp(\R) $.
\item \label{eu.rate0.stationary=0}
For any time-independent $\lambda\in\Psp(\R)$, $ \rate^\circ{}_{[s',s'']}(\lambda)=0$.
\end{enumerate}

We begin by bounding $\rate(\mu)$ from below.
Young's inequality and the property $ |\Sgn[\mu]| \leq 1/2 $ together give $ -\Sgn[\mu] \partial_x \h \geq - \frac18 - \frac{1}{2}(\partial_x \h)^2 $.
Insert this inequality into \eqref{e.rateM}--\eqref{e.rate}, and, within the result, move the term $ \int\!\!_{[0,\termts]} \d s\, \ip{\mu,-1/8} = -\termt/8 $ to the left hand side.
Doing so gives $\rate(\mu)+\termt/8\geq\rate^\circ{}_{[0,\termts]}(\mu)$.
Next, we mollify $\mu$. 
Define the time-space translation operator by $\ip{(\shift_{s,y}\nu)(s'),\f}:=\ip{\nu(s'+s),\f(\Cdot-y)}$, take $\varphi_\delta(x):=\exp(-x^2/(2\delta))/\sqrt{2\pi\delta}$, and use $\varphi_\delta$ to mollify $\mu$ to get 
\begin{align}
	\label{e.a1}
	\mu_\delta:=\int_\R \d s \int_\R\d y\, \varphi_\delta(s)\varphi_\delta(y)\shift_{s,y}\mu,
\end{align} 
with the convention $\mu(s)|_{s<0}:=\mu(0)$ and $\mu(s)|_{s>\termt}:=\mu(\termt)$.
Apply $\rate^\circ{}_{[0,\termts]}(\Cdot)$ to both sides of \eqref{e.a1}, use Property~\ref{eu.rate0.convex}, and use Property~\ref{eu.rate0.spacetranslation}.
Doing so gives
\begin{align}
	\rate^\circ{}_{[0,\termts]}(\mu_\delta)
	\leq
	\int_{\R}\d s\int_{\R}\d y\, \varphi_\delta(s)\varphi_\delta(y)\rate^\circ{}_{[0,\termts]}(\shift_{s,y}\mu)
	=
	\int_{\R}\d s\,\varphi_\delta(s)\,\rate^\circ{}_{[0,\termts]}(\shift_{s,0}\mu).
\end{align}
Next, for $ s\geq 0 $, using Properties~\ref{eu.rate0.timeadd}, \ref{eu.rate0.stationary=0}, and \ref{eu.rate0.timemono} in order gives
$
	\rate^\circ{}_{[0,\termts]}(\shift_{s,0}\mu) 
	=
	\rate^\circ{}_{[0,\termt-s]}(\mu)
	+
	\rate^\circ{}_{[\termt-s,\termts]}(\mu(\termt))
	=
	\rate^\circ{}_{[\termt-s,\termts]}(\mu)
	+ 
	0
	\leq
	\rate^\circ{}_{[0,\termts]}(\mu) + 0,
$
and similarly for $ s<0 $.
Hence 
\begin{align}
	\rate^\circ{}_{[0,\termts]}(\mu_\delta)
	\leq
	\Big(\int_{\R}\d s\, \varphi_\delta(s)\Big)\, \rate^\circ{}_{[0,\termts]}(\mu)
	=
	\rate^\circ{}_{[0,\termts]}(\mu)
	\leq
	\rate(\mu) + \frac{\termt}{8}.
\end{align}
Write $\rate^\circ_j:=\rate^\circ{}_{[s_{j-1},s_j]}$.
The last bound and Property~\ref{eu.rate0.timeadd} give 
\begin{align}
	\label{e.a2}
	\rate(\mu)+\frac{\termt}{8} \geq
	\rate^\circ_1(\mu_\delta)+\ldots+\rate^\circ_{|\tparti|}(\mu_\delta).
\end{align}

We next bound $\rate^\circ_j(\mu_\delta)$ from below.
Write $\cdf_\delta(s,x):=\cdf[\mu_\delta(s)](x)$ and $\quant_\delta(s,a):=\quant[\mu_\delta(s)](a)$ to simplify notation.
By the construction of $\mu_\delta$, both functions are $\Csp^\infty$.
Further, since the mollifier $\varphi_\delta$ is strictly positive everywhere, $\cdf_\delta(s,x)$ strictly increases in $x$, whereby $\cdf_\delta(s,\quant_\delta(s,a))=a$ for all $s$ and $a$.
Fix any $v\in\Csp^\infty[0,1]$ and let $\h(s,x):=\int_{-\infty}^x \d y \, v(\cdf_\delta(s,y))$.
For this test function, $\frac{\d~}{\d s} \h(s,\quant_\delta(s,x))=(\partial_s\h)_{\quant_\delta}+(\partial_x\h)_{\quant_\delta} \partial_s\quant_\delta$ and $(\partial_x\h)_{\quant_\delta} =(v\circ\cdf_\delta)_{\quant_\delta} $, which is equal to $v$ thanks to the relation $\cdf_\delta(s,\quant_\delta(s,a))=a$.
Hence $\frac{\d~}{\d s} \h(s,\quant_\delta(s,x))-(\partial_s\h)_{\quant_\delta}-(\partial_x\h)_{\quant_\delta}^2=v\,\partial_s\quant_\delta-v^2$.
Apply $\int\!\!{}_{[s_{j-1},s_j]}\d s\int\!\!{}_{[0,1]} \d a$ to both sides and compare the result with \eqref{e.rate0.j}.
Doing so gives
\begin{align}
	\label{e.a3}
	\rate^\circ_j(\mu_{\delta}) \geq \int_0^1 \d a \, \Big( v(a) \quant_\delta(s,a)\big|_{s_{j-1}}^{s_j} - v(a)^2(s_j-s_{j-1}) \Big).
\end{align}
Optimizing \eqref{e.a3} over $v\in\Csp^\infty[0,1]$ gives
\begin{align}
	\label{e.a4}
	\rate^\circ_j(\mu_{\delta})
	\geq
	(s_j-s_{j-1}) \int_0^1 \d a \, \frac{1}{4} \Big( \frac{\quant_\delta(s_j,a)-\quant_\delta(s_{j-1},a)}{s_j-s_{j-1}} \Big)^2.
\end{align}

We are now ready to conclude the desired result.
Combining \eqref{e.a2} and \eqref{e.a4} gives
\begin{align}
	\label{e.a5}
	\rate(\mu) + \frac{\termts}{8} \geq \frac{1}{4} \int\!\!{}_{[0,\termts]} \d s \int\!\!{}_{[0,1]} \d a \, ( \partial_\tparti \quant_\delta )^2.
\end{align}	
With the aid of Lemma~\ref{l.useful}\ref{l.useful.quant.conti}, it is not hard to check that $\quant_\delta\to\quant$ as $\delta\to 0$ Lebesgue a.e.\ on $[0,\termts]\times[0,1]$.
Given this property, sending $\delta\to0$ in \eqref{e.a5} with the aid of Fatou's lemma gives the desired result \eqref{e.p.rate=rateq.case2}.

\subsection{The rate function $ \rate_\star $ is a good}
\label{s.rate.good}
Here we prove that $ \rate_\star $ is good.
Recall from \eqref{e.rate.ic} that $ \rate_\star $ consists of $ \rate $ and a dependence on the initial condition.
It hence suffices to show that $ \rate=\rateq $ is a good rate function when restricted to $ \{\mu: \mu(0) = \mu_\Start \} $.
We will show that $ \rate $ is lower-semicontinuous on $ \Csp([0,\termts],\Psp(\R)) $ and that $ S := \{ \mu : \rateq(\mu) \leq r, \mu(0)=\mu_\Start \} $ is precompact.

Let us show the lower semicontinuity of $ \rate $. 
We begin with some reductions.
Recall from \eqref{e.rate} that $\rate$ is defined as a supremum.
Since the supremum of any set of continuous functions is lower semicontinuous, it suffices to check that, for any $\h\in{\Cbsp}^{1,1}([0,\termts],\R)$,
\begin{align}
	\label{e.ratesup.conti}
	\Csp([0,\termts],\Psp(\R)) \to \R,
	\quad
	\mu \longmapsto \rateM(\mu,\h) - \int_0^{\termt} \d s\,\frac{1}{2} \ip{\mu, (\partial_x\h)^2}
	\quad
	\text{is continuous.}
\end{align}
Every term in \eqref{e.rateM} and \eqref{e.rate} is readily seen to be continuous except for $ \int\!\!{}_{[0,\termts]} \d s \,\ip{\mu, \Sgn[\mu](\partial_x\h)} $.
It hence suffices to show that, for any $ \f\in\Cbsp(\R) $, the map $ \Psp(\R)\to\R $: $ \lambda\mapsto\ip{\lambda,\Sgn[\lambda]\f} $ is continuous.
To this end, combine Lemma~\ref{l.useful}\ref{l.useful.inverseCDF} and the identify \eqref{e.Sgn.id.int} to get
$
	\ip{ \lambda, \Sgn[\lambda] \f }
	=
	\int_0^1 \d a \, (\tfrac12 -a) \, \f(\quant(a)).
$
Take any sequence $ \lambda_k \Rightarrow \lambda $.
Note that $ \lambda_k \Rightarrow \lambda $ implies $ \quant[\lambda_k] \to \quant[\lambda] $ Lebesgue a.e.\ on $ [0,1] $.
This property together with the bounded convergence theorem gives $ \ip{ \lambda_k, \Sgn[\lambda_k] \f } =\int_0^1 \d a \, (\tfrac12 -a) \, \f(\quant[\lambda_k](a)) \to \int_0^1 \d a \, (\tfrac12 -a) \, \f(\quant(a)) = \ip{ \lambda, \Sgn[\lambda] \f } $.

To prepare for the proof of the precompactness, we derive a time-continuity estimate.
Recall from \eqref{e.rateq} that $ \rateq $ is defined as an integral over time.
Fix $ s'<s''\in[0,\termts] $, forgo the integral outside $ [s',s''] $, factor out $\frac12(s''-s') $ from the integral, and apply Jensen's inequality with respect to 
$ 
	(s''-s')^{-1}\int\!\!{}_{[s',s'']}\d s\int\!\!{}_{[0,1]}\d a 
$.
Doing so gives
\begin{align}
	\rateq(\mu) 
	\geq
	\frac12 (s''-s') \Big( 
		\frac{\int_0^1 \d a \, \quant[\mu(s)]|_{s'}^{s''}}{s''-s'} - \frac{\int_{s'}^{s''}\d s \int_{0}^{1} \d a \, \Sgn[\mu]_{\quant}}{s''-s'}
	\Big)^2.
\end{align}
Call the first and second terms within the last square $ b $ and $ b' $ respectively.
Use the inequality $ (b-b')^2 \geq \frac12 b^2 - {b'}^2 $ and note that $ |b'|\leq\norm{\Sgn[\mu]}_\infty \leq 1/2 $.
After being simplified, the result reads
$
	(s''-s')\rateq(\mu) + \frac{(s''-s')^2}{8} \geq \frac14 \int_0^1 \d a\, (\quant[\mu(s'')]-\quant[\mu(s')] )^2.
$
By the Cauchy--Schwarz inequality, the last integral is bounded from below by $ \int_{0}^{1}\d a\, |\quant[\mu(s'')]-\quant[\mu(s')]| $.
We arrive at the time-continuity estimate:
\begin{align}
	\label{e.good.conti}
	\frac14 \int_0^1 \d a \,\Big|\quant[\mu(s'')]-\quant[\mu(s')]\Big|
	\leq
	(s''-s')\rateq(\mu) + \frac{1}{8}(s''-s')^2. 
\end{align}

Based on \eqref{e.good.conti}, we fix $ r\in[0,\infty) $ and $\mu_\Start\in\Psp(\R)$ and show the precompactness of $ S:=\{ \mu : \rateq(\mu) \leq r, \mu(0)=\mu_\Start \} $. 
We will do so by verifying the conditions in Lemma~\ref{l.precompact}.
Referring to \eqref{e.wass.inverseCDF}--\eqref{e.dist}, we see that $ \dist(\mu(s''),\mu(s')) $ is bounded by 4 times the left hand side of \eqref{e.good.conti}.
Hence the equicontinuity of $S$, which is required by Lemma~\ref{l.precompact}\ref{l.precompact.conti}, follows.
To verify the condition in Lemma~\ref{l.precompact}\ref{l.precompact.tailbd}, take any $ \mu\in S $ and write $ \ip{\mu(s),\ind_{\R\setminus[-b,b]}} $ as $ \int\d a \, \ind\{|\quant[\mu(s)]|>b\} $.
Bound the last integral by $ \int\d a\, \ind\{|\quant[\mu(0)]|>b/2\} + \int\d a \,\ind\{|\quant[\mu(s)]-\quant[\mu(0)]|>b/2\} $.
Recognize the former integral as $ \ip{\mu(0),\ind_{\R\setminus[-b/2,b/2]}} $ and bound the second integral by using Markov's inequality and \eqref{e.good.conti} for $ (s',s'')=(0,s) $.
Doing so gives the bound $ (8/b)(s\rateq(\mu)+s^2/8)$, which is at most $(8r\termt +\termt^2)/b $.
Hence, $ \ip{\mu(s),\ind_{\R\setminus[-b,b]}} \leq \ip{\mu_\Start,\ind_{\R\setminus[-b/2,b/2]}} + (8r\termt +\termt^2)/b $.
From this, we see that the condition in Lemma~\ref{l.precompact}\ref{l.precompact.tailbd} is satisfied for a suitable choice of $ b_k\to\infty $.

\section{LDP for the attractive BPs: upper bound}
\label{s.upbd}
Here we prove the LDP upper bound in Theorem~\ref{t.ldp}.
We achieve this by first establishing the exponential tightness of $ \EM_N $ (defined at the beginning of Section~\ref{s.upbd.exptight}) and then proving the weak LDP upper bound (defined at the beginning of Section~\ref{s.upbd.weak}).
As was explained in Section~\ref{s.basic.m=1}, we consider $ \mm=1 $ only.

\subsection{Exponential tightness}
\label{s.upbd.exptight}
Here, we seek to prove that $\EM_N$ is exponentially tight, which means, for any $\e>0$, there exists a precompact $ S \subset \Csp([0,\termts],\Psp(\R)) $ such that
\begin{align}
	\label{e.exptight}
	\limsup_{N\to\infty} \frac{1}{N^3T} \log \P[ \EM_N \notin S ] \leq -\frac{1}{\e}.
\end{align}

The first step is to devise some events to control the BMs $\bm_\ii$ in \eqref{e.aBP}.
Consider the events
\begin{align}
	\label{e.event.bmcontrol.}
	\calU_N([s_1,s_2],v) 
	&:= 
	\Big\{ 
		\sum_{\ii=1}^N \sup_{s\in[s_1,s_2]}|\bm_{\ii}(Ts) - \bm_{\ii}(Ts_1)|  \leq N^2T v |s_2-s_1|^{1/3} 
	\Big\},
\\
	\label{e.event.bmcontrol}
		\calU_N(v)
		&:= 
		\bigcap_{\ell\geq 1} \bigcap_{j=1}^{\ell} \calU_N\big( \big[ \tfrac{j-1}{\ell}\termt, \tfrac{j}{\ell}\termts\big], v \big).
\end{align}
To control the summand in \eqref{e.event.bmcontrol.}, set
$
	\bar{\bm}_{\ii}[s_1,s_2]:=\sup_{s\in[s_1,s_2]} (\bm_{\ii}(Ts) - \bm_{\ii}(Ts_1))
$
and
$
	\und{\bm}_{\ii}[s_1,s_2]:=-\inf_{s\in[s_1,s_2]} (\bm_{\ii}(Ts) - \bm_{\ii}(Ts_1))
$
and write
\begin{align}
	\label{e.event.bmcontrol1}
	\sup_{s\in[s_1,s_2]}|\bm_{\ii}(Ts) - \bm_{\ii}(Ts_1)|
	=
	\max\{ \bar{\bm}_{\ii}[s_1,s_2], \und{\bm}_{\ii}[s_1,s_2] \}
	\leq
	\bar{\bm}_{\ii}[s_1,s_2] + \und{\bm}_{\ii}[s_1,s_2].
\end{align}
Sum \eqref{e.event.bmcontrol1} over $\ii=1,\ldots,N$.
For $\lambda\in\R$, we have the inequality
\begin{align}
	\label{e.event.bmcontrol1.}
	\E\big[e^{\lambda \sum_{\ii=1}^N \sup_{s\in[s_1,s_2]}|\bm_{\ii}(Ts) - \bm_{\ii}(Ts_1)| }\big]
	\leq
	\E\big[e^{\lambda \sum_{\ii=1}^{N}\bar{\bm}_{\ii}[s_1,s_2]} \cdot e^{\lambda\sum_{\ii=1}^{N}\und{\bm}_{\ii}[s_1,s_2]}\big].
\end{align}
The random variables $ \bar{\bm}_{\ii}[s_1,s_2] $ and $\und{\bm}_{\ii}[s_1,s_2]$ each have the same law as $ \sqrt{T(s_2-s_1)}|\bm_{\ii}(1)| $.
On the right hand side of \eqref{e.event.bmcontrol1.}, use the Cauchy--Schwarz inequality, use the just-mentioned property about the law, and use the property that $\bm_{1},\ldots,\bm_{N}$ are independent.
Doing so gives
\begin{align}
	\label{e.event.bmcontrol1..}
	\E\big[e^{\lambda \sum_{\ii=1}^N \sup_{s\in[s_1,s_2]}|\bm_{\ii}(Ts) - \bm_{\ii}(Ts_1)| }\big]
	\leq
%	\E\big[ e^{2\lambda \sum_{\ii=1}^{N} \sqrt{T(s_2-s_1)}|\bm_{\ii}(1)| }\big]
%	=
	\big( \E\big[ e^{2\lambda \sqrt{T(s_2-s_1)}|\bm_{1}(1)| }\big] \big)^{N}.
\end{align}
The inequality $\exp(|r|)\leq \exp(r)+\exp(-r)$ and the identity $\E[\exp(\pm r\bm_1(1))]=\exp(r^2/2)$ together give $\E[ e^{2\lambda \sqrt{T(s_2-s_1)}|\bm_{1}(1)| }]\leq 2 e^{2\lambda^2 T(s_2-s_1)}$.
Using this in \eqref{e.event.bmcontrol1..} gives
\begin{align}
	\label{e.event.bmcontrol1...}
	\E\big[e^{\lambda \sum_{\ii=1}^N \sup_{s\in[s_1,s_2]}|\bm_{\ii}(Ts) - \bm_{\ii}(Ts_1)| }\big]
	\leq
	2^N e^{2\lambda^2 NT(s_2-s_1)}.
\end{align}
Combining \eqref{e.event.bmcontrol1...} and Chernoff's bound $\P[A>a]\leq e^{-\lambda a}\E[e^{\lambda A}]$ with $A=\sum_{\ii=1}^N\sup_{s\in[s_1,s_2]}|\bm_{\ii}(Ts) - \bm_{\ii}(Ts_1)|$, $a=N^2T v |s_2-s_1|^{1/3}$, and $\lambda=N|s_2-s_1|^{-2/3}v/4$ yields
%\begin{align}
%	\label{e.event.bmcontrol2}
%	\P\Big[\,\sup_{s\in[s_1,s_2]}\sum_{\ii=1}^N|\bm_{\ii}(Ts) - \bm_{\ii}(Ts_1)| > NT v |s_2-s_1|^{1/3}  \Big] \leq 2 \exp\big(-\tfrac{1}{8} N^2Tv^2|s_2-s_1|^{-1/3}\big).	
%\end{align}
%Since $\bm_1,\ldots,\bm_N$ are independent, taking the product of \eqref{e.event.bmcontrol2} over $\ii=1,\ldots,N$ gives
\begin{align}
	\label{e.event.bmcontrol3}
	\P[\,\calU_N([s_1,s_2],v)^\comple] \leq 2^{N} \exp\big(-\tfrac{1}{8} N^3Tv^2|s_2-s_1|^{-1/3}\big).	
\end{align}
Next, use \eqref{e.event.bmcontrol3} and the union bound to bound $\P[\,\calU_N(v)^\comple] $.
We have, for all $v\geq 1$ and some universal constant $c<\infty$,
\begin{align}
	\label{e.event.bmcontrol4}
	\P[\,\calU_N(v)^\comple] 
	\leq 
	2^{N} \, \sum_{\ell=1}^\infty \ell \cdot e^{-\tfrac{1}{8} N^3Tv^2\ell^{1/3}} 
	\leq
	c 2^{N} \, e^{-\tfrac{1}{8} N^3Tv^2}.		
\end{align}
Under \eqref{e.scaling}, the factor $c2^N$ is negligible compared to $\exp(O(N^3T))$.
Use \eqref{e.event.bmcontrol4} to fix a large enough $v$ such that 
\begin{align}
	\limsup_{N\to\infty} \frac{1}{N^3T} \log \P\big[ \calU_N(v)^\comple \big]
	\leq
	-\frac{1}{\e}.
\end{align}
Having fixed $v$, we write $\calU_N(v)=\calU_N$ hereafter.

We next prove that, under $ \calU_N $, all samples of $ \EM_N $ are contained in a fixed precompact set.
This will imply the desired exponential tightness.
Proving this statement amounts to verifying the conditions in Lemma~\ref{l.precompact} --- for all realizations of $ \EM_N $ under $ \calU_N $ --- with a fixed choice of $ b_1,b_2,\ldots\to \infty $.

To verify the equicontinuity required by Lemma~\ref{l.precompact}\ref{l.precompact.conti}, apply \eqref{e.wass.coupling} to get
\begin{align}
	\label{e.event.bmcontrol6}
	\dist(\EM_N(s'), \EM_N(s'')) \leq \frac{1}{N}\sum_{\ii=1}^{N} \frac{|\X_{\ii}(Ts'')-\X_{\ii}(Ts')|}{NT},
	\qquad
	\text{for all }  s'<s''\in[0,\termts].
\end{align}
To bound the last sum, integrate \eqref{e.aBP} over $ [Ts',Ts''] $ and use $ |\sgn| \leq 1 $.
Doing so gives
\begin{align}
	\label{e.event.bmcontrol7}
	\frac{1}{N}\sum_{\ii=1}^{N} \frac{|\X_{\ii}(Ts'')-\X_{\ii}(Ts')|}{NT}
	\leq 
	\frac{|s''-s'|}{2} 
	+ 
	\frac{1}{N}\sum_{\ii=1}^{N} \frac{|\bm_{\ii}(Ts'')-\bm_{\ii}(Ts')|}{NT}.
\end{align}
Under $\calU_N$, the last term in \eqref{e.event.bmcontrol7} is bounded by $2v |s''-s'|^{1/3}$.
This is because, for any given $s'<s''\in[0,\termt]$, we can find $\ell\in\Z_{> 1}$ and $j\in\{1,\ldots,\ell-1\}$ such that $s',s''\in[\termt(j-1)/\ell,\termt(j+1)/\ell]$ and that $s''-s'> \termt/\ell$.
Hence
\begin{align}
\begin{split}
\label{e.aBP.control}
	&\text{under } \calU_N,
	\text{ for all } s'<s''\in[0,\termts],
\\
	&\quad
	\frac{1}{N}\sum_{\ii=1}^N \frac{1}{NT} |\X_{\ii}(Ts'')-\X_{\ii}(Ts')|
	\leq
	\frac{1}{2} |s''-s'|
	+
	2v |s''-s'|^{1/3}.
\end{split}
\end{align}
Combining \eqref{e.event.bmcontrol6} and \eqref{e.aBP.control} gives, under $\calU_N$, $\dist(\EM_N(s'), \EM_N(s'')) \leq |s''-s'|/2 + 2v |s''-s'|^{1/3} $, so the desired equicontinuity holds.

To verify the condition in Lemma~\ref{l.precompact}\ref{l.precompact.tailbd}, start by writing 
\begin{align}
	\label{e.exptight.tailbd}
	\ip{ \EM_N(s), \ind_{\R\setminus[-2b,2b]} }
	\leq
	\frac{1}{N} \sum_{\ii=1}^{N} \ind\{ \tfrac{1}{NT}|\X_{\ii}(Ts)-\X_{\ii}(0)| > b \} 
	+
	\ip{ \EM_N(0), \ind_{\R\setminus(-b,b)} }.
\end{align}
We have $\limsup_{N\to\infty}\ip{ \EM_N(0), \ind_{\R\setminus(-b,b)} } \leq \ip{ \mu_\Start, \ind_{\R\setminus(-b,b)} } $, because $ \EM_N(0)\Rightarrow \mu_\Start $ and because $ \R\setminus(-b,b) $ is closed.
Next, bound the summand in \eqref{e.exptight.tailbd} by $ b^{-1} |\X_{\ii}(Ts)-\X_{\ii}(0)|/(NT) $ and apply \eqref{e.aBP.control}.
The result gives that, under $ \calU_N $, the first term on the right side of \eqref{e.exptight.tailbd} is bounded by $ b^{-1} (s\cdot 2^{-1} +2v s^{1/3}) \leq b^{-1} (\termt/2+2v\termt^{1/3}) $.
These bounds together verify the condition in Lemma~\ref{l.precompact}\ref{l.precompact.tailbd} for $ b_k := (\termt/2+2v\termt^{1/3})2 k $.
 
\subsection{The weak upper bound} 
\label{s.upbd.weak} 
We begin by stating the goal.
First, given the exponential tightness, it suffices to prove a weak LDP upper bound, namely the LDP upper bound where the closed set is assumed to be compact.
Fix any compact $ K \subset \Csp([0,\termts],\Psp(\R)) $.
Recall $ \rate_\star $ from \eqref{e.rate.ic}.
If $ K \cap\{\mu : \mu(0) = \mu_\Start \} = \emptyset $, then $ \inf_{K} \rate_\star = +\infty $ and the desired upper bound follows trivially.
We hence assume $ K \cap\{\mu : \mu(0) = \mu_\Start \} \neq \emptyset $, whereby $ \inf_{K} \rate_\star = \inf_{K} \rate $.
Further, recall from \eqref{e.rate} that $ \rate $ is defined as a supremum over $ \h \in {\Cbsp}^{1,1}([0,\termts],\R) $.
Via an approximation argument, the supremum can be replaced by the one over $ \h \in {\Ccsp}^\infty([0,\termts],\R) $.
Hence our goal is to show
\begin{align}
	\label{e.upbd.weak}
	\limsup_{N\to\infty} \frac{1}{N^3 T} \log \P\big[ \EM_N \in K \big]
	\leq
	-\inf_{\mu\in K} \sup_{\h \in \Ccsp^{\infty}([0,\termts],\R)}\Big\{ \rateM(\mu,\h) - \int_{0}^{\termts} \d s \, \frac12 \Ip{ \mu, (\partial_x \h)^2 } \Big\}.
\end{align} 

Let us use the martingale method \cite{kipnis89} to prove \eqref{e.upbd.weak}.
Take any $ \h\in{\Ccsp}^{\infty}([0,\termts],\R) $ and apply It\^{o}'s calculus to $ N^3 T\ip{ \EM_N(s), \h(s) } =  N^2T \sum_{\ii} \h(s,X_\ii(Ts)/(NT)) $ with the aid of \eqref{e.aBP} to get
\begin{subequations}
\begin{align}
	\label{e.upbd.ito}
	N^3 T \ip{ \EM_N(s), \h(s) }\Big|_{0}^{s'}
	&-
	N^3 T \int_0^{s'} \d s \, \Ip{ \EM_N, \partial_s\h + \Sgn[\EM_N] (\partial_x\h) + \tfrac{1}{2N^2T}\partial_{xx}\h }
\\
	&=
	NT^{1/2} \sum_{\ii=1}^N \int_0^{s'} \d \bm_{\ii}(s)\,  (\partial_x \h)\big(s,\tfrac{1}{NT}X_\ii(Ts)\big).
\end{align}
\end{subequations}
This result implies that the expression in \eqref{e.upbd.ito}, when viewed as a process in $s'$, is a martingale with the quadratic variation $ N^3T\int_0^{s'} \d s\, \ip{\EM_N,(\partial_x\h)^2} $.
Hence $ \E[\exp(N^3 T U_N(\EM_N))]=1 $, where
\begin{align}
	U_N(\mu) := \rateM(\mu,\h) - \int_0^{\termt} \d s \, \frac{1}{2}\ip{\mu,(\partial_x\h)^2} - \frac{1}{2N^2T}\int_0^{\termt} \d s \, \ip{\mu,\partial_{xx}\h},
\end{align}
and $\rateM$ was defined in \eqref{e.rateM}.
Given any Borel $ A \subset \Csp([0,\termts],\Psp(\R)) $, write
\begin{align}
	\P[\EM_N\in A] = \E[ \ind_{A}(\EM_N) \exp(-N^3 T U_N(\EM_N))\exp(N^3 T U_N(\EM_N)) ],
\end{align}
bound the first two factors together by $ \sup_{\mu\in A} \exp(-N^3 T U_N(\mu)) $, use $ \E[\exp(N^3 T U_N(\EM_N))]=1 $ for the remaining factor, and  apply $ \limsup_{N} (N^3 T)^{-1} \log(\Cdot) $ to both sides of the result.
Doing so gives
\begin{subequations}
\begin{align}
	\limsup_{N\to\infty} \frac{1}{N^3 T} \log \P[ \EM_N \in A ]
	&\leq
	\lim_{N\to\infty} \sup_{\mu\in A} U_N(\mu)
\\
	&= 
	- \inf_{\mu \in A} \Big\{ 
		\rateM(\mu,\h) - \int\!\!{}_{[0,\termts]} \d s \, \frac{1}{2}\ip{\mu,(\partial_x\h)^2} 
	\Big\}.
\end{align}
\end{subequations}
Since this holds for all $ \h\in{\Ccsp}^\infty([0,\termts],\R) $, we further obtain
\begin{align}
	\label{e.upbd.swap}
	\limsup_{N\to\infty} \frac{1}{N^3 T} \log \P\big[ \EM_N \in A \big]
	&\leq
	- \sup_{\h\in \Ccsp^{\infty}([0,\termts],\R)} \inf_{\mu \in A} \Big\{ \rateM(\mu,\h) - \int_0^{\termt} \d s \, \frac{1}{2}\ip{\mu,(\partial_x\h)^2} \Big\}.
%\\
%	&= \inf_{\h\in \Ccsp^{\infty}([0,\termts],\R)} \sup_{\mu \in A} \Big\{ -\rateM(\mu,\h) + \int_0^{\termt} \d s \, \frac{1}{2}\ip{\mu,(\partial_x\h)^2} \Big\}.
\end{align}
We seek to swap the supremum and infimum on the right hand side.
To this end, apply Lemmata~3.2--3.3 in Appendix~2 in \cite{kipnis98}, with $ J_\h(\mu) = -\rateM(\mu,\h) + \int\!\!{}_{[0,\termts]} \d s \, \frac{1}{2}\ip{\mu,(\partial_x\h)^2} $.
Under our notation, these Lemmata assert that, if $J_\h(\mu)$ is upper-semicontinuous in $\mu$ for every $\h$, and if \eqref{e.upbd.swap} holds for every open $A$, then \eqref{e.upbd.weak} holds for every compact $ K$.
As shown after \eqref{e.ratesup.conti}, $ J_\h(\mu) $ is continuous in $ \mu $ for every $\h$.
We have proved \eqref{e.upbd.swap} for every Borel $A$.
Hence the desired weak upper bound \eqref{e.upbd.weak} follows.

\section{LDP for the attractive BPs: lower bound}
\label{s.lwbd}
We begin by setting up the goal of the proof.
As mentioned in Section~\ref{s.basic.m=1}, we consider $ \mm=1 $ only.
Indeed, proving the LDP lower bound amounts to proving that, for any $\mu\in\Csp([0,\termts],\Psp(\R))$ and $\e>0$,
$
	-\rate_\star(\mu)
	\leq 
	\liminf_{\e\to 0}\liminf_{N\to\infty} (N^3 T)^{-1}\log\P[\dist_{[0,\termts]}(\EM_N,\mu)<\e].
$ 
Recall from \eqref{e.rate.ic} that, when $\rate_\star(\mu)<\infty$, we have $\mu(0)=\mu_\Start$ and $\rate_\star(\mu)=\rateq(\mu)$.
Our goal is hence as follows.

\begin{prop}\label{p.ldp.lw}
For any $\mu\in\Csp([0,\termts],\Psp(\R))$ with $\mu(0)=\mu_\Start$ and $\rateq(\mu)<\infty$,
\begin{align}
	\label{e.lwbd.goal}
	\liminf_{r\to 0}\liminf_{N\to\infty} \frac{1}{N^3 T}\log\P\big[\dist_{[0,\termts]}(\EM_N,\mu)<r\big] \geq -\rateq(\mu).
\end{align} 
\end{prop}

\subsection{Proving a preliminary version of Proposition~\ref{p.ldp.lw}}
\label{s.lwbd.plclustering}
Let us introduce some classes of deviations.
We call a deviation $\xi\in\Csp([0,\termts],\Psp(\R))$ \textbf{clustering} if it is of the form $\xi(s)=\sum_{\cc=1}^n \mm_\cc\delta_{\xi_{\cc}(s)} $, for some $\xi_1\leq\ldots\leq\xi_{n}\in\Csp[0,\termts]$ and some $\mm_1,\ldots,\mm_n\in(0,1]$ that add up to $\mm=1$.
We call $\xi_1,\ldots,\xi_n$ the \textbf{clusters} or the \textbf{trajectories of clusters}.
We call a deviation \textbf{Piecewise-Linear(PL)-clustering} if it is clustering and the trajectories of its clusters are piecewise linear.

As a first step toward proving Proposition~\ref{p.ldp.lw}, here we prove a preliminary version of it where $\mu$ is replaced by a PL-clustering deviation $\xi$.
For a PL-clustering deviation, we will use a different way (than $\dist_{[0,\termts]}$) to measure how close the empirical measure is to the deviation.
Recall that the attractive BPs are ordered at the start: $\X_{1}(0)\leq\ldots\leq\X_{N}(0)$.
For a given PL-clustering $\xi$, to each $\X_\ii$ assign a cluster $\xi_\cc$ via the ordering of the \emph{indices}: 
\begin{align}
	\label{e.assignment}
	\cc_{N,\xi}(\ii):=\min\{ \cc : \mm_1+\ldots+\mm_\cc \geq i/N \},
	&&
	\group_{N,\xi}(\cc) := \cc_{N,\xi}^{-1}(\{\cc\}) = \{ \ii : \cc_{N,\xi}(\ii)=\cc\}.
\end{align}
For $\vec{x}\in(\Csp[0,T\termts])^{N}$, set 
\begin{align}
	\DIST_{N,s}(\vec{x},\xi)&:=\max\big\{ |\tfrac{1}{NT}x_{\ii}(Ts)-\xi_{\cc_{N,\xi}(\ii)}(s)|:\ii=1,\ldots,N \big\},
\\
	\DIST_{N,[s,',s'']}(\vec{x},\xi)&:=\sup_{s\in[s',s'']} \DIST_{N,s}(\vec{x},\xi).
\end{align}
We measure how close the system of attractive BPs is to $\xi$ by $\DIST_{N,s}(\vec{\X},\xi)$, which is a (much) finer measurement than $\dist(\EM_N(s),\xi(s))$.
By \eqref{e.wass.coupling},
\begin{align}
	\label{e.dist<DIST}
	\dist(\EM_N(s),\xi(s)) \leq \DIST_{N,s}\big(\vec{\X},\xi\big).
\end{align}
Hereafter, $c=c(v_1,v_2,\ldots)$ denotes a constant that depends only on the designated variables $v_1,v_2,\ldots$.
We seek to prove the following preliminary version of Proposition~\ref{p.ldp.lw}.
\begin{prop}\label{p.lwbd.clustering}
Given any PL-clustering $\xi$, there exists a $c=c(\xi)\in[1,\infty)$ such that
\begin{align}
\label{e.lwbd.clustering}
	\liminf_{r \to 0}
	\liminf_{N\to\infty}
	&\inf_{	\DIST_{N,0}(\vec{\X},\xi)\leq r}
	\Big\{
		\frac{1}{N^3T} \log 
		\P\big[ \DIST_{N,[0,\termts]}(\vec{\X},\xi) \leq c r \, \big| \, \vec{\X}(0) \big]
	\Big\}
	\geq
	-\rateq(\mu).
\end{align}
\end{prop}
\noindent{}
In \eqref{e.lwbd.clustering} and similarly hereafter, the conditional probability $\P[\ldots|\vec{\X}(0)]$ is viewed as a function of $\vec{\X}(0)$, and the infimum is taken over those $\vec{\X}(0)$s that satisfy $\DIST_{N,0}(\vec{\X},\xi) \leq r$.
\begin{proof}
We begin by setting up the notation.
Let $0=s_0<s_1<\ldots<s_\ell=\termt$ be the time between which the clusters of $\xi$ are linear.
For a small $\delta>0$ to be fixed shortly, we call $[s_0,s_1-\delta]$, $[s_1+\delta,s_2-\delta]$, \ldots, $[s_{\ell-1}+\delta,s_{\ell}]$ the \textbf{linear segments}, and call $[s_1-\delta,s_1+\delta]$, $[s_2-\delta,s_2+\delta]$, \ldots, the \textbf{transition segments}.
To fix the value of $\delta$, consider the maximum speed of all clusters $\dot\xi_\mathrm{max} := \sup_{[0,\termt],\cc} |\dot{\xi}_\cc|$, and note that, by the definition of the $s_j$s, within each $(s_{j-1},s_j)$, the clusters of $\xi$ either never meet or completely coincide.
We fix a small enough $\delta$ such that
\begin{align}
	\label{e.delta.small}
	&4\delta \dot\xi_\mathrm{max} \leq 2\cdot 3^{0}\, r, \text{ and}
\\
	\label{e.linears.order}
	\begin{split}
	&\text{for every linear segment, either }
	|\xi_\cc-\xi_{\cc'}| > 2\cdot 3^{2\ell} r \text{ everywhere on that linear segment, or}
	\\
	&
	\hphantom{\text{for every linear segment, either }}
	|\xi_\cc-\xi_{\cc'}| =0 \text{ everywhere on that linear segment.}
	\end{split}
\end{align}
%Within each linear or transition segment $[s',s'']$, we seek to prove a bound of the form:
%\begin{align*}
%	\liminf_{N\to\infty}
%	\inf_{\DIST_{N,s'}(\vec{\X},\xi)\leq \text{error}}\Big\{
%		\frac{1}{N^3 T} \log \P\big[\,\text{controlling event}[s',s'']\,\big|\, \vec{\X}(Ts') \,\big]
%	\Big\}
%	\geq
%	- \text{cost}[s',s''].
%\end{align*}
%The precise version of this inequality will be stated in \eqref{e.lwbd.linear} and \eqref{e.lwbd.transition}.
%The controlling event ensures that $\DIST_{N,[s',s'']}(\vec{\X},\xi)$ is small.
%To facilitate the control, we will choose a suitable error and require $\DIST_{N,s'}(\vec{\X},\xi)\leq\text{error}$.
%We will construct these events in such a way that, for any consecutive segments, the controlling event of the former implies the required error bound of the latter.
%This way the results within the segments can be concatenated.
%Within a linear segment, we will take the optimal cost.
%Within a transition segment, we will take a cost that, despite being suboptimal, is negligible as $\delta\to 0$. 
%
%Let us carry out the analysis within the linear segments and transition segments separately.

\medskip
\noindent\textbf{Step 1: analysis within a linear segment.}
Let us set up the notation and state the goal of this step.
Fix a linear segment $[s',s'']$ and define the event
\begin{align}
	\label{e.controlling}
	\control_{N,\beta}[s',s'']
	:=
	\big\{ \DIST_{N,[s',s'']}(\vec{\X},\xi) < 3^{\beta} r \big\}.
\end{align}
Hereafter, $\beta\in\{0,1,\ldots,2\ell-2\}$ is an auxiliary parameter. 
Let $\rateq_{[s',s'']}$ be the analog of $\rateq$ where the time integral (see \eqref{e.rateq}) is restricted to $[s',s'']$ and let $\dot{\varphi}:=\d\varphi/\d s$.
It is readily checked that
\begin{align}
	\label{e.clustering.rateq}
	\text{for a clustering } \nu = \sum_{\cc=1}^{n} \mm_\cc \delta_{\nu_\cc},
	\qquad
	\rateq_{[s',s'']}(\nu)
	=
	\int_{s'}^{s''} \d s\, \sum_{\cc=1}^{n} \frac{\mm_\cc}{2} \big( \dot{\nu_\cc} - \Sgn[\nu](\nu_\cc) \big)^2.
\end{align}
The goal of this step is to prove that
\begin{align}
	\label{e.lwbd.linear}
	\liminf_{N\to\infty}
	\inf_{ \DIST_{N,s'}(\vec{\X},\xi)\leq 3^{\beta} r }
	\Big\{
		\frac{1}{N^3 T} \log \P\big[ \,\control_{N,\beta+1}[s',s'']\,\big|\, \vec{\X}(Ts')\,\big]
	\Big\}
	\geq
	- \rateq_{[s',s'']}(\xi).
\end{align}
Recall from \eqref{e.linears.order} that, within $[s',s'']$, any pair of clusters either stay strictly apart or completely coincide.
After combining those clusters that completely coincide, we have
\begin{align}
	\label{e.linears.order.}
	|\xi_\cc-\xi_{\cc'}|_{[s',s'']} > 2\cdot 3^{2\ell} r,
	\qquad
	\forall \cc\neq \cc'.
\end{align}

The first step of proving \eqref{e.lwbd.linear} is to set up Girsanov's transform.
Within $[s',s'']$, the cluster $\xi_\cc$ travels at a constant velocity $\d\xi_\cc/\d s :=\dot{\xi}_\cc $.
Letting $v_\cc := \dot{\xi}_\cc - \pull_\cc$, where $\pull_\cc:=\tfrac12(\ldots-\mm_{\cc-1}+\mm_{\cc+1}+\ldots)$, we seek to apply Girsanov's transform to turn $\P$ into another law $\law$ where 
\begin{align}
	\label{e.lwbd.linear.SDE}
	\text{under } \law,
	\qquad
	\d \X_\ii = \sum_{\jj=1}^{N} \frac{1}{2} \sgn(\X_\jj-\X_\ii) \d s + N v_{\cc_{N,\xi}(\ii)} \d s + \d\bm_{\ii}.&
\end{align}
In plain words, we pursue a ``strategy'' where each $\X_{\ii}$ receives an additional drift $Nv_{\cc_{N,\xi}(\ii)}$.
The term $\dot{\xi}_\cc$ in $v_\cc$ helps $\X_{\ii}(T\Cdot)/(NT)$ follow the cluster $\xi_\cc$, while the term $-\pull_\cc$ counters the ``pulling'' from those particles with a different assigned cluster, namely the effect of the drift $\sum_{\jj\notin\group_{N,\xi}(\cc)}\sgn(\X_{\jj}-\X_{\ii})/2$.
Hereafter, we will use the phrase ``pulling'' similarly, to refer to the effect of the drift coming from a set of particles.

The next step is to set up a stopping time and prepare the relevant properties.
Let $\sigma$ be the first time when the condition required by $\control_{N,\beta+1}[s',s'']$ is violated, namely 
\begin{align}
	\sigma
	:=
	\inf\big\{ s\in[s',s'']: |\tfrac{1}{NT}\X_{\ii}(Ts)-\xi_{\cc_{N,\xi}(\ii)}(s)|= 3^{\beta+1}r, \text{for some } \ii \big\}.
\end{align}
Since $\beta<2\ell$, by \eqref{e.linears.order.}, particles with different assigned clusters stay strictly ordered within $[s',\sigma]$.
This property implies, for any $\ii\in\group_{N,\xi}(\cc)$, the pulling from particles outside $\group_{N,\xi}(\cc)$ is given by $\sum_{\jj\notin\group_{N,\xi}(\cc)}\frac{1}{2} \sgn(\jj-\ii)$.
The last expression sums to $N\frac12(-\mm_1\ldots-\mm_{\cc-1}+\mm_{\cc+1}+\ldots+\mm_n)=N\pull_\cc$.
Hence, within $s\in[s',\sigma]$ and for each $\ii\in\group_{N,\xi}(\cc)$,
\begin{align}
	\label{e.lwbd.step2.Pull}
	\sum_{\jj=1}^{N} \frac{1}{2} \sgn\big(\X_{\jj}(Ts)-\X_{\ii}(Ts)\big)
	=
	\sum_{\jj\in\group_{N,\xi}(\cc)} \frac{1}{2} \sgn\big(\X_{\jj}(Ts)-\X_{\ii}(Ts)\big) 
	+ 
	N\phi_\cc.
\end{align}

Having set up the stopping time $\sigma$, we next analyze the evolution of the particles within $s\in[s',\sigma]$ and under the condition 
$
	\DIST_{N,s'}(\vec{\X},\xi)\leq 3^{\beta} r,
$
which is required in \eqref{e.lwbd.linear}.
We will analyze separately the evolution of the center of mass (defined later) relative to their assigned clusters and the evolution of a particle relative to its center of mass.
Hereafter, unless otherwise noted, $s\in[s',\sigma]$.

We now define the center of mass and analyze its evolution.
Let $\X_{\cc}:=\sum_{\ii\in \group_{N,\xi}(\cc)}\X_{\ii}/(N\mm_\cc)$ be the center of mass.
Take the average of \eqref{e.lwbd.linear.SDE} over $\ii\in\group_{N,\xi}(\cc)$. 
The pulling from within $\ii\in\group_{N,\xi}(\cc)$ averages to zero; the pulling from outside of $\group_{N,\xi}(\cc)$ averages to $N\pull_\cc$ by \eqref{e.lwbd.step2.Pull}.
Hence $\d X_\cc = N\dot{\xi}_\cc \d s + \d \bm_\cc$, where $\bm_\cc:=\sum_{\ii\in \group_{N,\xi}(\cc)}\bm_{\ii}/(N\mm_\cc)$.
Integrating this equation gives 
\begin{align}
	|\tfrac{1}{NT}\X_\cc(Ts)-\xi_\cc(s)| \leq |\tfrac{1}{NT}\X_{\cc}(Ts')-\xi_\cc(s')| + \tfrac{1}{NT}|\bm_\cc(Ts)-\bm_\cc(Ts')|.
\end{align}
Under the condition
$
	\DIST_{N,s'}(\vec{\X},\xi)\leq 3^{\beta} r,
$
the first term on the right hand side is at most $3^{\beta} r$, so 
\begin{align}
\label{e.cm}
\begin{split}
	|\tfrac{1}{NT}\X_\cc(Ts)-\xi_\cc(s)| 
	&\leq 
	3^{\beta}r + \tfrac{1}{NT}|\bm_\cc(Ts)-\bm_\cc(Ts')|
\\
	&
	\leq
	3^{\beta}r + \tfrac{1}{NT}\max\big\{ |\bm_{\ii}(Ts)-\bm_{\ii}(Ts')| : \ii=1,\ldots,N \big\}.
\end{split}
\end{align}

We next analyze the motion of a particle relative to its center of mass.
Fix any $\cc$.
As seen in \eqref{e.lwbd.step2.Pull}, a particle with $\ii\in\group_{N,\xi}(\cc)$ feels the pulling from within and outside of $\group_{N,\xi}(\cc)$.
Intuitively, the former should only pull particles in $\group_{N,\xi}(\cc)$ closer together. 
To make this statement precise, we appeal to a comparison result from \cite{sarantsev2019camparison}.
Write $\group_{N,\xi}(\cc)=\{\ii_1,\ldots,\ii_{2}\}$, rank $\{\X_\ii(Ts)\}_{\ii\in\group_{N,\xi}}$ to get $\X_{(\ii_1)}(Ts)\leq\ldots\leq\X_{(\ii_2)}(Ts)$, let $Z_{\ii}(Ts):=\X_{(\ii+1)}(Ts)-\X_{(\ii)}(Ts)$ be the gap process, and let $L_{\ii,\ii+1}$ denote the localtime of $Z_{\ii}$ at $0$.
By \cite[Proposition~3]{ichiba2013strong}, the ranked processes satisfy
\begin{align}
	\label{e.Xi}
	\d \X_{(\ii)} 
	&= 
	\Big( N\pull_\cc + Nv_{\cc} + \sum_{\jj\in\group_{N,\xi}(\cc)} \frac{1}{2}\sgn(\jj-\ii) \Big) \d s
	+
	\d W_{\ii} + \tfrac{1}{2}\d L_{\ii-1,\ii} - \tfrac{1}{2}\d L_{\ii,\ii+1},
\end{align}
where $W_{\ii_1},\ldots,W_{\ii_2}$ are independent standard BMs, and we assume $W_{\ii}(s')=0$ without loss of generality.
Next, consider the analog of $\X_{(\ii)}$ without the inner pulling.
Namely, for $\ii\in\group_{N,\xi}(\cc)$, consider
\begin{align}
	\label{e.barXi}
	\d \bar{\X}_{(\ii)} = N\big( \pull_\cc + v_{\cc} \big) \d s 
	+ \d W_{\ii} + \tfrac{1}{2}\d \bar{L}_{\ii-1,\ii} - \tfrac{1}{2}\d \bar{L}_{\ii,\ii+1},
	\quad
	\bar{\X}_{\ii}(Ts')=\X_{(\ii)}(Ts'),
\end{align}
where $\bar{L}_{\ii,\ii+1}$ is the localtime of $\bar{\X}_{(\ii+1)}-\bar{\X}_{(\ii)}$ at $0$.
%The system \eqref{e.barXi} has as unique strong solution.
%This, for example, follows from the results in \cite[Section~2.1]{karatzas2016systems}.
Let $\bar{Z}_{i}:=\bar{\X}_{(\ii+1)}-\bar{\X}_{(\ii+1)}$ denote the gap process and let $f_i(s):= W_{\ii+1}(s) - W_{\ii}(s) + \X_{(\ii)}(Ts')$.
From \eqref{e.Xi}--\eqref{e.barXi}, we infer the equations of the gap processes
\begin{align}
	\label{e.Zi}
	\d Z_{(\ii)} 
	&= 
	\d f_{\ii} - \d s - \tfrac{1}{2}\d L_{\ii-1,\ii} + \d L_{\ii,\ii+1} - \tfrac{1}{2}\d L_{\ii+1,\ii+2},
\\
	\label{e.barZi}
	\d \bar{Z}_{(\ii)} 
	&= 
	\d f_{\ii} - \tfrac{1}{2}\d L_{\ii-1,\ii} + \d L_{\ii,\ii+1} - \tfrac{1}{2}\d L_{\ii+1,\ii+2},
\end{align}
for $\ii=\ii_1,\ldots,\ii_2-1$, with the convention that $L_{\ii_i-1,\ii_i}:=0$ and $L_{\ii_2,\ii_{2}+1}:=0$.
Now apply \cite[Theorem~3.1]{sarantsev2019camparison} to our $Z=(Z_{\ii})_{\ii\in[\ii_1,\ii_2)}$ and $\bar{Z}=(\bar{Z}_{\ii})_{\ii\in[\ii_1,\ii_2)}$, with $f_{\ii}(s)-(s-s')$ being the driving function for $Z_{\ii}$ and with $f_{\ii}$ being the driving function for $\bar{Z}_{\ii}$.
The result gives $Z_{\ii}(Ts)\leq\bar{Z}_{\ii}(Ts)$ for all $\ii\in[\ii_1,\ii_2-1)$ and $s\in[s',\sigma]$.
Further, letting $\bar{\X}_{\cc}:=\sum_{\ii\in\group_{N,\xi}(\cc)}\bar{\X}_{\ii}/|\group_{N,\xi}(\cc)|$ denote the center of mass of $(\bar{\X}_{(\ii)})_{\ii\in\group_{N,\xi}(\cc)}$ and averaging \eqref{e.barXi} over $\ii\in\group_{N,\xi}(\cc)$, we see that $\X_{\cc}=\bar{\X}_{\cc}$.
This property, together with the comparison result of the gap processes, gives
\begin{align}
\label{e.comparison}
	\max\big\{ |\X_{\ii}-\X_{\cc}| : \ii\in\group_{N,\xi}(\cc) \big\} 
	\leq 
	\max\big\{ |\bar{\X}_{(\ii)} - \bar{\X}_{\cc} | : \ii\in\group_{N,\xi}(\cc) \big\}.
\end{align}
Next, let $\bar{\bm}_{\ii}$, for $\ii\in\group_{N,\xi}(\cc)$, denote independent standard BMs, and consider
\begin{align}
	\label{e.Yi}
	\Y_{\ii}(Ts) := \bar{\bm}_{\ii}(Ts) - \bar{\bm}_{\ii}(Ts') + N(\pull_\cc + v_{\cc}) (Ts-Ts') + \X_{\ii}(Ts'),
\end{align}
for $s\in[s',\sigma]$ and $\ii\in\group_{N,\xi}(\cc)$.
If we rank $\{\Y_{\ii} \}_{\ii\in\group_{N,\xi}(\cc)}$, the resulting process has the same law as $(\bar{\X}_{(\ii)})_{\ii\in\group_{N,\xi}(\cc)}$.
This property can be checked from \eqref{e.barXi} by using the generators for example, which we omit.
Given this property, defining the event
\begin{align}
	\label{e.event.calF}
	\calF_N(r') := \big\{  |\X_{\ii}(Ts)-\X_{\cc}(Ts)| < NT(r'+3^{\beta} r), \ \forall\ii\in\group_{N,\xi}(\cc), \forall s\in[s',\sigma] \big\},
\end{align}
and letting $\Y_\cc:=\sum_{\ii\in\group_{N,\xi}(\cc)}\Y_{\ii}/|\group_{N,\xi}(\cc)|$ denote the center of mass of $(\Y_{\ii})_{\ii\in\group_{N,\xi}(\cc)}$, we have
\begin{align}
\label{e.comparison...}
\begin{split}
	\law&\big[\calF_N(r') \big|  \vec{\X}(Ts') \big]
\\
	&
	\geq
	\law\Big[ |\Y_{\ii}(Ts)-\Y_{\cc}(Ts)| < NT(r'+3^{\beta} r),\ \forall\ii\in\group_{N,\xi}(\cc), \forall s\in[s',\sigma]  \Big].
\end{split}
\end{align}
Next, we seek to bound $|\Y_{\ii}(Ts)-\Y_{\cc}(Ts)|$.
Recall the expression of $\Y_{\ii}(Ts)$ for \eqref{e.Yi}.
Averaging \eqref{e.Yi} over $\ii\in\group_{N,\xi}(\cc)$ gives the expression 
$
	\Y_{\cc}(Ts)
	=
	(\group_{N,\xi}(\cc))^{-1} \sum_{\jj\in\group_{N,\xi}(\cc)}(\bar{\bm}_{\jj}(Ts) - \bar{\bm}_{\jj}(Ts')) 
	+ 
	N(\pull_\cc + v_{\cc}) (Ts-Ts') 
	+ 
	\X_{\cc}(Ts').
$
Inserting these expressions into $|\Y_{\ii}(Ts)-\Y_{\cc}(Ts)|$, and bound the result using the triangle inequality.
Doing so gives
\begin{align}
	\label{e.Yibound}
	|\Y_{\ii}(Ts)-\Y_{\cc}(Ts)|
	\leq
	2\max_{\jj\in\group_{N,\xi}(\cc)} |\bar{\bm}_{\jj}(Ts)-\bar{\bm}_{\jj}(Ts')|
	+
	|\X_{\ii}(Ts')-\X_{\cc}(Ts'))|.
\end{align}
Under the condition $\DIST_{N,s'}(\vec{\X},\xi)\leq 3^{\beta} r$, the last term in \eqref{e.Yibound} is bounded by $NT\,3^{\beta} r$.
Hence, under the condition $\DIST_{N,s'}(\vec{\X},\xi)\leq 3^{\beta} r$, 
\begin{align}
\label{e.comparison..}
	\max_{\ii\in\group_{N,\xi}(\cc)} \big| \Y_{\ii}(Ts)-\Y_{\cc}(Ts) \big|
	\leq
	2\max_{\ii\in\group_{N,\xi}(\cc)} \big| \bar{\bm}_{\ii}(Ts)-\bar{\bm}_{\ii}(Ts') \big| + NT\,3^{\beta} r.
\end{align}
Using this property in \eqref{e.comparison...} gives
\begin{align}
\label{e.comparison.}
\begin{split}
	\inf_{ \DIST_{N,s'}(\vec{\X},\xi)\leq 3^{\beta} r }
	&\P\big[\calF_N(r') \big|  \vec{\X}(Ts') \big]
\\
	&\geq
	\P\big[ |\bar{\bm}_{\ii}(Ts)-\bar{\bm}_{\ii}(Ts')| < \tfrac{1}{2}NTr',\ \forall\ii\in\group_{N,\xi}(\cc), \forall s\in[s',\sigma]  \big].
\end{split}
\end{align}

Based on what we obtained so far, we now complete the proof of \eqref{e.lwbd.linear}.
Combining \eqref{e.cm} and \eqref{e.event.calF} gives, under the condition
$
	\DIST_{N,s'}(\vec{\X},\xi)\leq 3^{\beta} r
$
and within $s\in[s',\sigma]$,
\begin{align}
	\label{e.lwbd.step2.results}
	\big|\tfrac{1}{NT}\X_{\ii}(Ts)-\xi_{\cc_{N,\xi}(\ii)}(s)\big|
	\leq 
	2\cdot 3^{\beta} r + \tfrac{1}{NT}\max_{ \ii=1,\ldots,N} |\bm_{\ii}(Ts)-\bm_{\ii}(Ts')| + r'
	\quad
	\text{on }\calF_N(r').
\end{align}
We will take $r'\to 0$ later, so let us assume $r'<r/2$ hereafter. 
Consider the event
\begin{align}
	\label{e.event.calE}
	\calE_N(r') := \{ |\bm_{\ii}(Ts)-\bm_{\ii}(Ts')| < NTr', \forall s\in[s',s''],\forall \ii=1,\ldots,N \}.
\end{align}
Under $\calF_N(r')\cap\calE_N(r')$ and given the assumption $r'<r/2$, the right hand side of \eqref{e.lwbd.step2.results} is strictly less than $3^{\beta+1} r$.
Hence, under the condition $\DIST_{N,s'}(\vec{\X},\xi)\leq 3^{\beta} r$, we have
$
	\calF_N(r')\cap\calE_N(r')
	\subset
	\{ \sigma>s'' \}
	\subset
	\control_{N,\beta}[s',s'']
$.
Now apply Lemma~\ref{l.girsanov}\ref{l.girsanov.1} with the $v_{\ii}=v_{\cc_{N,\xi}(\ii)}$ defined above \eqref{e.lwbd.linear.SDE} and with $\calE=\calF_N(r')\cap\calE_N(r')$.
In the result recognize that $(s''-s')(2N)^{-1}\sum_{\ii=1}^N v_{\cc_{N,\xi}(\ii)}^2=\rateq_{[s',s'']}(\xi)$ with the aid of \eqref{e.clustering.rateq}, and bound $N^{-1}\sum_{\ii=1}^N |v_{\cc_{N,\xi}(\ii)}|r'$ by $c(\xi) r'$.
%\begin{align}
%	(s''-s')\sum_{\ii=1}^N \frac{v_{\cc_{N,\xi}(\ii)}^2}{2N} = \rateq_{[s',s'']}(\xi)
%\end{align}
Doing so gives
\begin{align}
\label{e.step1.rersult}
\begin{split}
	&\inf_{ \DIST_{N,s'}(\vec{\X},\xi)\leq 3^{\beta} r }
	\Big\{
		\frac{1}{N^3 T} \log \P\big[ \,\calF_N(r')\cap\calE_N(r')\,\big|\, \vec{\X}(Ts')\,\big]
	\Big\}
\\
	&
	\geq
	\inf_{ \DIST_{N,s'}(\vec{\X},\xi)\leq 3^{\beta} r }
	\frac{1}{N^3 T}\log\law\big[\,\calF_N(r')\cap\calE_N(r')\,\big|\, \vec{\X}(Ts')\,\big]
	-
	c(\xi) r'
	-
	\rateq_{[s',s'']}(\xi).
\end{split}
\end{align}
On the left hand side of \eqref{e.step1.rersult}, bound the probability from above by $\P[ \,\control_{N,\beta}[s',s'']\,|\, \vec{\X}(Ts')\,]$, and send $N\to\infty$ on both sides.
From \eqref{e.comparison.} and \eqref{e.event.calE}, it is not hard to check that the first term on the right hand side of \eqref{e.step1.rersult} converges to $0$ as $N\to\infty$.
Finally sending $r'\to 0$ gives the desired result~\eqref{e.lwbd.linear}.

\medskip
\noindent\textbf{Step 2: analysis within a transition segment.}
Let us set up the notation and state the goal of this step.
Fix a transition segment $[s',s'']$, consider the line (in the $(s,x)$ plane) that connects $(s',\xi_{\cc_{N,\xi}(\ii)}(s'))$ and $(s'',\xi_{\cc_{N,\xi}(\ii)}(s''))$, and let 
\begin{align}
	\label{e.velocity}
	\velo_\ii
	=
	\velo_\ii[s',s''] 
	:= 
	\frac{\xi_{\cc_{N,\xi}(\ii)}(s'') - \xi_{\cc_{N,\xi}(\ii)}(s')}{s''-s'}
	=
	 N \int_{(\ii-1)/N}^{\ii/N} \d a\, \frac{\quant[\xi(s'')]-\quant[\xi(s')]}{s''-s'}
\end{align}
be the velocity when traveling along the line.
Consider
\begin{align}
\label{e.controlling.}
	\controll_{N,\beta}[s',s'']
	&:=
	\Big\{ \, \big|\tfrac{1}{NT}\X_{\ii}(Ts)-\xi_{\cc_{N,\xi}(\ii)}(s')-(s-s')\velo_\ii \big| \leq 3^\beta r, \forall s\in[s',s''], \forall \ii \Big\},
\\
	\label{e.cost}
	\cost_{[s',s'']}(\lambda,\lambda')
	&:= 
	(s''-s') \int_0^1 \d a \, \Big( \frac{|\quant[\lambda']-\quant[\lambda]|}{s''-s'} + \frac{1}{2} \Big)^2.
\end{align}
By \eqref{e.delta.small}, it is not hard to check from \eqref{e.controlling.} that
\begin{align}
	\label{e.controlling..}
	\controll_{N,\beta}[s',s'']
	\subset
	\big\{ \DIST_{N,[s',s'']}\big(\vec{\X},\xi\big) \leq 3^{\beta+1} r \big\}.
\end{align}
The goal of this step is to prove
\begin{align}
\label{e.lwbd.transition}
	\liminf_{N\to\infty} 
	\ \inf
	\Big\{
		\frac{1}{N^3 T}\log\P\big[ \, \controll_{N,\beta+1}[s',s''] \, \big| \, & \vec{\X}(Ts') \big]
	\Big\}
	\geq
	-\cost_{[s',s'']}(\xi(s'),\xi(s'')),
\end{align}
where the infimum runs over $\DIST_{N,s'}(\vec{\X},\xi(s'))\leq 3^\beta r$, just like in Step~1.
%As said, the cost here is suboptimal.
%Later when applying \eqref{e.lwbd.transition}, we will make the segment $[s',s'']$ short so that the collective cost (over all transition segments) becomes negligible.

To prove \eqref{e.lwbd.transition}, consider the law $\law$ under which
$
	\d\X_\ii = N\velo_{\ii}\d s + \d\bm_{\ii}.
$
Apply Lemma~\ref{l.girsanov}\ref{l.girsanov.2} with $v_\ii=\velo_{\ii}$, $J_{\ii}=\{1,\ldots,N\}$, $\calE=\controll_{N,\beta}[s',s'']$, and $p=2$.
Doing so gives
\begin{align}
\label{e.lwbd.transition.girsanov}
\begin{split}
	\frac{1}{N^3T}&\log\P\big[ \, \controll_{N,\beta+1}[s',s''] \, \big| \, \vec{\X}(Ts') \, \big]
\\
	&\geq
	\frac{2}{N^3T}\log\law\big[ \, \controll_{N,\beta+1}[s',s''] \, \big| \, \vec{\X}(Ts') \, \big]
	-
	\frac{s''-s'}{N}\sum_{\ii=1}^N \Big( |\velo_{\ii}| + \frac{1}{2} \Big)^2.
\end{split}
\end{align}
It is readily checked that the first term on the right hand side of \eqref{e.lwbd.transition.girsanov} tends to $0$ as $N\to\infty$, uniformly over $\DIST_{N,s'}(\vec{\X},\xi(s'))\leq 3^\beta r$.
To bound the last term in \eqref{e.lwbd.transition.girsanov}, use \eqref{e.velocity} and Jensen's inequality to get
\begin{align}	
\label{e.lwbd.transition.jensen}
	\Big( |\velo_{\ii}| + \frac{1}{2} \Big)^2
	\leq
	N\int_{(\ii-1)/N}^{\ii/N} \d a\, \Big( \frac{|\quant[\xi(s'')] -\quant[\xi(s')]}{s''-s'} + \frac{1}{2} \Big)^2.
\end{align}
Applying $(s''-s')N^{-1}\sum_{\ii=1}^{N}$ on both sides of \eqref{e.lwbd.transition.jensen} gives
$
	(s''-s')N^{-1}\sum_{\ii} (|\velo_\ii{}|+1/2)^2
	\leq
	\cost_{[s',s'']}(\xi(s'),\xi(s'')).
$
The desired result \eqref{e.lwbd.transition} hence follows.

\medskip
\noindent\textbf{Step 3: Combining Steps~1--2 to complete the proof.}
Apply \eqref{e.lwbd.linear} and \eqref{e.lwbd.transition} in alternating order over the segments $[s_0,s_1-\delta]$, $[s_1-\delta,s_1+\delta]$, \ldots, $[s_{\ell-1}+\delta,s_{\ell}]$, with $\beta=0,1,\ldots,2\ell-2$.
By \eqref{e.controlling} and \eqref{e.controlling.},
$\control_{N,1}[s_0,s_1-\delta] \subset \{ \DIST_{N,s_1-\delta}(\vec{\X},\xi)\leq 3^1 \}$,
$\controll_{N,2}[s_1-\delta,s_1+\delta] \subset \{ \DIST_{N,s_1+\delta}(\vec{\X},\xi)\leq 3^2 \}$,
\ldots, 
so the resulting bounds can be concatenated.
The result gives
\begin{align}
\label{e.lwbd.clustering.}
\begin{split}
	\liminf_{N\to\infty}
	&\inf_{\DIST_{N,0}(\vec{\X},\xi)\leq r}
	\Big\{
		\frac{1}{N^3T} \log 
		\P\big[ \DIST_{N,[0,\termts]}(\vec{\X},\xi) \leq 3^{2\ell-1} r \, \big| \, \vec{\X}(0) \big]
	\Big\}
\\
	&\geq
	- \Big( \sum \rateq_{[s',s'']}(\xi) + \sum \cost_{[s',s'']}(\xi(s'),\xi(s'')) \Big),
\end{split}
\end{align}
where the first sum runs over all linear segments, and the second runs over all transition segments.
The first sum is at most $\rateq(\xi)$, while the second sum tends to zero as $\delta$ tends to zero, as is readily checked from \eqref{e.cost} and the assumption that $\partial_s\quant[\mu]\in\Lsp^2([0,\termts]\times[0,1])$.
Recall that $\delta$ is chosen so that \eqref{e.delta.small}--\eqref{e.linears.order} hold.
With $\xi$ being PL-clustering, we can choose $\delta$ such that $\delta \leq c(\xi) r$.
This way, $\delta\to 0$ as $r\to 0$.
Sending $r\to 0$ in \eqref{e.lwbd.clustering.} gives the desired result \eqref{e.lwbd.clustering}.
%
%
%By \eqref{e.controlling} and \eqref{e.controlling..}, under the controlling events, $|\X_{\ii}(Ts)/(NT)-\xi_{\cc_{N,\xi}(\ii)(s)}|\leq 3^{2\ell} r$, for all $s\in[0,\termts]$ and $\ii=1,\ldots,N$.
%Finally, the collective cost of achieving the controlling events is $\sum \rateq_{[s',s'']}(\xi) + \sum \cost_{[s',s'']}(\xi(s'),\xi(s''))$, where the first sum runs over all linear segments and the second runs over all transition segments.
%
%Recall that $\delta$ is chosen so that \eqref{e.delta.small}--\eqref{e.linears.order} hold.
%As $r\to 0$, the so chosen $\delta$ also tends to $0$.
\end{proof}

\subsection{Proving Proposition~\ref{p.ldp.lw} under Assumption~\ref{assu.truncation}}
\label{s.lwbd.truncated}
Based on the results in Section~\ref{s.lwbd.plclustering}, here we prove Proposition~\ref{p.ldp.lw} under an additional assumption.
Identify $\{\lambda\in\Psp(\R):\supp(\lambda)\subset[-b,b]\}$ with $\Psp[-b,b]$.
Namely, view those measures supported in $[-b,b]$ as measures on $[-b,b]$.

\begin{assu}\label{assu.truncation}
There exists a $b<\infty$ such that $\{\EM_N(0)\}_{N\in\Z_{>0}},\{\mu(s)\}_{s\in[0,\termts]}\subset\Psp[-b,b]$.
\end{assu}
\noindent{}%
Namely, we assume that $\EM_N(0)$ and $\mu(s)$ have a compact support, uniformly over $N\in\Z_{>0}$ and $s\in[0,\termts]$.
Fixing a $\mu$ and $\EM_N(0)$ that satisfy $\rateq(\mu)<\infty$, $\EM_N(0)\Rightarrow\mu(0)$, and Assumption~\ref{assu.truncation}, we seek to prove \eqref{e.lwbd.goal} for this $\mu$ and $\EM_N(0)$.

The proof requires an approximation tool. 
We state it here and put its proof in Appendix~\ref{s.a.approx}:
For any $[s',s'']\subset[0,\termt]$,
\begin{align}
\label{e.lwbd.plapprox}
\begin{split}
	\inf\big\{ &\dist_{[s',s'']}(\xi,\mu) + |\rateq_{[s',s'']}(\xi)-\rateq_{[s',s'']}(\mu)| \, :
\\ 
	&\xi \text{ is PL-clustering }, \supp(\xi(s))\subset[-b,b], \forall s\in[s',s''] \big\} = 0.
\end{split}
\end{align}

Let us outline the proof.
We will construct a small $s_0$ and use \eqref{e.lwbd.plapprox} to construct a PL-clustering $\xi$ that approximates $\mu$, and perform analysis over $[0,s_0]$ and $[s_0,\termts]$ separately.

Let us construct $s_0$ and $\xi$.
Fix an $\e>0$.
By the assumption $\rateq(\mu)<\infty$, $\partial_s\quant[\mu]\in\Lsp^2([0,\termts]\times[0,1])$.
Using this property and the time continuity of $\mu$ to find an $s_0$ such that $\cost_{[0,s_0]}(\mu(0),\mu(s_0))<\e$ and $\sup_{[0,s_0]}\dist(\mu(s),\mu(s_0))<\e$.
Turning to the construction of $\xi$, we begin by noting that $\cost_{[0,s_0]}:(\Psp[-b,b])^2\to\R$ is continuous, which is not hard to verify from \eqref{e.dist} and \eqref{e.cost}.
Granted this continuity, use \eqref{e.lwbd.plapprox} for $[s',s'']=[s_0,\termts]$ to obtain a PL-clustering $\xi$ such that $\dist_{[s_0,\termts]}(\xi,\mu)<\e$, $|\rateq_{[s_0,\termts]}(\xi)-\rateq_{[s_0,\termts]}(\mu)|<\e$, and
\begin{align}
	\label{e.property0}
	\cost_{[0,s_0]}(\mu(0),\xi(s_0))<\e.
\end{align}
Given that $\dist(\EM_N(0),\mu(0))\to 0$, for all $N$ large enough, $\cost_{[0,s_0]}(\EM_N(0),\xi(s_0))<\e$.

Let us perform analysis over $[0,s_0]$ and $[s_0,\termts]$, beginning with the former.
We seek to adapt the argument in Step~2 in Section~\ref{s.lwbd.plclustering}.
More precisely, we seek to apply the argument with $\velo_\ii[s',s'']$ (defined in \eqref{e.velocity}) replaced by
$
	\velo_\ii(s_0) := (\xi_{\cc_{N,\xi}(\ii)}(s_0) - \X_{\ii}(0)/(NT))/s_0
$
and $\controll_{N,\beta}[s',s'']$ (defined in \eqref{e.controlling.}) replaced by
\begin{align}
	\label{e.controlling.i}
	\controlll_{N,r}[0,s_0]
	:=
	\big\{ \, \big|\tfrac{1}{NT}\X_{\ii}(Ts)-\tfrac{1}{NT}\X_{\ii}(0)-\velo_\ii(s_0) s \big| < r, \forall s\in[0,s_0],\forall \ii \big\}.
\end{align}
Applying the argument in Step~2 in Section~\ref{s.lwbd.plclustering} with these adaptations gives
\begin{align}
	\label{e.lw.truncated.1}
	\liminf_{r\to0}\liminf_{N\to\infty}
	\frac{1}{N^3 T}\log\P\big[ \, \controlll_{N,r}[0,s_0]\big]
	\geq
	-\limsup_{N\to\infty} \cost_{[0,s_0]}(\EM_N(0),\xi(s_0)).
\end{align}
Unlike in \eqref{e.lwbd.transition}, we need not impose any constraint on $\vec{\X}(0)$ in \eqref{e.lw.truncated.1}.
This is because when defining the event in \eqref{e.controlling.i}, we use the reference point $\X_{\ii}(0)/(NT)$ instead of $\xi_{\cc_{N,\xi}(\ii)}(0)$; compared with \eqref{e.controlling.}.
Given that $\dist_{[0,s_0]}(\mu(s),\mu(0))<\e$, $ \dist(\xi(s_0),\mu(s_0))<\e $, and $\dist(\EM_N(0),\mu(0))\to 0$, for all $N$ large enough we have $\controlll_{N,r}[0,s_0]\subset \{ \dist_{[0,s_0]}(\EM_N,\mu)<r+2\e \}$.
Also, note that 
\begin{align}
	\label{e.property1}
	\controlll_{N,r}[0,s_0] \subset\{\DIST_{N,s_0}(\vec{\X},\xi)\leq r\}.
\end{align}
Move on to $[s_0,\termts]$.
Apply Proposition~\ref{p.lwbd.clustering} with $[0,\termts]\mapsto [s_0,\termts]$.
In the result, use \eqref{e.dist<DIST} and $\dist_{[s_0,\termts]}(\xi,\mu)<\e$.
Doing so gives
\begin{align}
	\label{e.lw.truncated.2}
	\liminf_{r\to 0} \liminf_{N\to\infty} \ \inf\Big\{
		\frac{1}{N^3T} \log \P\big[\, \dist_{[s_0,\termts]}(\EM_N,\xi) < \e + c r \, \big| \, \vec{\X}(Ts_0) \,\big] 
	\Big\}
	\geq
	-\rateq_{[s_0,\termts]}(\mu),
\end{align}
where the infimum runs over $\DIST_{N,s_0}(\vec{\X},\xi)\leq r$.
Further,
\begin{align}
	\label{e.property2}
	-\rateq_{[s_0,\termts]}(\mu) \geq-\rateq_{[0,\termts]}(\mu)=-\rateq(\mu).
\end{align}

Combine \eqref{e.lw.truncated.1} and \eqref{e.lw.truncated.2} with the aid of \eqref{e.property1}.
In the result, use \eqref{e.property0} and \eqref{e.property2}, and send $\e\to 0$.
Doing so gives \eqref{e.lwbd.goal} under Assumption~\ref{assu.truncation}.

\subsection{Removing Assumption~\ref{assu.truncation}}
\label{s.lwbd.remove.truncation}
Having proven \eqref{e.lwbd.goal} under Assumption~\ref{assu.truncation}, we explain how the same result follows without the assumption.
Taking any $\mu\in\Csp([0,\termts],\Psp(\R))$ with $\rateq(\mu)<\infty$ and $\mu(0)=\mu_\Start$, we seek to show that \eqref{e.lwbd.goal} holds for $\EM_N$ and this $\mu$.

We begin by truncating $\mu$ and $\EM_N$ so that the result satisfies Assumption~\ref{assu.truncation}.
Fix a sequence $1/2>r_1>r_2>\ldots\to 0$ such that $\quant[\mu_\Start]$ is continuous at every $r_k$ and $1-r_k$.
Fix any $k$.
Recall the scaling operator $\scale_\mm$ from Section~\ref{s.basic.m=1}.
Apply the procedure in Section~\ref{s.basic.divide} with $(\mm_1,\mm_2,\mm_3)=(r_k,1-2r_k,r_k)$ to get $\mu_1,\mu_2,\mu_3$ and take $\nu:=\mu_2$ as the truncated deviation.
By Lemma~\ref{l.useful}\ref{l.useful.tailbd}, the truncated deviation $\nu$ satisfies Assumption~\ref{assu.truncation}.
Next, truncate the empirical measure similarly.
Recall that $\X_{1}(0)\leq \ldots \leq \X_{N}(0)$.
Set $N_- := \lceil Nr_k \rceil $, $N_+ := \lfloor N(1-r_k)\rfloor $, $N':=N_+-N_-$, and
$
	\tEM_{N}(s) := \frac{1}{N} \sum_{N_-\leq\ii\leq N_+} \delta_{\X_{\ii}(Ts)/(NT)}.
$
Note that $\EM_N(0)\Rightarrow\mu_\Start$ is equivalent to $\quant[\EM_N(0)]$ converging to $\quant[\mu_\Start]$ everywhere the latter is continuous.
Recall that $\quant[\mu_\Start]$ is continuous at each $r_k$ and $1-r_k$.
These properties imply, as $N\to\infty$, $\tEM_N(0)\Rightarrow\nu(0)$, $X_{N_-}(0)/(NT)\to\quant[\mu_\Start](r_k)$, and $X_{N_+}(0)/(NT)\to\quant[\mu_\Start](1-r_k)$.
The last two convergences show that $\tEM_N(0)$ is supported in an $N$-independent bounded interval.

Next, we claim that
\begin{align}
	\label{e.lwbd.close.meausre}
	\dist_{[0,\termts]}(\EM_N,\mu)
	\leq
	\dist_{[0,\termts]}(\tEM_N,\nu)
	\leq 
	c\cdot(r_k + \tfrac{1}{N}).
\end{align}
Recall from Section~\ref{s.basic.divide} that the graph of $\cdf[\nu(s)]=\cdf[\mu_2(s)]$ is obtained by taking the graph of $\cdf[\mu(s)]$ between the levels $r_k$ and $1-r_k$ and shift the graph down so that the lower level is at $0$.
The same property holds for $(\cdf[\tEM_{N}(s)],\cdf[\EM_N(s)])$ with $(r_k,1-r_k)\mapsto(N_-/N,1-N_+/N)$.
By construction, $N_-/N$ and $N_+/N$ differ from $r_k$ by at most $1/N$.
Using these properties gives
\begin{align}
	\label{e.lwbd.close.meausre.}
	\int_{-k}^{k} \d a \, \big| \cdf[\EM_N(s)] - \cdf[\mu(s)] \big|
	\leq
	\int_{-k}^{k} \d a \, \Big(
		\big| \cdf[\tEM_N(s)] - \cdf[\nu(s)] \big|
		+	
		c  \cdot \big( r_k + \frac{1}{N} \big)
	\Big).
\end{align}
Combining \eqref{e.lwbd.close.meausre.} and \eqref{e.dist} gives \eqref{e.lwbd.close.meausre}.

Next, we would like to apply the result of Section~\ref{s.lwbd.truncated} with $\EM_N$ and $\mu$ being replaced by $\tEM_N$ and $\nu$.
However, note that the particles in $\tEM_N$ are not autonomous and feel the pulling from particles outside of $\tEM_N$.
Namely, after being restricted to $\ii\in[N_-,N_+]$, the system of equations \eqref{e.aBP} still contains terms $\frac12\sgn(\X_{\jj}-\X_{\ii})$ from outside of the system $\jj\notin[N_-,N_+]$.
To resolve this issue, consider the law $\law$ under which 
\begin{align}
	&\d \X_\ii = \sum_{\jj\in[N_-,N_+]} \frac{1}{2} \sgn(\X_\jj-\X_\ii) \d s + \d\bm_{\ii},
	&&
	\text{ for all } \ii\in[N_-,N_+],&
\\
	&\d\X_\ii = \sum_{\jj=1}^{N} \frac{1}{2} \sgn(\X_\jj-\X_\ii) \d s + \d\bm_{\ii},
	&&
	\text{ for all } \ii\notin[N_-,N_+].&
\end{align}
Under $\law$, the truncated empirical measure $\tEM_N$ evolves autonomously.
Apply Lemma~\ref{l.girsanov}\ref{l.girsanov.2} with $[s',s'']=[0,\termts]$, $v_\ii=0$, $J_{\ii}=\{1,\ldots,N_--1,N_++1,\ldots,N\}$ for $\ii\in[N_-,N_+]$, and $J_{\ii}=\emptyset$ for $\ii\notin[N_-,N_+]$.
In the result, bound $|J_\ii|/(2N)$ by $cr_k$.
Doing so gives
\begin{align}
	\label{e.lwebd.girsanov}
	\frac{1}{N^3T}\log\P\big[ \dist_{[0,\termts]}(\tEM_N,\nu)<r \big]
	\geq
	\frac{p}{N^3T}\log\law\big[ \dist_{[0,\termts]}(\tEM_N,\nu)<r \big]
	-
	\frac{p\termt\,(cr_k)^2}{2(p-1)}.
\end{align}
Next, apply the result of Section~\ref{s.lwbd.truncated} with $(\EM_N,\mu,\P)\mapsto (\tEM_N,\nu,\law)$.
Note that, even though $\tEM_{N}$ and $\nu$ do not have unit mass, by the scaling arguments in Section~\ref{s.basic.m=1}, the result still applies.
Doing so gives, for any $p>1$ and $k\in\Z_{>0}$,
\begin{align}
	\label{e.lwbd.tEM}
	\liminf_{r\to 0}
	\liminf_{N\to\infty} \frac{1}{N^3T} \log \P\big[ \dist_{[0,\termts]}(\tEM_N,\nu)<r \big]
	\geq
	- p \rateq(\nu) - \frac{p\termt\, (cr_k)^2}{2(p-1)}.
\end{align}

We are now ready to complete the proof.
So far we have omitted most $k$ dependence.
Restore it by writing $\nu=\nu_k$, $\tEM_N = \tEM_{N,k}$, and $N'=N'_k$.
%The results from the second last paragraph give $\dist_{[0,\termts]}(\tEM_{N,k},\nu_{k})+\dist_{[0,\termts]}(\EM_{N},\mu) \leq c\,r_k$.
By the construction of $\nu=\nu_k$ and the assumption $\rateq(\mu)<\infty$, it is not hard to check that $\rateq(\nu_k)\to\rateq(\mu)$ as $k\to\infty$.
%Note that $\lim_{N\to\infty}N'_k/N=1-2r_k$, which tends to $1$ as $k\to\infty$.
In \eqref{e.lwbd.tEM}, send $k\to\infty$ with the aid of this property and \eqref{e.lwbd.close.meausre}.
Finally sending $p\to 1$ completes the proof.

\section{applications to the moment Lyapunov exponents}
\label{s.momentL}
As was explained in Section~\ref{s.basic.m=1}, we consider $ \mm=1 $ only.
\subsection{Proof of Theorem~\ref{t.optimal}}
\label{s.momentL.optimal}
We begin with some notation.
Let $\xx_\cc$, $\xxi$, $\mm_\cc$, and the optimal deviation $\optimal=\sum_{\cc=1}^n \mm_\cc\delta_{\optimalc{\cc}}$ be as in Section~\ref{s.results.momentL}.
Recall branches from there and recall that $\Branch$ denotes the set of branches.
Let $\mm_\bb:=\sum_{\cc\in\bb}\mm_\cc$ be the total mass within a given branch $\bb$. 
By definition, branches are disjoint intervals in $\{1,\ldots,n\}$, so we can order them, and we use $<$ and $\leq$ to denote the ordering.
For a given branch $\bb$, we write $\bb\ominus 1$ and $\bb\oplus 1$ for the respective branches that precedes and succeeds $\bb$ in $\Branch$.

To facilitate the proof, let us prepare a few properties of the optimal clusters $\optimal_1,\ldots,\optimal_n$.
From their definition given after Theorem~\ref{t.optimal}, it is not hard to check that
\begin{align}
	\label{e.optimal.de}
	\dot{\optimal}_\cc =\Sgn[\optimal](\optimalc{\cc})+\diner_\bb,
	\quad
	\text{for all } \cc\in\bb
	\text{ and for all } s\in[0,\termts] \text{ except when a merge happens},
\end{align}
where $\diner_\bb$ is given in \eqref{e.optimal.frominertia}.
Recall also that those $\optimalc{\cc}$ and $\optimalc{\cc'}$ belonging to different branches do not meet within $(0,\termts)$.
Using this property to take the mass-weighted average of \eqref{e.optimal.de} over $\cc\in\bb$ gives
\begin{align}
	\label{e.optimal.bb.de}
	\dot{\optimal}_\bb=\pull_\bb+\diner_\bb,
	\qquad
	\pull_\bb := \tfrac12(\ldots-m_{\bb\ominus 1}+m_{\bb\oplus 1}+\ldots).
\end{align}
This equation shows that $\dotoptimalc{\bb}$ is a constant.
Combining \eqref{e.optimal.de}--\eqref{e.optimal.bb.de} gives $\dot{\optimal}_\cc -\Sgn[\optimal](\optimalc{\cc})=\dot{\optimal}_\bb-\pull_\bb$, for all $\cc\in\bb$.
Inserting this identity into \eqref{e.clustering.rateq} for $\nu=\optimal$ and $[s',s'']=[0,\termts]$ gives
\begin{align}
	\label{e.optimal.rateq}
	\rateq(\optimal)
	=
	\sum_{\bb\in\Branch} \termt \frac{\mm_\bb}{2}(\dotoptimalc{\bb}-\pull_\bb)^2.
\end{align}

We now begin the proof, starting with a reduction.
Set the starting and ending conditions to be $\mu_\Start:= \sum_{\cc}\mm_\cc\delta_{\xx_\cc}$ and $\mu_\End:=\delta_{\xxi}$, with $1=\mm_1+\ldots+\mm_n$.
%Using \eqref{e.mom.rep} in reverse in \eqref{e.mini} gives
By \eqref{e.minii},
\begin{align}
\label{e.mini.}
\begin{split}
	\momshe&\big( \xxi \xrightarrow{\termt} (\vecxx,\vecmm) \big)
\\
	&=
	\tfrac{\termt}{24}
	+
	\Ip{ \mu_\Start^{\otimes 2}, \tfrac{1}{4}|x-x'|} - \Ip{ \mu_\End^{\otimes 2}, \tfrac{1}{4}|x-x'|}
	-
	\inf\big\{ \rateq(\mu) : \mu(0)=\mu_\Start, \mu(\termt)=\mu_\End \big\}.
\end{split}
\end{align}
Take any $\mu\in\Csp([0,\termts],\Psp(\R))$ with $\mu(0)=\mu_\Start$ and $\mu(\termts)=\mu_\End$.
Given \eqref{e.mini.}, it suffices to prove that $\rateq(\mu)\geq\rateq(\optimal)$ and that the equality holds only if $\mu=\optimal$.
Without loss of generality, assume $\rateq(\mu)<\infty$, which implies $\partial_s\quant[\mu]\in\Lsp^2([0,\termts]\times[0,1])$.

\medskip
\noindent\textbf{Step 1: proving that $\rateq(\mu)\geq\rateq(\optimal)$.}
We begin with some notation.
First, write $\quant[\mu(s)](a)=\quant(s,a)$ and $\Sgn[\mu(s)](\quant(s,a))=\Sgn_\quant(s,a)$ to simplify notation.
Next, apply the procedure in Section~\ref{s.basic.divide} with $(\mm_1,\mm_2,\ldots)=(\mm_\bb)_{\bb\in\Branch}$ to get $(\mu_\bb)_{\bb\in\Branch}$, where the branches are ordered as described previously.
Let $\MM_\bb := \sum_{\bb'\leq\bb} \mm_{\bb'} $ be the cumulative mass.
% and write $\quant[\mu(s)](\MM_\bb):=\quant_\bb(s)$ so that each $\mu_\bb(s)$ is supported in $[\quant_{\bb\ominus 1}(s),\quant_{\bb}(s)]$. 

Let us derive a lower bound on $\rateq(\mu)$.
Refer to \eqref{e.rateq} for the definition of $\rateq(\mu)$ and divide the integral $\int_0^1\d a$ into $\int_{\MM_{\bb\ominus 1}}^{\MM_\bb} \d a$, $\bb\in\Branch$, with the convention that $\MM_{\text{(first cluster)}\ominus 1}:= 0$.
Note that this procedure is equivalent to the procedure of dividing $\mu$ in the previous paragraph.
In each of the resulting integral, multiply and divide by $\termt\mm_\bb$ and apply Jensen's inequality to get
\begin{align}
	\rateq(\mu)
	&=
	\sum_{\bb\in\Branch} \termt\frac{\mm_\bb}{2}
	\int_{0}^{\termt} \frac{\d s}{\termt} \int_{\MM_{\bb\ominus 1}}^{\MM_\bb} \frac{\d a}{\mm_{\bb}} \, \Big(\partial_s\quant - \Sgn_{\quant}\Big)^2 
\\
	\label{e.optimal.>1}
	&\geq
	\sum_{\bb\in\Branch} \termt\frac{\mm_\bb}{2}
	\Big(
		\int_{0}^{\termt} \frac{\d s}{\termt} \int_{\MM_{\bb\ominus 1}}^{\MM_\bb} \frac{\d a}{\mm_\bb} \, \partial_s\quant 
		- 
		\int_{0}^{\termt} \frac{\d s}{\termt} \int_{\MM_{\bb\ominus 1}}^{\MM_\bb} \frac{\d a}{\mm_\bb} \, \Sgn_{\quant}
	\Big)^2.
\end{align}

Next we derive an expression for the first integral in \eqref{e.optimal.>1}. 
Evaluate the integral over $[0,\termts]$ by the fundamental theorem of calculus, recognize the result as $\ip{\mu_\bb(s),x}|^{\termt}_{0}/(\termt\mm_\bb)$, and use $\mu(0)=\optimal(0)$ and $\mu(\termts)=\optimal(\termts)$.
Recall the center of mass $\optimalc{\bb}:=\sum_{\cc\in\bb} (\mm_\cc \optimalc{\cc})/\mm_\bb$.
The resulting expression is hence $\optimalc{\bb}(s)|_{0}^{\termt}/\termt$.
As shown in \eqref{e.optimal.bb.de}, $\dot{\optimal}_\bb$ is a constant, so
\begin{align}
	\label{e.optimal.1}
	\int_{0}^{\termt} \frac{\d s}{\termt} \int_{\MM_{\bb\ominus 1}}^{\MM_\bb} \frac{\d a}{\mm_\bb} \, \partial_s\quant
	=
	\dot{\optimal}_\bb
\end{align}

Next we derive an expression for the last integral in \eqref{e.optimal.>1}. 
Recall from \eqref{e.Sgn.id} that the integrand takes two different forms depending on whether $a_-=a_+$ or $a_-<a_+$, where $a_\pm=a_\pm[\mu(s)](a)$ is defined in \eqref{e.a+-}.
This dichotomy is particularly relevant at $a=\MM_{\bb\ominus 1}$ and $a=\MM_{\bb}$, which are the boundaries of the last integral in \eqref{e.optimal.>1}. 
With this in mind, set $\MM^\pm_{\bb} = \MM^\pm_{\bb}(s) := a_\pm[\mu(s)](M_b)$, 
and note that
\begin{align}
	\label{e.M.mono}
	\MM^-_{\bb} \leq \MM_{\bb} \leq \MM^+_{\bb}.
\end{align}
Divide the integral over $[\MM_{\bb\ominus 1},\MM_{\bb}]$ into integrals over $[\MM_{\bb\ominus 1}, \MM_{\bb\ominus 1}^+]$, $[\MM^+_{\bb\ominus 1}, \MM^-_{\bb}]$, and $[\MM^-_{\bb},\MM_{\bb}]$ and use~\eqref{e.Sgn.id} to evaluate the integrals:
\begin{align}
	\label{e.optimal.2}
	\int_{\MM_{\bb\ominus 1}}^{\MM_\bb} \frac{\d a}{\mm_\bb} \, \Sgn_{\quant}
	=
	\frac{1}{\mm_\bb} \Big(
		A
		+
		\int_{\MM^+_{\bb\ominus 1}}^{\MM^-_\bb} \d a \, (\tfrac12 -a)
		+
		A'
	\Big),
\end{align}
where
$
	A := \frac{1}{2}(\MM^+_{\bb\ominus}-\MM_{\bb\ominus 1})(1-\MM^+_{\bb\ominus 1}-\MM^-_{\bb\ominus 1})
$ 
and 
$
	A' = \frac{1}{2}(\MM_{\bb}-M^-_{\bb})(1-\MM^+_{\bb}-\MM^-_{\bb}).
$
Recall $\pull_\bb$ from \eqref{e.optimal.bb.de} and write it as
\begin{align}
	\label{e.optimal.3}
	\pull_\bb
	=
	\frac{1}{\mm_\bb} \int_{\MM_{\bb\ominus 1}}^{\MM_\bb} \d a \, (\tfrac12 -a)
	=
	\frac{1}{\mm_\bb} \Big(
		A''
		+
		\int_{\MM^+_{\bb\ominus 1}}^{\MM^-_\bb} \d a \, (\tfrac12 -a)
		+
		A'''
	\Big),
\end{align}
where
$
	A'' := \frac{1}{2}(\MM^+_{\bb\ominus 1}-\MM_{\bb\ominus 1})(1-\MM^+_{\bb\ominus 1}-\MM_{\bb\ominus 1})
$ 
and 
$
	A''' := \frac{1}{2}(\MM_{\bb}-\MM^-_{\bb})(1-\MM_{\bb}-\MM^-_{\bb}).
$
Subtract \eqref{e.optimal.3} from \eqref{e.optimal.2} and simplify $(A-A'')$ and $(A'-A''')$ in the result.
Doing so gives
\begin{align}
	\label{e.optimal.4}
	\int_{\MM_{\bb\ominus 1}}^{\MM_\bb} \frac{\d a}{\mm_\bb} \, \Sgn_{\quant}
	=
	\pull_\bb + \frac{1}{\mm_\bb} D_{\bb\ominus 1} - \frac{1}{\mm_\bb} D_\bb,
\end{align}
where
\begin{align}
	\label{e.D}
	D_\bb:= \tfrac{1}{2}(\MM^+_{\bb}-\MM_{\bb})(\MM_{\bb}-\MM^-_{\bb}),
\end{align}
with the convention that $ D_{\text{(first branch)}\ominus 1}:=0 $.

We are now ready to complete the proof of Step~1.
Insert \eqref{e.optimal.1} and \eqref{e.optimal.4} into \eqref{e.optimal.>1}, expand the resulting square, and compare the result with \eqref{e.optimal.rateq}.
Doing so gives 
\begin{align}
	\label{e.optimal.5.}
	\rateq(\mu)
	&\geq
	\rateq(\optimal)
	+
	\sum_{\bb\in\Branch} 
	\frac{\termt\mm_{\bb}}{2} 
	\Big( \int_0^{\termt} \frac{\d s}{\termt} \, \frac{D_{\bb}-D_{\bb\ominus 1}}{\mm_{\bb}} \Big)^2
	+
	\sum_{\bb\in\Branch}
	\big( \dot\optimal_{\bb}-\pull_{\bb}\big) 
	\int_0^{\termt} \d s \, \big(D_{\bb}-D_{\bb\ominus 1}\big).
\end{align}
Rewrite the last sum as 
$
	\sum_{\bb\in\Branch} \int_0^{\termt} \d s \, 
	( - \dot\optimal_{\bb\oplus 1}+\pull_{\bb\oplus 1}+ \dot\optimal_{\bb}-\pull_{\bb}) D_{\bb}
$,
with the convention that $\dot{\optimal}_{\text{(last branch)}\oplus 1}:=0$ and $\pull_{\text{(last branch)}\oplus 1}:=0$.
We arrive at 
\begin{subequations}
\label{e.optimal.5}
\begin{align}
\label{e.optimal.5a}
	\rateq(\mu)
	\geq
	\rateq(\optimal)
	&+
	\sum_{\bb\in\Branch} \frac{1}{2\termt\mm_{\bb}} 
	\Big( \int_0^{\termt} \d s \, \big( D_{\bb}-D_{\bb\ominus 1} \big) \Big)^2
\\
\label{e.optimal.5b}
	&+\sum_{\bb\in\Branch}
	\int_{0}^{\termt} \d s
	\big( -\dot\optimal_{\bb\oplus 1}+\dot\optimal_{\bb} +\pull_{\bb\oplus 1}-\pull_{\bb} \big) \, D_\bb.
\end{align}
\end{subequations}
By \eqref{e.optimal.bb.de}, $(-\dot\optimal_{\bb\oplus 1}-\dot\optimal_{\bb} +\pull_{\bb\oplus 1}-\pull_{\bb}) = -\diner_{\bb\oplus 1} + \diner_{\bb}$.
This quantity is strictly positive, as is readily checked from \eqref{e.optimal.frominertia}.
Also, $D_\bb$ is nonnegative by \eqref{e.M.mono} and \eqref{e.D}.
Hence $\rateq(\mu) \geq \rateq(\optimal) $.

\medskip
\noindent\textbf{Step 2: proving that $\rateq(\mu)=\rateq(\optimal)$ implies $\mu=\optimal$.}
The strategy is to extract information on $\mu$ from the condition $\rateq(\mu)=\rateq(\optimal)$.
Set $\quant_{\bb}(s):=\quant[\mu(s)](M_{\bb})$ and note that $\mu_\bb(s)$ is supported in $[\quant_{\bb\ominus 1}(s),\quant_{\bb}(s)]$.
As explained after \eqref{e.optimal.5}, $(\dot\optimal_{\bb\oplus 1}-\dot\optimal_{\bb} -\pull_{\bb\oplus 1}+\pull_{\bb})>0$; recall that $D_\bb\geq 0$ by \eqref{e.M.mono} and \eqref{e.D}.
Using these properties and the condition $\rateq(\mu)=\rateq(\optimal)$ in \eqref{e.optimal.5b} gives $D_\bb=0$ for all $\bb$.
This property together with \eqref{e.M.mono} and \eqref{e.D} forces $M^-_\bb=M^+_\bb$.
This being true for all $\bb$ implies that $\mu_\bb(s)$ has no atoms at $\quant_{\bb\ominus 1}(s)$ and $\quant_{\bb}(s)$.

We continue to extract information from the condition $\rateq(\mu)=\rateq(\optimal)$.
The condition forces the inequality in \eqref{e.optimal.>1} to be an equality, which in turn forces
$ 
	\partial_s\quant - \Sgn_{\quant}
$
to be a constant Lebesgue a.e.\ on $[0,\termts]\times[\MM_{\bb\ominus 1}, \MM_\bb]$.
By \eqref{e.optimal.1}, \eqref{e.optimal.4}, and $D_\bb=0$, the constant is
$
	\dot{\optimal}_\bb-\pull_\bb
$. 
Hence
\begin{align}
	\label{e.t.optimal.quant.de}
	\partial_s\quant - \Sgn_{\quant}= \dot{\optimal}_\bb-\pull_\bb,
	\qquad
	\text{Lebesgue a.e.\ on } [0,\termts]\times[\MM_{\bb\ominus 1}, \MM_\bb].
\end{align}
We seek to ``localize'' \eqref{e.t.optimal.quant.de} onto $\mu_\bb$.
More precisely, we seek to rewrite \eqref{e.t.optimal.quant.de} in terms of $\quant[\mu_\bb]$ and $\Sgn[\mu_\bb]$.
First, by the construction of $\mu_\bb$ in the first paragraph in Step~1, $\quant[\mu_\bb(s)](a+\MM_{\bb\ominus 1})=\quant(s,a)$ for all $a\in[0,\mm_\bb]$, so $(\partial_s\quant)|_{(s,a+\MM_{\bb\ominus 1})}=(\partial_s\quant[\mu_\bb])|_{(s,a)}$.
Next, the definition \eqref{e.Sgn} of $\Sgn[\Cdot]$ gives
\begin{align}
	\label{e.optimal.Sgnbb}
	\Sgn[\mu(s)](x) 
	= 
	\Sgn[\mu_\bb(s)](x) 
	- \frac12 \sum_{\bb'<\bb} \ip{ \mu_{\bb'}(s), \ind_{(-\infty,x)} }
	+ \frac12 \sum_{\bb'>\bb} \ip{ \mu_{\bb'}(s), \ind_{(x,\infty)} }.
\end{align}
Since $\mu_\bb(s)$ is supported in $[\quant_{\bb\ominus 1}(s), \quant_{\bb}(s)]$, we consider those $x$ in this interval only.
As mentioned in the previous paragraph, $\mu_{\bb'}(s)$ has not atoms at $\quant_{\bb'\ominus 1}(s)$ and $\quant_{\bb'}(s)$, for all $\bb'$.
Using this property in \eqref{e.optimal.Sgnbb} simplifies the last two sums together into $\pull_\bb$, defined in \eqref{e.optimal.bb.de}.
We now rewrite \eqref{e.t.optimal.quant.de} as
\begin{align}
	\label{e.evolutionary.}
	\partial_s\quant[\mu_\bb] - \Sgn[\mu_\bb](\quant[\mu_\bb]) = \dot{\optimal}_\bb,
	\quad
	\text{Lebesgue a.e.\ on } [0,\termts]\times[0, \mm_\bb].	
\end{align}
Further using \eqref{e.optimal.bb.de} gives
\begin{align}
	\label{e.optimal.evolutionary.mubb}
	\partial_s\quant[\mu_\bb] - \Sgn[\mu_\bb](\quant[\mu_\bb]) = \pull_\bb+\diner_\bb,
	\quad
	\text{Lebesgue a.e.\ on } [0,\termts]\times[0, \mm_\bb].
\end{align}

Apply the procedure in Section~\ref{s.basic.divide} to divide $\mu_\bb$ according to the masses $(\mm_{\cc})_{\cc\in\bb}$ to get $(\mu_{\cc})_{\cc\in\bb}$.

Taking any $\cc\in\bb$, let us show that $\mu_\cc(s)$ is supported at a single point.
This amounts to showing that $\quant[\mu_\bb(s)]|_{a\in U_\cc}$ is a constant, where
$
	U_\cc:=( \sum_{\cc'\in\bb, \cc'<\cc} \mm_{\cc'},  \sum_{\cc'\in\bb, \cc'\leq \cc} \mm_{\cc'})
$.
Given~\eqref{e.evolutionary.}, consider any $a_1<a_2\in U_\cc$ such that, for $j=1,2$,
\begin{align}
	\label{e.evolutionary}
	\partial_s\quant[\mu_\bb(s)](a_j) - \Sgn[\mu_\bb(s)](\quant[\mu_\bb])(a_j) = \dot{\optimal}_\bb
	\text{ holds for Lebesgue a.e.\ } s\in[0,\termts].
\end{align}
It is readily checked from \eqref{e.Sgn} and \eqref{e.cdf.quant} that $\Sgn[\lambda](x)$ is nonincreasing in $x$ and $\quant[\lambda](a)$ is nondecreasing in $a$.
Hence $\Sgn[\mu_\bb(s)](\quant[\mu_\bb])(a_1)\geq \Sgn[\mu_\bb(s)](\quant[\mu_\bb])(a_2)$.
Take the difference of \eqref{e.evolutionary} for $j=1$ and $j=2$, use the last inequality, integrate the result over $[0,s]$, and use $\quant[\mu(0)](a_1)=\quant[\mu(0)](a_2)=\xx_\cc$.
The result gives $\quant[\mu(s)](a_1)\geq \quant[\mu(s)](a_2)$.
Since $\quant[\lambda](a)$ is nondecreasing in $a$, $\quant[\mu(s)]|_{a\in[a_1,a_2]}$ is a constant.
Sending $[a_1,a_2]\to U_\cc$ shows that $\quant[\mu_\bb(s)]|_{a\in U_\cc}$ is a constant.
Hence $\mu_\cc(s)$ is supported at a single point; let $x_{\mu,\cc}(s)=x_{\mu,\cc}$ denote this point.

It remains only to show that $x_{\mu,1},\ldots,x_{\mu,n}$ coincidence with $\optimalc{1},\ldots,\optimalc{n}$.
The measure $\mu(s)$ being continuous in $s$ implies that $x_{\mu,\cc}(s)$ is too.
Hence, there exists an $s_0>0$ such that $x_{\mu,1},\ldots,x_{\mu,n}$ do not meet within $[0,s_0)$.
Within $s\in[0,s_0)$, the equation \eqref{e.optimal.evolutionary.mubb} gives $\dot{x}_{\mu,\cc} = \Sgn[\mu_\bb](x_{\mu,\cc}) + \pull_\bb + \diner_\bb$, for $\cc\in\bb$.
The first two terms on the right hand side add to $\frac12(-\mm_1-\ldots-\mm_{\cc-1}+\mm_{\cc+1}+\ldots+\mm_n)=:\pull_\cc $, so $x_{\mu,\cc}$ moves at the constant velocity $\pull_\cc+\diner_\bb$ within $[0,s_0)$.
This description of evolution matches that of the optimal clusters $\optimalc{1},\ldots,\optimalc{n}$, so $x_{\mu,\cc}|_{[0,s_0]}=\optimalc{\cc}|_{[0,s_0]}$, for all $\cc$.
By a time-continuity argument, this equality extends the first time when a merge happens in the optimal clusters.
Take this first merge time as the new starting time and run the same argument.
Continuing inductively completes the proof.

\subsection{Proof of Corollary~\ref{c.momentL}}
\label{s.momentL.moml}
Recall that we consider $\mm=1$ only.

\ref{c.momentL1}\ 
Given \eqref{e.2}, the proof amounts to showing that
\begin{align}
	\label{e.c.momentL1}
	\frac{1}{N^3T} \log \P\big[ \,\big|\tfrac{1}{NT}\X_{\ii}(T\termt)-\xxi\big| \leq \alpha, \forall \ii \big]
	+
	\inf\Big\{ \rateq(\mu) : \mu(0)=\sum_{\cc=1}^n \mm_\cc\delta_{\xx_\cc}, \mu(\termt)=\delta_{\xxi} \Big\}
\end{align}
tends to zero as $N\to\infty$ first and $\alpha\to 0$ later, where the attractive BPs start from $\X_{\ii}(0)/(NT)=\xx_{\cc_{N,\optimal}(\ii)}$ or equivalently $\EM_N(0)=\sum_{\cc=1}^n \mm_\cc \delta_{\xx_\cc}$.
By \eqref{e.wass.coupling}, the event $\{|\X_{\ii}(T\termt)/NT-\xxi|\leq r,\forall \ii\}$ implies $\dist(\EM_N(\termt),\mu_\End)\leq r$. 
Hence, by Theorem~\ref{t.ldp} and the contraction principle, the limit of \eqref{e.c.momentL1} is $\leq 0$.
To prove the reverse inequality, note that, by Theorem~\ref{t.optimal}, the infimum in \eqref{e.c.momentL1} is $\rateq(\optimal)$.
Applying Proposition~\ref{p.lwbd.clustering} with $\xi=\optimal$ gives the reversed inequality $\geq 0$.

\medskip
\ref{c.momentL2}\ Let us set up the notation and goal.
Fix a nonempty $A\subset\{1,\ldots,n\}$, recall $\cc_{N,\optimal}$ from \eqref{e.assignment}, and let $\group_{N,\optimal}(A)=\{ \ii: \cc_{N,\optimal}(\ii) \in A \}$.
Since $A$ is nonempty, $ |\group_{N,\optimal}(A)|/N \geq \min\{\mm_1,\ldots,\mm_n\}:=\mm_*>0 $.
Recall $\Dist_{N,[s',s'']}$ from \eqref{e.Dist}--\eqref{e.Dist[]}, and, for $\beta>0$, consider the event $\calD_N:=\{\Dist_{N,[0,\termts]}(\X_\ii,\optimal)> \beta, \forall \ii\in\group_N(A)\}$.
Similar to \ref{c.momentL1}, proving \ref{c.momentL2} amounts to proving that
\begin{align}
	\label{e.c.momentL2}
	\limsup_{\alpha\to 0} \limsup_{N\to\infty}
	\frac{1}{N^3T} \log \P\big[ \big\{\big|\tfrac{1}{NT}\X_{\ii}(T\termt)-\xxi\big| \leq \alpha, \forall \ii\big\}\cap\calD_N \big]
	<
	-\rateq(\optimal).
\end{align}

The strategy of proving \eqref{e.c.momentL2} is to derive a lower bound on $\dist_{[0,\termts]}(\EM_N,\optimal)$ from the condition imposed by $\calD_N$.
First, it is not hard to check from \eqref{e.dist} that, for a large enough $k_0=k_0(\optimal)$,
\begin{align}
	\label{e.c.momentL2.}
	\dist_{[0,\termt]}(\EM_N,\optimal)
	\geq
	\dist(\EM_N(s),\optimal(s))
	\geq
	 2^{-k_0}\min\Big\{1, \frac{1}{N}\sum_{\ii=1}^{N} \Dist_{N,s}\big(\X_{\ii},\optimal\big) \Big\}.
\end{align}
Under $\calD_N$, for every $\ii\in\group_{N,\optimal}(A)$, there exists a random $s(\ii)\in[0,\termts]$ such that $\Dist_{N,s(\ii)}(\X_{\ii},\optimal)>\beta$.
If these $s(\ii)$s happen to be all the same, by \eqref{e.c.momentL2.}, $\dist_{[0,\termt]}(\EM_N,\optimal) \geq 2^{-k_0} \mm_* \min\{1,\beta\}$.
In general, those $s(\ii)$s are different, so we need some time-continuity estimates.
Recall the event $\calU_N(v)$ from \eqref{e.event.bmcontrol}.
Use the bounds derived after \eqref{e.event.bmcontrol} to fix a $v$ so that 
\begin{align}
	\label{e.Ugood}	
	\limsup_{N\to\infty} \frac{1}{N^3T}\log\P\big[\,\calU_N(v)^\comple\big] < - \rateq(\optimal),
\end{align}
and write $\calU_N(v)=\calU_N$ hereafter.
We claim that, there exists a $c_1=c_1(v,\optimal)<\infty$ such that
\begin{align}
	\label{e.fraction}
	\text{under } \calD_N\cap\calU_N,
	\quad
	\sup_{s\in[0,\termts]}\big\{ \big| \{ \ii\in\group_{N,\optimal}(A) : \Dist_{N,s}(\X_{\ii},\optimal)>\beta/2 \} \big|  \big\} 
	\geq 
	N\mm_*/c_1.
\end{align}
Namely, under $\calU_N$, a $1/c_1$ fraction of $\ii\in\group_{N,\optimal}(A)$ simultaneously satisfy $\Dist_{N,s}(\X_{\ii},\optimal)>\beta/2$ at least once within $s\in[0,\termts]$.
To see why, divide $[0,\termts]$ into $\ell$ equally spaced subintervals.
Within each subinterval, apply the continuity estimate \eqref{e.aBP.control} and use the continuity of $\optimal_{\cc}(s)$, $\cc=1,\ldots,n$.
By choosing $\ell$ large enough (depending only on $v$, $\optimal$, and $\beta$), we ensure that, 
under $\calU_N$, the bound
\begin{align}
	\label{e.thebound}
	\max_{k=1,\ldots,\ell}
	\sup_{s,s'\in[\frac{(k-1)\termt}{\ell},\frac{k\termt}{\ell}]}
	\Big|
		\big| \tfrac{1}{NT}\X_{\ii}(Ts)-\optimal_{\cc_{N,\optimal}(\ii)}(s) \big| 
		- 
		\big| \tfrac{1}{NT}\X_{\ii}(Ts')-\optimal_{\cc_{N,\optimal}(\ii)}(s') \big| 
	\Big|
	\leq
	\frac{\beta}{2}
\end{align}
holds for at least $N(1-\mm_{*}/2)$ many $\ii\in\{1,\ldots,N\}$.
Since $|\group_{N,\optimal}(A)|\geq N\mm_{*}$, the bound~\eqref{e.thebound} holds for at least $N\mm_{*}/2$ many $\ii\in\group_{N,\optimal}(A)$.
The claim \eqref{e.fraction} hence follows with $c_1=2\ell$.
Set $c_2:=2^{-k_0}\min\{1,\beta\mm_*/(2c_1)\}>0$.
Combining \eqref{e.c.momentL2.}--\eqref{e.fraction} gives
\begin{align}
	\label{e.fraction.}
	\calD_N\cap\calU_N \subset \big\{ \dist_{[0,\termt]}(\EM_N,\optimal) \geq c_2  \big\}.
\end{align}

Let us complete the proof of \eqref{e.c.momentL2}.
First, by \eqref{e.fraction.} and Theorem~\ref{t.ldp},
\begin{align}
\label{e.c.momentL2..}
\begin{split}
	\limsup_{\alpha\to 0}&\limsup_{N\to\infty}
	\frac{1}{N^3T} \log \P\big[ \big\{\big|\tfrac{1}{NT}\X_{\ii}(T\termt)-\xxi\big| \leq \alpha, \forall \ii\big\}\cap\calD_N\cap\calU_N \big]
\\
	&\leq
	\inf\Big\{ \rateq(\mu): \dist_{[0,\termts]}(\mu,\optimal) \geq c_2 : \mu(0)=\sum_\cc\mm_\cc\delta_{\xx_\cc}, \mu(\termts)=\mm\delta_{\xxi} \Big\}.
\end{split}
\end{align}
By Theorem~\ref{t.optimal}, the right hand side of \eqref{e.c.momentL2..} is strictly smaller than $-\rateq(\optimal)$.
Given this and \eqref{e.Ugood}, the desired result~\eqref{e.c.momentL2} now follows.

\section{Connection to $\ratekpz$ and the limit shape}
\label{s.matching}
Here, we establish a few properties of $\ratekpz$ and the limit shape.
In doing so, we prove Theorem~\ref{t.matching}: Part~\ref{t.matching.legendre} is proven in Sections~\ref{s.matching.convex}, \ref{s.matching.homeo}, and \ref{s.matching.legendre}; Parts~\ref{t.matching.shape}--\ref{t.matching.intermediate} are proven in Section~\ref{s.matching.shape}.

\subsection{Basic properties of $\ratekpz$}
\label{s.matching.convex}
Fix $t$ and $(\xx_1<\ldots<\xx_n)$ and writing $\ratekpz(t,\vecxx,\vechv)=\ratekpz(\vechv)$, $\hf{,t,\vecxx,\vechv}=\hf{,\vechv}$, $\hvspace(t,\vecxx)=\hvspace$, and $\hvspacec(t,\vecxx)=\hvspacec$ to simplify notation. 

Let us show that $\ratekpz:\hvspace\to[0,\infty)$ is strictly convex, which implies that $\ratekpz|_{\hvspacec} $ is strictly convex.
The key is to recognize $\hf{,\vechv}$ as the minimizer of a variational problem.
Recall that $\parab(t,x)=-x^2/(2t)$.
For any piecewise $\Csp^1$ function $\f:\R\to\R$ such that $\f \geq \parab(t)$ and $\f(x)=\parab(t,x)$ for all large enough $|x|$, consider $\ratebm(\f) := \int_{\R} \d x \, ( \frac12( \partial_x \f)^2 - \frac12(\partial_x\parab(t))^2) $.
Note that the integral is well-defined and finite because $\f(x)=\parab(t,x)$ for all large enough $|x|$.
It is not hard to check that
\begin{align}
	\label{e.matching.convex.var}
	\ratekpz(\vechv)
	:=
	\ratebm(\hf{,\vechv})
	=
	\inf\big\{ \ratebm(\f) \,:\, \f (\xx_\cc) = \hv_\cc, \, \cc=1,\ldots, n \big\}.
\end{align}
Namely, $\hf{,\vechv}$ is the minimizer of the infimum in \eqref{e.matching.convex.var}.
For any $\gamma\in(0,1)$, use the convexity of $u\mapsto u^2/2$ to write
\begin{align}
	\label{e.matching.convex}
	\tfrac12 \big(\gamma \partial_x\hf{,\vechv_1} +(1-\gamma)\partial_x\hf{,\vechv_2} \big)^2
	\leq
	\gamma \tfrac12 \big(\partial_x\hf{,\vechv_1} \big)^2	
	+
	(1-\gamma) \tfrac12 \big(\partial_x\hf{,\vechv_2} \big)^2.
\end{align}
Subtract $(\partial_x(-\parab(t)))^2/2$ from both sides and integrate both sides over $x\in\R$.
The right hand side is $\gamma\ratekpz(\vechv_1)+(1-\gamma)\ratekpz(\vechv_2)$.
The left hand side is at least $\ratekpz(\gamma\vechv_1+(1-\gamma)(\vechv_2))$ by the variational characterization~\eqref{e.matching.convex.var}.
This proves the convexity.
To prove the strictness, note that the equality in \eqref{e.matching.convex} holds only if $\partial_x\hf{,\vechv_1}=\partial_x\hf{,\vechv_2}$.
This being true for Lebesgue a.e.\ $x\in\R$ forces $\vechv_1=\vechv_2$.

Next, we show that, for every $\vechv\in\hvspace$,
\begin{align}
	\label{e.ratekpz.gradient}
	\partial_{\cc}\ratekpz(\vechv)
	:=
	\partial_{\hv_\cc}\ratekpz(t,\vecxx,\vechv) 
	= 
	(\partial_x \hf{,\vechv})(\xx_\cc^-) - (\partial_x \hf{,\vechv})(\xx_\cc^+),
	\quad
	\text{for all } \cc=1,\ldots,n.
\end{align}
Let $L_{\cc,\cc+1}=L_{\cc,\cc+1}(x)$ denote the linear function such that $L_{\cc,\cc+1}(\xx_{\cc})=\hv_{\cc}$ and that $L_{\cc,\cc+1}(\xx_{\cc+1})=\hv_{\cc+1}$, and let $\xx_{\cc,-}:=\inf\{x\geq \xx_{\cc-1} : \partial_x\hf{,\vechv}(x)=\partial_x\hf{,\vechv}(\xx_\cc^-)\}$, with the convention $\xx_{0}:=-\infty$.
In words, $\xx_{\cc,-}$ is $\xx_{\cc-1}$ when $L_{\cc,\cc+1}|_{[\xx_{\cc},\xx_{\cc+1}]}\geq\parab(t)|_{[\xx_{\cc},\xx_{\cc+1}]}$, otherwise $\xx_{\cc,-}$ is the ``tangent point'' to the left of $\xx_\cc$; the tangent points are those labeled by triangles in Figure~\ref{f.shape.termt}.
Define $\xx_{\cc,+}$ similarly.
To prove \eqref{e.ratekpz.gradient}, it suffices to consider those $\vechv\in\hvspace$ such that
\begin{subequations}
\label{e.ratekpz.gradient.assum}
\begin{align}
	&\text{ either } L_{\cc-1,\cc}|_{[\xx_{\cc-1},\xx_{\cc}]}>\parab(t)|_{[\xx_{\cc-1},\xx_{\cc}]} \text{ or } \xx_{\cc-1,+}<\xx_{\cc,-} \text{ are both tangent points},
\\
	&\text{ either } L_{\cc,\cc+1}|_{[\xx_{\cc},\xx_{\cc+1}]}>\parab(t)|_{[\xx_{\cc},\xx_{\cc+1}]} \text{ or }
	\xx_{\cc,+}<\xx_{\cc+1,-} \text{ are both tangent points.}
\end{align}
\end{subequations}
Indeed, for fixed $\hv_1,\ldots,\hv_{\cc-1},\hv_{\cc+1},\ldots,\hv_{n}$, there exist at most two values of $\hv_{\cc}$ for which \eqref{e.ratekpz.gradient.assum} fails.
Also, as is readily checked, the right hand side \eqref{e.ratekpz.gradient} is continuous on $\hvspace$.
Hence, once \eqref{e.ratekpz.gradient} is proven for those $\vechv\in\hvspace$ satisfying \eqref{e.ratekpz.gradient.assum}, the result extends to all $\vechv\in\hvspace$.

Let us prove \eqref{e.ratekpz.gradient} under the assumption \eqref{e.ratekpz.gradient.assum}.
Under this assumption,
perturbing $\hv_\cc$ changes $\hf{,\vechv}$ only within $x\in[\xx_{\cc,-},\xx_\cc]\cup[\xx_\cc,\xx_{\cc,+}]$, so
\begin{align}
\label{s.matching.convex1}
\begin{split}
	\partial_{\cc}\ratekpz(\vechv) 
	=
	&\partial_{\hv_\cc} \int_{\xx_{\cc,-}}^{\xx_\cc} \d x \, \big( \tfrac12( \partial_x \hf{,\vechv})^2 - \tfrac12(\partial_x\parab(t))^2 \big)
\\
	&+
	\partial_{\hv_\cc} \int_{\xx_\cc}^{\xx_{\cc,+}} \d x\, \, \big( \tfrac12( \partial_x \hf{,\vechv})^2 - \tfrac12(\partial_x\parab(t))^2 \big).
\end{split}
\end{align}
In the integrals in \eqref{s.matching.convex}, only $\xx_{\cc,-},\xx_{\cc,+},\hf{,\vechv}$ may depend on $\hv_{\cc}$.
Under \eqref{e.ratekpz.gradient.assum}, either $\partial\xx_{\cc,-}/\partial\hv_{\cc}=0$ or $(\partial_x \hf{,\vechv})(\xx_{\cc,-})=(\partial_x\parab)(t,\xx_{\cc,-})$.
Hence, in \eqref{s.matching.convex1}, the contribution of differentiating $\xx_{\cc,-}$ is zero.
The same holds for $\xx_{\cc,+}$.
What remains is the contribution of differentiating the integrands:
\begin{align}
	\partial_{\cc}\ratekpz(\vechv) 
	=
	\int_{[\xx_{\cc,-},\xx_\cc]} \d x \, \partial_x \hf{,\vechv} \ \frac{1}{\xx_\cc-\xx_{\cc,-}}
	+
	\int_{[\xx_\cc,\xx_{\cc,+}]} \d x \, \partial_x \hf{,\vechv} \Big( \frac{-1}{\xx_{\cc,+}-\xx_\cc} \Big).
\end{align}
Note that the integrands are constant, so the last expression evaluates to the right hand side of \eqref{e.ratekpz.gradient}.

\subsection{The map $\nabla\ratekpz:\hvspacec\to[0,\infty)^n$ is a homeomorphism}
\label{s.matching.homeo}
Notation as in Section~\ref{s.matching.convex}.

First, by the strict convexity from Section~\ref{s.matching.convex}, $\nabla\ratekpz$ is injective, and from~\eqref{e.ratekpz.gradient}, it is not hard to check that $\nabla\ratekpz$ is continuous. 

To prepare for the rest of the proof, for any $\vecmm\in[0,\infty)^n$, we consider $S(\vecmm):=\{\vechv\in\hvspacec: \partial_\cc \ratekpz(\vechv)\leq \mm_\cc, \cc=1,\ldots,n\}$ and establish a few properties of it.
First, the set $S(\vecmm)$ is nonempty because it contains $(-\xx_\cc^2/(2t))_{\cc=1}^n$.
Next, we claim that, for any compact $K\subset\R^n$,
$
	S(K)
	:=
	 \cup_{\vecmm\in K } S(\vecmm) 
$
is compact.
It is not hard to check that $S(K)$ is closed.
By definition, every $\vechv\in S(K)$ satisfies $\hv_1 \geq -\xx^2_1/(2t),\ldots,\hv_{n} \geq -\xx^2_{n}/(2t)$, so we only need upper bounds on $\hv_1,\ldots,\hv_n$.
Sum the formula \eqref{e.ratekpz.gradient} over $\cc,\ldots,n$ to get
\begin{align}
	(\partial_{x} \hf{,\vechv})(\xx^-_{\cc}) - (\partial_{x} \hf{,\vechv})(\xx^+_{n}) 
	= 
	\partial_{\cc}\ratekpz(\vechv) + \ldots + \partial_{n}\ratekpz(\vechv)
	\leq
	\mm_\cc+\ldots+\mm_n.
\end{align}
Recall $\xx_{\cc,\pm}$ from after \eqref{e.ratekpz.gradient}, and note that $(\partial_{x}\hf{,\vechv})(\xx^+_{n})=(\partial_x\parab)(t,\xx_{n,+})\leq(\partial_x\parab)(t,\xx_n)$.
Hence $\partial_{x} (\hf{,\vechv})(\xx^-_{\cc}) $ is bounded from above, and the bound can be chosen uniformly over $\vecmm\in K$.
Using this property inductively for $\cc=1,\ldots,n$ shows that $\hv_1,\ldots,\hv_n$ are bounded from above, uniformly over $\vecmm\in K$.

Let us prove that $\nabla\ratekpz$ is surjective.
Let $\pi_\cc:\R^n\to\R$ be the projection onto the $\cc$th coordinate, take any $\vecmm\in[0,\infty)^n$, consider $\hv_{*,\cc'}:=\sup\pi_{\cc'}(S(\vecmm))$, and set $\vechv_*:=(\hv_{*,\cc'})_{\cc'=1}^n$.
Fix any $\cc$.
By the construction of $\vechv_*$ and the compactness of $S(\vecmm)=S(\{\vecmm\})$, there exists a convergent sequence $\vechv_{(k)}\to\vechv_{(\infty)}$ in $S(\vecmm)$ such that $\hv_{(\infty),\cc} = \hv_{*,\cc}$ and $\hv_{(\infty),\cc'}\leq \hv_{*,\cc'}$ for all $\cc'$.
These properties together with the property that $\hf{,\vechv_{(\infty)}}$ is concave gives that 
\begin{align}
	(\partial_x \hf{,\vechv_*})(\xx_\cc^-) - (\partial_x \hf{,\vechv_*})(\xx_\cc^+)
	\leq
	(\partial_x \hf{,\vechv_{(\infty)}})(\xx_\cc^-) - (\partial_x \hf{,\vechv_{(\infty)}})(\xx_\cc^+).
\end{align}
The right hand side is at most $\mm_\cc$ because $\vechv_{(\infty)}\in S(\vecmm)$.
We arrive at the inequality
\begin{align}
	\label{e.matching.homeo1}
	\partial_\cc\ratekpz(\vechv_*)
	=
	(\partial_x \hf{,\vechv_*})(\xx_\cc^-) - (\partial_x \hf{,\vechv_*})(\xx_\cc^+)
	\leq
	\mm_\cc.
\end{align}
Next, we have the following property that can be checked by referring to Figure~\ref{f.shape.termt}:
For any $\vechv\in\hvspacec$ (and in fact for any $\vechv\in\hvspace$), when $\hv_\cc$ increases while other components remain fixed, the quantity $ \partial_\cc\ratekpz = (\partial_x \hf{,\vechv})(\xx_\cc^-) - (\partial_x \hf{,\vechv})(\xx_\cc^+) $ increases while other ($\partial_{\cc'}\ratekpz$)s decrease or stay the same.
This property forces the inequality in \eqref{e.matching.homeo1} to be an equality: Otherwise, by increasing $\hv_{\cc}$, we can increase  $\partial_{\cc}\ratekpz$ while maintaining the inequality $\partial_{\cc'}\ratekpz\leq\mm_{\cc'}$ for all $\cc'$, which contradicts the construction of $\vechv_*$.
Since $\cc$ was arbitrary, $ \partial_\cc\ratekpz(\vechv_*)=\mm_\cc$, for $\cc=1,\ldots,n$, and this gives the desired surjectivity.

Finally, we prove that $(\nabla\ratekpz)^{-1}$ is continuous.
The previous paragraph gives $(\nabla\ratekpz)(S(\vecmm))\ni\vecmm$.
Take any $b<\infty$ and consider $S':=S([0,b]^n)$. 
We have $(\nabla\ratekpz)(S')\supset [0,b]^n $.
Since $S'$ is compact, the continuity of $\nabla\ratekpz$ implies the continuity of $(\nabla\ratekpz)^{-1}$ on $[0,b]^n$.

\subsection{Limit shape and its shocks}
\label{s.matching.shape}
Fix $\vechv\in\hvspacec=\hvspacec(\termt,\vecxx)^\circ$.
%The limit shape $\mps$ is the entropy solution of the backward equation \eqref{e.iburgers.back} with the terminal condition $\mps(\termt,\Cdot)=\hf{}=\hf{,\termt,\vecxx,\vechv}$.
Hereafter, we will often drop the $\termt,\vecxx,\vechv$ dependence to simplify notation.

We begin by recalling some PDE background related to $\mps$; we refer to \cite[Ch.~3]{evans2022} for an introduction on this topic.
Consider Burgers' equation and its Hamilton--Jacobi equation
\begin{align}
	\label{e.iburgers}
	\partial_t \u = \tfrac12 \partial_x( \u^2),
	\qquad
	\partial_t \h = \tfrac12 (\partial_x \h)^2,
	\qquad
	(t,x)\in(0,\termts]\times\R.
\end{align}
The two equations are related by $\partial_x \h = \u$.
These equations can have multiple weak solutions under a given initial condition, but have a unique entropy solution given by the Hopf--Lax operator:
\begin{align}
	\label{e.hl.fw}
	\h(t) = \HL_{t}\big(\h(0)\big),
	\qquad
	\HL_{t}(\f)(x) := \sup\Big\{ - \frac{(x-y)^2}{2t} + \f(y) : y\in\R \Big\}. 
\end{align}
Now, \emph{time reverse} what is described above and consider the backward version of \eqref{e.iburgers}--\eqref{e.hl.fw}:
\begin{align}
	\label{e.iburgers.back}
	&\partial_s\h(\termt-s,x) = -\tfrac12 (\partial_x \h(\termt-s,x))^2,
	\qquad
	\partial_s\u(\termt-s,x) = -\tfrac12 \partial_x(\u(\termt-s,x))^2,
\\
	\label{e.hl}
	&(\HLop_s(\f))(x) := \inf_{y\in\R} \Big\{ \frac{(x-y)^2}{2s} + \f(y) \Big\}.
\end{align}
Under such notation, the limit shape $\mps$, defined in \eqref{e.mps}, is obtained by taking $\hf{}=\hf{,\termt,\vecxx,\vechv}$ as the terminal condition and evolving it backward by $\HLop$. 
As is readily checked, a weak solution of \eqref{e.iburgers.back} is also a weak solution \eqref{e.iburgers}.
On the other hand, an entropy solution of \eqref{e.iburgers.back} (which we call backward entropy) is in general \emph{not} an entropy solution of \eqref{e.iburgers} (which we call forward entropy).
Hence $\mps$ is a weak solution of \eqref{e.iburgers} and \eqref{e.iburgers.back}, is backward entropy, but is not forward entropy.
Put $\mpu:=\partial_x\mps$.
The \textbf{characteristics} are linear trajectories in spacetime along which $\mpu$ is constant.
The \textbf{shocks} are trajectories in spacetime across which $\mpu$ is not continuous.

Let us give a geometric description of $\mps$.
Write $\mps(\termt)=\hf{}$ as an infimum of lines:
\begin{subequations}
\label{e.hf.linear}
\begin{align}
\label{e.hf.linear.1}
	\hf{}
	=
	\min\big\{	&\min\big\{ 
		L^\bb_{\triangleleft}, L^\bb_{\cc,\cc+1}, L^\bb_{\triangleright} : \bb\in\Branchs
	\big\},
\\
\label{e.hf.linear.2}
	&\inf\big\{ L : L \text{ is a tangent line of } \parab(\termt) \text{ at places where } \parab(\termt)=\hf{}
	\big\}
	\big\}.
\end{align}
\end{subequations}
Here, $\Branchs$ is a partition of $\{1,\ldots,n\}$ into intervals, where each $\bb\in\Branchs$ is a maximal set of $\cc$s such that $(\hf{}-\parab(\termt))|_{x\in[-\xx_\cc,\xx_{\cc'}]}>0$ for all $\cc,\cc'\in\bb$; we will show later that $\Branchs=\Branch$.
Let $\tangent^\bb_{\triangleleft}$ and $\tangent^\bb_{\triangleright}$ be the left and right tangent points as depicted in Figure~\ref{f.shape.termt.lines.1}.
The minimum in \eqref{e.hf.linear.1} runs over all $\bb\in\Branchs$ and $\cc,\cc+1\in\bb$, and accounts for the piecewise linear part of $\hf{}$, as depicted in Figure~\ref{f.shape.termt.lines.1}.
The infimum in \eqref{e.hf.linear.2} accounts for the parabolic part of $\hf{}$, as depicted in Figure~\ref{f.shape.termt.lines.2}.
Next, apply $\HLop_{s}$ to both sides of \eqref{e.hf.linear} and exchange the infimum over $y$ (in \eqref{e.hl}) with other minimums and infimums to get
\begin{subequations}
\label{e.hf.linear.}
\begin{align}
\label{e.hf.linear..1}
	\mps(\termt-s)
	=
	\min&\big\{	\min\big\{ 
		\HLop_{s}(L^\bb_{\triangleleft}), \HLop_{s}(L^\bb_{\cc,\cc+1}), \HLop_{s}(L^\bb_{\triangleright}) : \bb\in\Branchs
	\big\},
\\
\label{e.hf.linear..2}
	&\inf\big\{ \HLop_{s}(L) : L \text{ is a tangent line of } \parab(\termt) \text{ at places where } \parab(\termt)=\hf{}
	\big\}
	\big\}.
\end{align}
\end{subequations}
For a linear function $L(x)=v_1x+v_2$, it is readily checked that $\HLop_s(L)=L-(s/2)(\partial_xL)^2$.
It is also readily checked that $\HLop_s(\parab(\termt))=\parab(\termt-s)$.
Hence, for $s\in(0,\termt)$, the limit shape $\mps(\termt-s)$ is obtained by vertically shifting the lines in Figures~\ref{f.shape.termt.lines.1}--\ref{f.shape.termt.lines.2} by $-(s/2)\cdot(\text{slope})^2$ and taking the infimum of the result.
Note that every  shifted line stays above $\parab(\termt-s)$, since $\HLop_s$ preserves orders: Namely, $f_1\geq f_2$ implies $\HLop_s(f_{1}) \geq \HLop_s(f_2)$.
Within each $\bb\in\Branchs$, the leftmost and rightmost lines (those indexed by $\triangleleft$ and $\triangleright$) touch the parabola $\parab(\termt-s)$ at tangent.
At the tangent points, $\partial_{x}\parab(\termt-s)=-x/(\termt-s)$ is equal to the slopes of those lines, so the $x$ coordinates of the tangent points are $(\termt-s)\tangent^\bb_{\triangleleft}/\termt$ and $(\termt-s)\tangent^\bb_{\triangleright}/\termt$, which trace out the dashed lines in Figure~\ref{f.shocks}.

\begin{figure}
\begin{minipage}{.51\linewidth}
\boxed{\includegraphics[width=\linewidth]{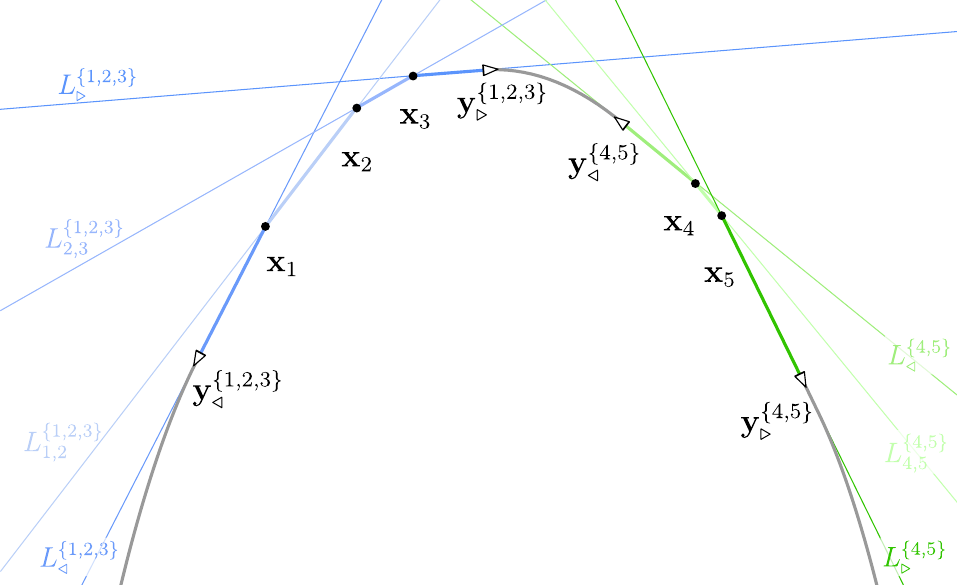}}
\caption{$\hf{}$ as an infimum of lines, the piecewise linear part.
In this Figure, $\Branchs=\{\{1,2,3\},\{4,5\}\}$.
}
\label{f.shape.termt.lines.1}
\end{minipage}
\hfill
\begin{minipage}{.47\linewidth}
\boxed{\includegraphics[width=\linewidth]{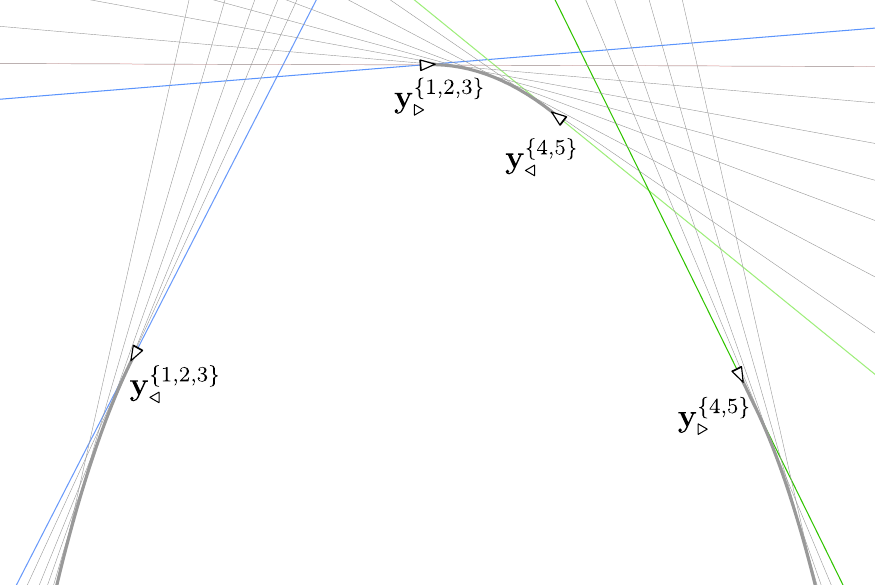}}
\caption{$\hf{}$ as an infimum of lines, the parabolic part.
}
\label{f.shape.termt.lines.2}
\end{minipage}
\end{figure}

\begin{figure}
\fbox{\includegraphics[width=.5\linewidth]{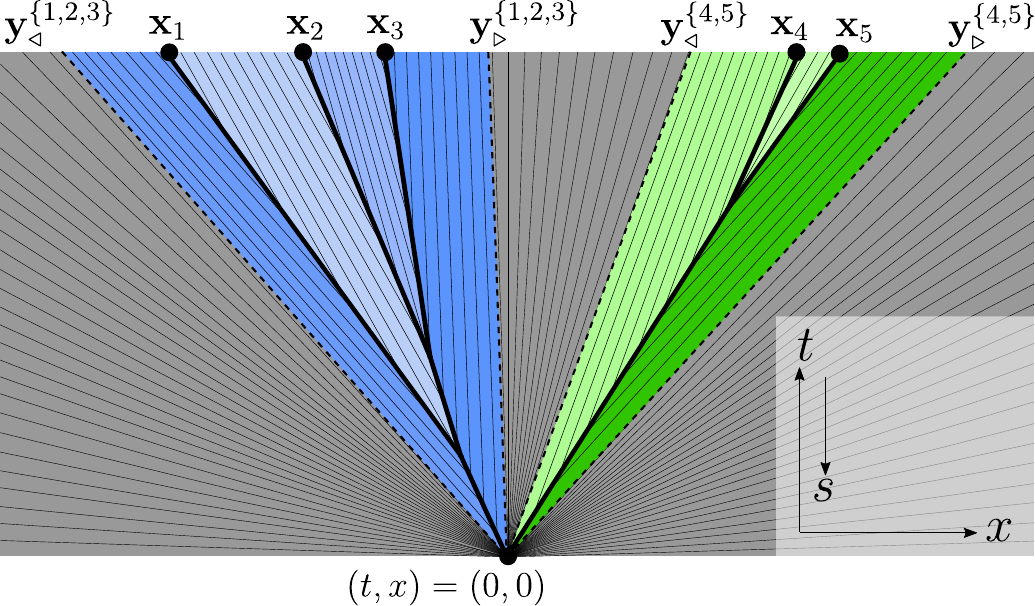}}
\caption{The shocks (thick solid lines), characteristics (thin solid lines), and tangent points (dashed lines).}
\label{f.shocks}
\end{figure}

Based on the preceding description of $\mps$, we infer some properties of $\mps$ and its shocks.
Let $ \cone_\bb:=\{(t,x):t\in[0,\termts], x\in[(t/\termts)\tangent^{\bb}_{\triangleleft},(t/\termts)\tangent^{\bb}_{\triangleright}]\}$ be the spacetime region bounded by the tangent points.
They are the colored regions in Figure~\ref{f.shocks}.
\begin{enumerate}[leftmargin=20pt, label=(\Roman*)]
\item \label{eu.shock.outside}
Outside $\cup_{\bb\in\Branchs}\cone_\bb$, $\mps(t,x)=\parab(t,x)=-x^2/(2t)$, so the characteristics are straight lines that connect $(t,x)=(0,0)$ and $\{\termt\}\times\R$; see Figure~\ref{f.shocks}.
\item \label{eu.shock.constant}
Within each $\cone_\bb$, $\mpu$ is piecewise constant, with values given by the slopes of $L^\bb_{\triangleleft}$, $L^\bb_{\triangleright}$, and a subset of $\{L^\bb_{\cc,\cc+1}: \cc,\cc+1\in\bb\}$; the jumps of $\mpu$ occur exactly along the shocks; see Figure~\ref{f.shocks}.
\end{enumerate}
Let $\shock{\cc}=\shock{\cc}(s)$ denote the shocks, parameter in the backward time $s$.
Let $\intermt,\xxi_{\aa}$ be as in the notation of Theorem~\ref{t.matching}\ref{t.matching.intermediate.}.
By the Rankine--Hugoniot relation (see \cite[Ch.~3]{evans2022} for example)
\begin{align}
	\label{e.rankinehugoniot}
	\dotshock{\aa}|_{s=\termt-\intermt}
	:=
	(\tfrac{\d~}{\d s} \shock{\aa})(\termt-\intermt)
	=
	\tfrac12 ( \mpuu{\aa}^- + \mpuu{\aa}^+),
	\quad
	\text{where }
	\mpuu{\aa}^\pm:= \mpu(\intermt,\shock{\aa}(\termt-\intermt)^\pm).
\end{align}
This relation and Properties~\ref{eu.shock.outside}--\ref{eu.shock.constant} together give the following property.
\begin{enumerate}[leftmargin=20pt, label=(\Roman*)]
\setcounter{enumi}{2}
\item \label{eu.shock.linear}
For each $\bb\in\Branchs$, those shocks $\shock{\cc}$ with $\cc\in\bb$ stay within $\cone_\bb$ and travel at constant velocities except when they meet. 
Further, shocks $\shock{\cc}$ and $\shock{\cc'}$ merge within $s\in(0,\termts)$ if and only if they belong to the same $\bb\in\Branchs$, so $\Branchs=\Branch$.
\end{enumerate}
By Property~\ref{eu.shock.constant}, $\mps$ is linear in a neighborhood on either side of a shock and hence solves \eqref{e.iburgers} \emph{classically} there.
Using this property gives the following.
\begin{enumerate}[leftmargin=20pt, label=(\Roman*)]
\setcounter{enumi}{3}
\item \label{eu.shock.diff}
Let $\intermt,\xxi_{\aa}$ be as in the notation of Theorem~\ref{t.matching}\ref{t.matching.intermediate.}.
For all $\intermt\in(0,\termts)$ except when shocks merge, 
$
	\frac{\d~}{\d t}\mps(t,\shock{\aa}(\termt-t))|_{t=\intermt}
	=
	\frac12(\mpuu{\aa}^-)^2 -\mpuu{\aa}^-\,\dotshock{\aa}
$,
where $\mpuu{\aa}^-:= \mpu(t,\shock{\aa}(\termt-\intermt)^-)$.
\end{enumerate}

Let us prove Theorem~\ref{t.matching}\ref{t.matching.intermediate}.
To prove \eqref{e.legendre.dual.intermediate}, recall that $\nabla_{\vechv}\ratekpz(\termt,\vecxx,\vechv)=\vecmm$ and combine this relation with the formula \eqref{e.ratekpz.gradient} for $t=\termt$ to get $\mpuu{\cc}^--\mpuu{\cc}^+ = \mm_\cc$.
Next, use the formula \eqref{e.ratekpz.gradient} for $t=\intermt$ to get $\nabla_{\vechv}\ratekpz(\intermt,\vecxxi,\vechvi) = \mpuu{\aa}^--\mpuu{\aa}^+$,
telescope the right hand side as $\sum_{\cc\in\Aranch(\aa)}(\mpuu{\cc}^--\mpuu{\cc}^+)$, use $\mpuu{\cc}^--\mpuu{\cc}^+ = \mm_\cc$, and recognize the resulting sum as $\mmi_\aa$.
Doing so concludes \eqref{e.legendre.dual.intermediate}.
To prove \eqref{e.treesg}, let $\optimal$ be the optimal deviation for \eqref{e.mini} with $\xxi=0$.
From the definition of the optimal deviation (in Section~\ref{s.results.momentL}), it is readily checked that $\optimal(\Cdot+(\termt-\intermt))|_{[0,\intermt]}$ is the optimal deviation for \eqref{e.mini} with $[0,\termt]\mapsto[0,\intermt]$, $\xxi=0$, $(\vecxx,\vecmm)\mapsto(\vecmmi,\vecxxi)$.
Similarly, for each $\aa=1,\ldots,n'$, the deviation $\sum_{\cc\in\Aranch(\aa)}\mm_\cc\delta_{\optimal_\cc}|_{[0,\termt-\intermt]}$ is the optimal deviation for \eqref{e.mini} with $[0,\termt]\mapsto[0,\termt-\intermt]$, $\xxi=\xxi_\aa$, $(\vecxx,\vecmm)\mapsto(\xx_\cc,\mm_\cc)_{\cc\in\Aranch(\aa)}$.
These properties together with Theorem~\ref{t.optimal} give \eqref{e.treesg}.

Let us prove Theorem~\ref{t.matching}\ref{t.matching.shape}.
With $\vecxx,\vechv$ having been fixed, set $\vecmm:=(\nabla\ratekpz)(\termt,\vecxx,\vechv)$.
For this $\vecmm$, let $(\optimalc{\cc})_{\cc=1}^n$ be the optimal clusters.
Take any intermediate time $\intermt\in(0,\termts]$ and let $\aa\in\{1,\ldots,n'\}$, $\Aranch(\aa)$, $\mmi{}_{\aa}$ be as in Notation for Theorem~\ref{t.matching}\ref{t.matching.intermediate.}, and recall $\diner_\bb$ from before Theorem~\ref{t.optimal}.
We seek to prove that, for any $\aa\in\{1,\ldots,n'\}$ with $\Aranch(\aa)\subset\bb$,
\begin{align}
	\label{e.match.ode}
	\dotshock{\aa}=\frac12\sum_{\aa'=1}^{n'} \sgn(\shock{\aa'}-\shock{\aa}) \mmi_{\aa'} + \diner_\bb.
\end{align}
Referring to the definition of the optimal clusters (before Theorem~\ref{t.matching}\ref{t.matching.intermediate.}), we see that \eqref{e.match.ode} means that the shocks evolve the same way (in the backward time $s$) as the optimal clusters.
Since the shocks and optimal clusters share the same starting condition, namely $(\shock{\cc}(0))_{\cc=1}^n=(\xx_{\cc})_{\cc=1}^{n}=(\optimalc{\cc}(0))_{\cc=1}^n$, once \eqref{e.match.ode} is proven, the desired result $(\shock{\cc})_{\cc=1}^n=(\optimalc{\cc})_{\cc=1}^n$ follows.
To prove \eqref{e.match.ode}, fix $\aa,\bb$ with $\Aranch(\aa)\subset\bb$. 
Consider the set $\{\aa': \Aranch(\aa')=\bb\}$ and write it at $\{\aa': \Aranch(\aa')=\bb\}=[\aa^{\triangleleft},\ldots,\aa^{\triangleright}]\cap\Z$.
Let $\mpuu{\triangleleft}^\bb$ and $\mpuu{\triangleright}^\bb$ denote the respective slopes of $L^\bb_{\triangleleft}$ and $L^\bb_{\triangleright}$.
By Property~\ref{eu.shock.constant}, $\mpuu{\triangleleft}^\bb=\mpuu{\aa^{\triangleleft}}^-$, $\mpuu{\triangleright}^\bb=\mpuu{\aa^{\triangleright}}^+$, and $\mpuu{\aa'}^+=\mpuu{\aa'+1}^-$ for all $\aa'\in[\aa^{\triangleleft},\aa^{\triangleright}-1]$.
Use these properties to telescope the right hand side of \eqref{e.rankinehugoniot} as
\begin{align}
	\label{e.shock}
	\dotshock{\aa} 
	= 
	\frac{1}{2} \Big(	
		\mpuu{\triangleleft}^\bb 
		+
		\sum_{\aa'\in[\aa^{\triangleleft},\aa)} \big( -\mpuu{\aa'}^- + \mpuu{\aa'}^+ \big)
		+
		\sum_{\aa'\in(\aa,\aa^{\triangleright}]} \big( \mpuu{\aa'}^- - \mpuu{\aa'}^+ \big)
		+\mpuu{\triangleright}^\bb 
	\Big).
\end{align}
Next, combine \eqref{e.legendre.dual.intermediate} and \eqref{e.ratekpz.gradient} to get 
\begin{align}
	\label{e.mm.u.replation}
	\mmi{}_{\aa} 
	= 
	\mpuu{\aa}^- - \mpuu{\aa}^+,
	\quad
	\aa = 1,\ldots,n'.
\end{align}
Inserting \eqref{e.mm.u.replation} into \eqref{e.shock} gives
\begin{align}
	\label{e.shock.}
	\dotshock{\aa} 
	= 
	\frac{1}{2} \big( \mpuu{\triangleleft}^\bb + \mpuu{\triangleright}^\bb \big)
	+
	\sum_{\aa'\in[\aa^{\triangleleft},\aa^{\triangleright}]}\, \sgn(\shock{\aa'}-\shock{\aa}) \mmi_{\aa'}.
\end{align}
Let $\shock{\bb}:=\sum_{\cc\in\bb} \mm_\cc \shock{\cc}/\mm_\bb$.
Multiply both sides of \eqref{e.shock.} by $\mm_{\aa}$, sum both sides over all $\aa\in[\aa^{\triangleleft},\aa^{\triangleright}]$, and divide the result by $\mm_{\bb}$.
Doing so gives $\dotshock{\bb}=\frac{1}{2} (\mpuu{\triangleleft}^\bb +\mpuu{\triangleright}^\bb) $, which is constant.
Write the left hand side as $\dotshock{\bb}=\frac12(\ldots -\mm_{\bb\ominus 1}+\mm_{\bb\oplus 1}+\ldots)+\diner_\bb'$, for some constant $\diner_\bb'$.
Integrate $\dotshock{\bb}$ over $s\in[0,\termts]$, use $\shock{\bb}(0)=\optimal_\bb(0)$ and $\shock{\bb}(\termt)=0=\optimal_\bb(\termt)$, and compare the result with \eqref{e.optimal.frominertia}.
Doing so shows that $\diner_{\bb}'=\diner_\bb$.
Hence $\frac{1}{2} (\mpuu{\triangleleft}^\bb +\mpuu{\triangleright}^\bb)=\frac12(\ldots -\mm_{\bb\ominus 1}+\mm_{\bb\oplus 1}+\ldots)+\diner_\bb$.
Inserting this into \eqref{e.shock.} gives the desired result \eqref{e.match.ode}.

\subsection{Legendre transform}
\label{s.matching.legendre}
Here we prove the first statement in Theorem~\ref{t.matching}\ref{t.matching.legendre}, by showing that the Legendre transform of $\ratekpz(\vechv)$ gives $\momshe(\vecmm)$.
Once this is done, since $\ratekpz$ is strictly convex and since $\nabla\ratekpz:\hvspacec\to[0,\infty)^n$ is a homeomorphism, it will follow that $\momshe(\vecmm)$ is also strictly convex and is the Legendre transform of $\ratekpz(\vechv)$.
Still use the notation in Theorem~\ref{t.matching}.
Set $G := \vecmm\cdot\vechv - \ratekpz(\vechv)$, where the dot denotes the Euclidean inner product.
Our goal is to show $G=\momshe(\vecmm)$.
To this end, we devise a time-dependent version of $G$:
\begin{align}
	\label{e.matching.G}
	G(t) 
	:=
	\sum_{\cc=1}^{n} \mm_\cc \mps(t,\shock{\cc}(\termt-t))
	-
	\sum_{\bb\in\Branchs}
	\int_{t\tangent^{\bb}_{\triangleleft}/\termt}^{t\tangent^{\bb}_{\triangleright}/\termt}
	\d x \, \Big( \frac12 \mpu(t)^2 - \frac{1}{2} \big( \partial_x\parab(t) \big)^2 \Big).
\end{align}
At $t=\termt$, using $\mps(\termt)=\hf{}$, $\shock{\cc}(\termt-\termt)=\xx_\cc$, and Property~\ref{eu.shock.outside} verifies that $G(\termt)=G$. 
Using Properties~\ref{eu.shock.outside}--\ref{eu.shock.constant}, it is not hard to check that $\mps(t,\shock{\cc}(\termt-t))\to 0$ as $t\to 0$, for all $\cc$, and that $\mpu$ and $\partial_x\parab(t)$ are uniformly bounded on $\cone_\bb$ for all $\bb$.
These properties together give that $G(t)\to 0$ as $t\to 0$.
We seek to show that, for all $\intermt$ except when shocks merge,
\begin{align}
	\label{e.matching.goal}
	\big( \tfrac{\d~}{\d t}G \big)(\intermt)
	=
	\frac{1}{24} \sum_{\aa=1}^{n'} {\mmi_\aa}^3
	-
	\sum_{\aa=1}^{n'} \frac12 \mmi_\aa \, \dotshock{\aa}^2.
\end{align}
Once this is done, using $\shock{\cc}=\optimalc{\cc}$, integrating both sides over $t\in(0,\termts]$, and comparing the result with $\momshe(\vecmm)=\mom_{[0,\termts]}(\sum_{\cc}\mm_\cc\optimalc{\cc})$ (see \eqref{e.mom}) will yield the desired result $G=\momshe(\vecmm)$.

Let us differentiate the first sum in \eqref{e.matching.G} at $t=\intermt$ and simplify the result.
Take any $\intermt$ at which no shocks merge, write the first sum in \eqref{e.matching.G} as $\sum_{\aa=1}^{n'} \mmi_\aa \mps(t,\shock{\aa}(\termt-t))$, differentiate this expression in $t$ with the aid of Properties~\ref{eu.shock.diff}.
Doing so gives
\begin{align}
	\label{e.matching.G.sum1}
	\frac{\d~}{\d t}\text{(first sum in \eqref{e.matching.G})}\big|_{t=\intermt}
	=
	\sum_{\aa=1}^{n'} \mmi_\aa \cdot \big( \tfrac{1}{2} (\mpuu{\aa}^-)^2 - \dotshock{\aa}\cdot \mpuu{\aa}^- \big).
\end{align}
Combine \eqref{e.rankinehugoniot}, \eqref{e.ratekpz.gradient}, and \eqref{e.legendre.dual.intermediate} to get the relation
$
	\mpuu{\aa}^- = \frac{1}{2}\mmi_\aa + \dotshock{\aa}
$
and insert it into the right hand side of \eqref{e.matching.G.sum1}.
The result gives
$
	\frac{\d~}{\d t}\text{(first sum in \eqref{e.matching.G})}|_{t=\intermt}
	=
	\frac{1}{8} \sum_{\aa=1}^{n'} {\mmi_\aa}^3 - \sum_{\aa=1}^{n'}  \frac12 \mmi_\aa \,\dotshock{\aa}^2.
$

Next we treat the second sum in \eqref{e.matching.G}.
The contribution of differentiating the boundary points $t\tangent^{\bb}_{\triangleleft}/\termt$ and $t\tangent^{\bb}_{\triangleright}/\termt$ is zero, because the integrand evaluates to zero at those points by Property~\ref{eu.shock.outside}.
Next, recall from Property~\ref{eu.shock.constant} that $\mpu$ is piecewise constant within the integral in \eqref{e.matching.G}.
This gives 
\begin{align}
	\int_{t\tangent^{\bb}_{\triangleright}/\termt}^{t\tangent^{\bb}_{\triangleleft}/\termt} \d x \, \partial_t \tfrac12 (\mpu)^2
	=
	\frac12 \sum_{\aa:\Aranch(\aa)\subset\bb}((\mpuu{\aa}^+)^2-(\mpuu{\aa}^-)^2)\dotshock{\aa}.
\end{align}
Next, straightforward calculations verify the following relation.
(One can also use the machinery of entropy-entropy flux pairs to see it, which we will not do here.)
\begin{align}
	\label{e.matching.1}
	\int_{t\tangent^{\bb}_{\triangleright}/\termt}^{t\tangent^{\bb}_{\triangleleft}/\termt} \d x \, \partial_t \frac12 (\partial_x \parab(t))^2
	=
	\frac13 \big(\partial_x \parab(t,t\tangent^{\bb}_{\triangleleft}/\termt)\big)^3
	-
	\frac13 \big(\partial_x \parab(t,t\tangent^{\bb}_{\triangleright}/\termt)\big)^3.
\end{align}
By Property~\ref{eu.shock.outside}, the right hand side of \eqref{e.matching.1} is equal to $\frac{1}{3}\mpu(t,t\tangent^{\bb}_{\triangleleft}/\termt)^3-\frac{1}{3}\mpu(t,t\tangent^{\bb}_{\triangleright}/\termt)^3$.
Telescope the last expression into $\frac{1}{3}\sum_{\aa:\Aranch(\aa)\subset\bb}((\mpuu{\aa}^+)^3-(\mpuu{\aa}^-)^3)$.
So far, we have
\begin{align}
	\label{e.matching.G.sum2}
	\frac{\d~}{\d t}\text{(second sum in \eqref{e.matching.G})}\big|_{t=\intermt}
	=
	\sum_{\aa=1}^{n'} 
	\Big( 
		\frac{( (\mpuu{\aa}^+)^2 + (\mpuu{\aa}^-)^2 )(-\dotshock{\aa})}{2}
		-
		\frac{(\mpuu{\aa}^+)^3-(\mpuu{\aa}^-)^3}{3}
	\Big).
\end{align}
Inset \eqref{e.rankinehugoniot} into the right hand side of \eqref{e.matching.G.sum2} and simplify the result.
Doing so gives\\
$
	\frac{\d~}{\d t}\text{(second sum in \eqref{e.matching.G})}|_{t=\intermt}
	=
	-\sum_{\aa=1}^{n'}(\mpuu{\aa}^+-\mpuu{\aa}^-)^3/12,
$
which, by \eqref{e.mm.u.replation}, is equal to $-\sum_{\aa=1}^{n'}{\mmi_\aa}^3/12$.

Combining the results in the last two paragraphs gives \eqref{e.matching.goal} and hence completes the proof.

%%%%%%%%%%%%%%%%%%%%%%%%%%%%%%%%%%%%%%%%%%%%%%%%%%%%%%%%%%%%%%%%%%%%%%%%%%%%%%%%%%%%%%%%%%%%%%%%%%%%%%%%%%%%%%%%%%%%%%%%%%%%%%%

\appendix

\section{Basic properties}
\label{s.a.basic}

%\begin{proof}[Proof of the transformation \eqref{e.transformation.aBP}]
%Apply Tanaka's formula 
%$
%	\int_0^t \d s \, \delta_0(\X_\ii-\X_\jj)
%	=
%	-\int_0^t \d (\X_{\ii}-\X_{\jj}) \, \frac12 \sgn(\X_\ii-\X_\jj)
%	+
%	\frac{1}{2}|\X_\ii(s)-\X_\jj(s)| \, |_{s=0}^{s=t}
%$
%in \eqref{e.feynmankac}, note that $-\sum_{\ii<\jj}\int_0^t \d (\X_{\ii}-\X_{\jj}) \, \frac12 \sgn(\X_\ii-\X_\jj)$ has quadratic variation $t((\mm N)^3-(\mm N))/12$ under $\Ebm$, and use Girsanov's theorem. 
%Doing so gives \eqref{e.transformation.aBP}.
%We note that a similar application of Tanaka's formula was used in \cite{chen15} to access the one-point integer moments of the SHE.
%\end{proof}
%
%\begin{rmk}\label{r.conjugation}
%The transformation \eqref{e.transformation.aBP} can also be derived by conjugating the Hamiltonian of the delta Bose gas by its ground state; see \cite[Eq.~(5)--(6)]{ledoussal2022ranked}.
%\end{rmk}

\begin{proof}[Proof of the scaling identities~\eqref{e.rate.scaling}]
First, it is not hard to check that, for any $\lambda\in \mm\Psp(\R)$,
\begin{align}
	\label{e.scaling.relation}
	\cdf[\tfrac{1}{\mm}\scale_\mm\lambda](x)=\tfrac{1}{\mm}\cdf[\lambda](\mm x),
	\
	\quant[\tfrac{1}{\mm}\scale_m\lambda](a)=\tfrac{1}{\mm}\quant[\lambda](\mm a),
	\
	\Sgn[\tfrac{1}{\mm}\scale_\mm\lambda](x)=\tfrac{1}{\mm}\Sgn[\lambda](\mm x).
\end{align}
Using the preceding scaling relations for $\quant$ and $\Sgn$ in \eqref{e.rateq} and performing a change of variables $\mm a\mapsto a$ give the desired scaling relation $\rateq(\frac{1}{\mm}\scale_\mm\mu)=\rateq(\mu)/\mm^3$ for $\rateq$.
To prove the scaling relation for $\rate$, use \eqref{e.scaling.relation} in \eqref{e.rate} and call $\g(s,x):=\h(s,x/\mm)$.
After being simplified, the result reads
\begin{align}
	\label{e.scaling.1}
	\rate(\tfrac{1}{\mm}\scale_\mm\mu)
	=
	\sup_{\g} \Big\{ \frac{1}{\mm} \rateM(\mu,\g) - \frac{\mm}{2} \int_0^{\termt} \d s \, \ip{ \mu, (\partial_x\g)^2 } \Big\},
\end{align}
where the supremum runs over $\g\in\Cbsp^{1,1}([0,\termts],\R)$.
A supremum of this form can be expressed as a supremum of a Rayleigh quotient: For $u,v\in(0,\infty)$,
\begin{align}
	\label{e.rayleigh}
	\sup_{\g} \Big\{ u \rateM(\mu,\g) - \frac{v}{2} \int_{[0,\termts]} \d s \, \ip{ \mu, (\partial_x\g)^2 } \Big\} 
	= 
	\frac{
		\sup_{\g} (u\rateM(\mu,\g))^2
	}{
		2v\int_{[0,\termts]} \d s \, \ip{ \mu, (\partial_x\g)^2 }
	}.
\end{align}
with the convention that $ 0/0 := 0 $ and that $ b^2/0 := \infty $ for $ b \neq 0 $.
To see why, rewrite the supremum on the left hand side as a supremum over $ b\g $, with $ b \in\R $, and optimize over $ b $ for a fixed $ \g $.
Using the Rayleigh quotient expression in \eqref{e.scaling.1} gives $\rate(\tfrac{1}{\mm}\scale_\mm\mu)=\rate(\mu)/\mm^3$.
\end{proof}

\begin{proof}[Proof of Lemma~\ref{l.useful}\ref{l.useful.tailbd}--\ref{l.useful.quant.conti}]
To prove Part~\ref{l.useful.tailbd}, assume the contrary: There exist $ s_1,s_2,\ldots \in [0,\termts] $ and $ \e_0>0 $ such that $ \ip{\mu(s_k),\ind_{\R\setminus(-k,k)}} \geq \e_0>0 $ for all $k$.
After passing to a subsequence, we have $ s_k\to s_0\in[0,\termts] $.
Since $ \R\setminus(-k,k) $ is closed and since $ \mu(s) $ is continuous in $s$, $ \limsup_{\ell\to\infty} \ip{ \mu(s_\ell), \ind_{\R\setminus(-k,k)}} \leq \ip{\mu(s_0),\ind_{\R\setminus(-k,k)}} $ for every $ k\in\Z_{>0} $.
The left hand side is at least $ \e_0 $.
Sending $ k\to\infty $ gives $ 0<\e_0 \leq \liminf_{k\to\infty}\ip{\mu(s_0),\ind_{\R\setminus[-k,k]}} $, contradicting $\mu(s_0)\in\Psp(\R)$.

To prove Part~\ref{l.useful.quant.conti}, write $ \quant[\mu(s)](a) = \quant(s,a) $ and let $ \omega(s,a) := \limsup_{s'\to s} \quant(s',a) - \liminf_{s'\to s} \quant(s',a) $ denote the oscillation in time of $ \quant $ at $ (s,a) $.
Fix $k\in \Z_{>0}$. By \eqref{e.dist},
\begin{align}
\begin{split}
	\dist&(\mu(s_1),\mu(s_2)) 
\\
	&\geq 2^{-k} \min\Big\{
		1,\int_0^1 \d a \ |\quant(s_1,a)-\quant(s_2,a)| \ind\{|\quant(s_1,a)|\leq k\} \ind\{|\quant(s_2,a)|\leq k\}
	\Big\}.
\end{split}
\end{align}
This inequality and the continuity of $ \mu(s) $ in $s$ imply that the set $ \{a\in[0,1]: \omega(s,a)>0, \text{ for some } s \text{ with } |\quant(s,a)|<k\} $ has zero Lebesgue measure.
Combining this with the result of Part~\ref{l.useful.tailbd} gives the desired result.
\end{proof}

\begin{proof}[Proof of Properties~\ref{eu.rate0.timeadd}--\ref{eu.rate0.stationary=0} after \eqref{e.rate0.j}]
To prove Property~\ref{eu.rate0.timeadd}, observe that ${\Cbsp}^{1,1}([s_1,s_3],\R)$ can be embedded into ${\Cbsp}^{1,1}([s_1,s_2],\R)$ and ${\Cbsp}^{1,1}([s_2,s_3],\R)$, so the inequality $\geq$ follows.
For the reverse inequality, take any $\h_{12}\in{\Cbsp}^{1,1}([s_1,s_2],\R)$ and $\h_{23}\in{\Cbsp}^{1,1}([s_2,s_3],\R)$ and concatenate them as follows:
Take any $ \varphi \in \Csp^\infty(\R) $ that is increasing, with $ \varphi|_{(-\infty,0]} =0 $ and $ \varphi|_{[1,\infty)} =1 $, set $ \varphi_\delta(s) := \varphi(s/\delta) $, and consider 
\begin{align}
	\h_\delta(s,x) := \h_{12}(s,x)\varphi_\delta(s_2-s) + \h_{23}(s,x)\varphi_\delta(s-s_2) \in \Cbsp^{1,2}([s_1,s_3],\R).
\end{align}
Indeed, $ \int_{s_1}^{s_3}\d s \,\ip{\mu,(\partial_x\h_\delta)^2} \to \int_{s_1}^{s_2}\d s\,\ip{\mu,(\partial_x\h_{12})^2} + \int_{s_2}^{s_3}\d s \,\ip{\mu,(\partial_x\h_{23})^2} $ as $ \delta\to 0 $.
Next, write 
%$ 
%	\int_{s_1}^{s_3} \d s\, \ip{\mu,\partial_s\h_\delta} 
%	- \int_{s_1}^{s_2} \d s\, \ip{\mu,\partial_s\h_{12}} 
%	- \int_{s_2}^{s_3} \d s\, \ip{\mu,\partial_s\h_{23}} 
%$ 
%as 
%$ 
%	- \int_{s_2-\delta}^{s_2}\d s\,\ip{\mu,\h_{12} \partial_s\varphi_\delta}
%	+ \int_{s_2}^{s_2+\delta}\d s\,\ip{\mu,\h_{23} \partial_s\varphi_\delta} 
%	+ O(\delta) 
%$.
\begin{align}
	\int_{s_1}^{s_3} &\d s\, \ip{\mu,\partial_s\h_\delta} 
	- \int_{s_1}^{s_2} \d s\, \ip{\mu,\partial_s\h_{12}} 
	- \int_{s_2}^{s_3} \d s\, \ip{\mu,\partial_s\h_{23}} 
\\
	\label{e.appdx1}
	=&
	- \int_{s_2-\delta}^{s_2}\d s\,\ip{\mu,\h_{12} \partial_s\varphi_\delta}
	+ \int_{s_2}^{s_2+\delta}\d s\,\ip{\mu,\h_{23} \partial_s\varphi_\delta} 
	+ O(\delta) 	
\end{align}
It is not hard to check that \eqref{e.appdx1} converges to $ -\ip{\mu(s_2),\h_{12}(s_2)}+\ip{\mu(s_2),\h_{23}(s_2)} $ as $\delta\to 0$.
The preceding results together verify Property~\ref{eu.rate0.timeadd}.

Property~\ref{eu.rate0.timemono} follows from Property~\ref{eu.rate0.timeadd} and the property that $ \rate^\circ{}_{[s',s'']} \geq 0 $; the latter property is readily seen from \eqref{e.rate0.j}.
Property~\ref{eu.rate0.convex} follows since the expression within the supremum in \eqref{e.rate0.j} is linear in $ \mu $.
Property~\ref{eu.rate0.spacetranslation} follows by renaming the test function $ \h(s,x-y):=\til{h}(s,x) $ in \eqref{e.rate0.j}.
Property~\ref{eu.rate0.stationary=0} follows since for a time-independent $ \mu=\lambda\in\Psp(\R) $ the first two terms within the supremum in \eqref{e.rate0.j} cancel with each other, and the last term is nonpositive.
\end{proof}

\begin{proof}[Proof of \eqref{e.mom.rep}]
We consider $\mm=1$ only; the result for general $\mm>0$ follows from the result for $\mm=1$ through the scaling argument in the proof of \eqref{e.rate.scaling}.
Without loss of generality, assume $\partial_s\quant[\mu]\in\Lsp^2([0,\termts]\times[0,1])$; otherwise both sides of \eqref{e.mom.rep} are $+\infty$ by definition.
Write $ \quant[\mu(s)](a) = \quant(s,a) $ and $ \Sgn[\mu(s)](x) = \Sgn(s,x) $.
Expand the integrand in \eqref{e.rateq} and insert the result into the Left Hand Side (LHS) of \eqref{e.mom.rep} to get
\begin{align}
	\label{e.A}
	\text{LHS of \eqref{e.mom.rep}}
	=
	\frac{\termt}{24} + \Ip{ \mu(s)^{\otimes 2}, \tfrac{1}{4}|x-x'| }\Big|^{s=\termt}_{s=0}
	-
	\int_{0}^{\termts} \d s \int_{0}^{1} \d a \, \frac12 (\partial_s\quant)^2 + A_2 - A_3,
\end{align}
where
%$
%	A_1 := \int\!\!{}_{[0,\termts]} \d s \int\!\!{}_{[0,1]} \d a \, \frac12 (\partial_s\quant)^2,
%$
$
	A_2 := \int\!\!{}_{[0,\termts]} \d s \int\!\!{}_{[0,1]} \d a \, (\partial_s\quant)\, \Sgn,
$
and
$
	A_3 := \int\!\!{}_{[0,\termts]} \d s \int\!\!{}_{[0,1]} \d a \, \frac12 \Sgn^2.
$
%The term $-A_1$ contributes to the last term in \eqref{e.mom.rep}.
For $A_2$, we claim that we can replace $\Sgn(s,a)$ with $(\tfrac12 - a)$ to get
\begin{align}
	\label{e.A2.claim}
	A_2 
	= 
	\int_{0}^{\termts} \d s \int_0^1 \d a \, \big( \partial_s\quant[\mu] \big)\cdot\big(\tfrac12 - a\big).
\end{align}
Recall $a_\pm[\mu(s)](a)$ from \eqref{e.a+-} and write them as $a_\pm(s,a)$.
Given \eqref{e.Sgn.id}, this claim follows if we can show, for almost every $(s,a)$, $(\partial_s\quant)(s,\Cdot)$ is almost everywhere a constant on $(a_-(s,a),a_+(s,a))$.
Fix any $(s_0,a_0)$ such that $\quant$ is differentiable in $s$ at $(s_0,a_0)$ and write $a_{\pm}(s_0,a_0)=:a_{0,\pm}$.
For every $a\in(a_{0,-},a_0)$, $\quant(s_0,a)=\quant(s_0,a_0)$ and $\quant(s,a)\leq\quant(s,a_0)$ for all $s$.
If $\quant$ is differentiable in $s$ at $(s_0,a)$, the preceding properties force $(\partial_s\quant)(s_0,a)=(\partial_s\quant)(s_0,a_0)$.
The same conclusion holds for $a\in(a_0,a_{0,+})$.
Under the assumption that $\partial_s\quant[\mu]\in\Lsp^2([0,\termts]\times[0,1])$, $\quant$ is differentiable in $s$ Lebesgue a.e., so the claim follows.
Evaluate the $s$ integral in \eqref{e.A2.claim}, use \eqref{e.Sgn.id.int} for $\f(x)=x$ in reverse, and use Lemma~\ref{l.useful}\ref{l.useful.inverseCDF} in reverse to get $A_{2}=\ip{\mu(s),x\Sgn[\mu(s)]}|^{s=\termt}_{s=0}$.
Recalling $\Sgn[\mu(s)]$ from \eqref{e.Sgn}, we recognize the last expression as $A_{2}=\ip{ \mu(s)^{\otimes 2}, \frac{1}{2}x\,\sgn(x'-x) } \,|^{s=\termt}_{s=0}$.
Since the measure $\mu(s)^{\otimes 2}$ is symmetric in $x$ and $x'$, we can symmetrize the function $\frac{1}{2}x\,\sgn(x'-x)$ to get 
\begin{align}
	\label{e.A2.}
	A_{2}
	=
	\ip{ \mu(s)^{\otimes 2}, \tfrac{1}{4}(x-x')\sgn(x'-x) } \, \big|^{s=\termt}_{s=0}
	=
	-\ip{ \mu(s)^{\otimes 2}, \tfrac{1}{4}|x-x'| } \, \big|^{s=\termt}_{s=0}.
\end{align}
As for $A_3$, the identity \eqref{e.Sgn.id} gives 
\begin{align}
	A_3 
	= 
	\frac12 \int_{0}^{\termts} \d s \Big( \int_{0}^{1} \d a 
		\,\big(\tfrac12-a\big)^2 - \sum \int_{a_-}^{a_+}\d a \, \Big(\big(\tfrac12-a\big)^2 - \big(\tfrac{1-a_--a_+}{2}\big)^2\Big)
	\Big), 
\end{align}
where the sum runs over all pairs of $(a_-,a_+)$ such that $a_-<a_+$.
Evaluating the integrals and simplifying the result give
\begin{align}
	\label{e.A3}
	A_3
	=
	\frac{\termt}{24} - \int_{0}^{\termts} \d s \, \sum \frac{1}{24}(a_+-a_-)^3
	=
	\frac{\termt}{24} - \int_{0}^{\termts} \d s \,\sum_{x}\frac{1}{24}\ip{\mu,\ind_{\{x\}}}^3,
\end{align}
where the sum over $x$ runs over atoms of $\mu(s)$. 
Inserting \eqref{e.A2.} and \eqref{e.A3} into \eqref{e.A} gives \eqref{e.mom.rep}.
\end{proof}

\section{Proof of \eqref{e.lwbd.plapprox}}
\label{s.a.approx}
We take $[s',s'']=[0,\termts]$ (whence $\rateq_{[0,\termts]}(\Cdot):=\rateq(\Cdot)$) to simplify notation.

We say \textbf{$\mu$ well approximates $\mu'$ to within $\e$} if $\dist_{[0,\termts]}(\mu,\mu')<\e$ and $|\rateq(\mu)-\rateq(\mu')|<\e$.
We say \textbf{a class of deviations well approximates $\mu'$} if, for every $\e>0$, there exists a deviation in that class that well approximates $\mu'$  to within $\e$.

As the first step, we well approximate $\mu$ by clustering deviations.
For $n\in\Z_{>0}$, apply the procedure in Section~\ref{s.basic.divide} with $\mm_1=\ldots=\mm_n=1/n$ to get $\mu_c$, $\cc=1,\ldots,n$.
As explained there, each $\mu_\cc(s)$ is continuous in $s$.
Let $x_{\gamma,\cc}(s) := n \ip{\mu_\cc(s),x} $, which belongs to $\Csp[0,\termts]$ thanks to the continuity of $\mu_\cc$ and thanks to Assumption~\ref{assu.truncation}.
Set $\gamma^{(n)}(s):= \frac{1}{n}\sum_{\cc} \delta_{x_{\gamma,\cc}(s)}$, which is clustering.
Consider (twice of) the total variation norm
$	
	\norm{ \lambda-\lambda' }_\mathrm{tv}
	:=
	\sup_{x\in\R} | \cdf[\lambda](x) - \cdf[\lambda'](x) |.
$
We claim that 
$
	\norm{ \gamma^{(n)}(s)-\mu(s) }_\mathrm{tv}
	\leq
	1/n.
$
To see why, consider $S_\cc:=\{ x: \cdf[\mu(s)](x) \in [(\cc-1)/n,\cc/n) \} $ for $\cc=1,\ldots,n$ separately.
Within each $S_\cc$, the function $\cdf[\gamma^{(n)}(s)]$ makes a single jump from the level $(\cc-1)/n$ to the level $\cc/n$; this is seen from the construction of $\gamma^{(n)}$.
Hence the difference $|\cdf[\gamma^{(n)}(s)](x) - \cdf[\mu(s)](x)|$ is at most $1/n$, so the claim follows.
We have that $\gamma^{(n)}(s)\to\mu(s)$ in the total variation norm, uniformly over $s\in[0,\termts]$.
This implies $\dist_{[0,\termt]}(\gamma^{(n)},\mu)\to 0$.
Next, we show $\rateq(\gamma^{(n)})\to\rateq(\mu)$.
Note that by construction,
$
	\quant[\gamma^{(n)}(s)]|_{a\in I_\cc} = \gamma_\cc(s) := n \ip{\mu(s),x} = n \int_{I_\cc} \d a\, \quant[\mu(s)](a),
$
where $ I_\cc := [(\cc-1)/n,\cc/n) $.
Namely, $\quant[\gamma^{(n)}(s)]$ is obtained by averaging $\quant[\mu(s)]$ over intervals of length $1/n$.
By using this property and $\partial_s\quant[\mu]\in\Lsp^2([0,\termts]\times[0,1])$, which follows from the assumption $\rate(\mu)=\rateq(\mu)<\infty$, it is not hard to check that $\partial_s\quant[\gamma^{(n)}] \to \partial_s\quant[\mu]$ in $\Lsp^2([0,\termts]\times[0,1])$.
Also, $\norm{ \Sgn[\gamma^{(n)}(s)]-\Sgn[\mu(s)]}_\infty \leq \norm{\frac12\sgn}_\infty\norm{\gamma^{(n)}(s)-\mu(s)}_\mathrm{tv} = \frac12\norm{\gamma^{(n)}(s)-\mu(s)}_\mathrm{tv} \leq 1/(2n) $.
Combining these properties gives $ \rateq(\gamma^{(n)}) \to \rateq(\mu) $.

As the second step, writing $\gamma^{(n)}=:\gamma$ to simplify notation and fixing an arbitrary $\e>0$, we well approximate $\gamma$ by a finitely-changing-clustering $\zeta$ to within $\e$.
A deviation $\zeta$ being \textbf{finitely-changing-clustering} means that it is clustering and that there exist $0=s_0<s_1<\ldots<s_\ell=\termt$ such that, within each $(s_{k-1},s_k)$, each pair of clusters either never touch or completely coincide.
We will construct a sequence $\zeta^{(1)},\ldots,\zeta^{(n)}$ of clustering deviations that satisfy the following conditions.
\begin{enumerate}[leftmargin=20pt, label=(\roman*)]
\item \label{eu.finitely.changing.1}
$\dist_{[0,\termts]}(\zeta^{(j-1)},\zeta^{(j)})\leq \e/n$ and $|\rateq(\zeta^{(j-1)})-\rateq(\zeta^{(j)})|\leq \e/n$, with the convention $\zeta^{(0)}:=\gamma$.
\item \label{eu.finitely.changing.2}
Each $\zeta^{(j)}$ has $n$ clusters $\zeta^{(j)}_{1}\leq\ldots\leq\zeta^{(j)}_n$ of mass $1/n$, and the finitely-changing property holds up to the index $j$.
More precisely, there exist $0=s_0<s_1<\ldots<s_\ell=\termt$, which may depend on $j$, such that within each $(s_{k-1},s_k)$, each pair in $\zeta^{(j)}_1,\ldots,\zeta^{(j)}_j$ either never touches or completely coincides.
\end{enumerate}
Once constructed, $\zeta^{(n)}=:\zeta$ gives the finitely-change clustering deviation.

We now construct the $\zeta^{(j)}$s by induction on $j$.
For $j=1$, simply take $\zeta^{(1)}=\gamma$.
Assume $\zeta^{(1)},\ldots,\zeta^{(j-1)}$ have been constructed.
We will construct $\zeta^{(j)}=:\zeta^\new$ out of $\zeta^{(j-1)}=:\zeta^\old$.
To begin the construction, keep all but the $j$th clusters unchanged: $\zeta^\new_{j'} := \zeta^\old_{j'}$ for all $j'\neq j$.
Next, to construct the $j$th cluster for $\zeta^\new$, consider $O=\{s\in(0,\termts):\zeta^\old_{j-1}(s)<\zeta^\old_{j}(s)\}$, which is open, and write $O$ as the union of countable disjoint open intervals $(a_1,b_1),(a_2,b_2),\ldots$.
By Condition~\ref{eu.finitely.changing.1}, $\rateq(\zeta^\old)<\infty$, so $\dot{\zeta}^\old_\cc\in\Lsp^2[0,\termts]$ for all $\cc$.
Given this property, find an $\ell_0$ large enough such that, with $O'':=\cup_{\ell>\ell_0} (a_\ell,b_\ell)$,
\begin{align}
	\label{e.finitelychanging.1}
	&&\int_{O''} \d s \, \sum_{\cc=j-1,j} \frac{1}{n}\big| \dot{\zeta}^{\old}_{\cc} \big| \leq \frac{\e}{n},
	&&
	\int_{O''} \d s \, \frac{1}{2n}\sum_{\cc=j-1,j}  \big( \big|\dot{\zeta}^{\old}_{\cc} \big|+ 1 \big)^2 \leq \frac{\e}{n}.&
\end{align}
Set $O':=(a_1,b_1)\cup\ldots\cup(a_{\ell_0},b_{\ell_0})$.
Keep the $j$th cluster unchanged within $O'$ and perturb it to match the $(j-1)$th cluster outside $O'$.
More explicitly, $\zeta^\new_{j}|_{O'} := \zeta^\old_{j}|_{O'}$ and $\zeta^\new_{j}|_{[0,\termts]\setminus O'} := \zeta^\old_{j-1}|_{[0,\termts]\setminus O'}$.

We next check that the $\zeta^\new$ so constructed satisfies the required conditions.
First, given that $\zeta^\old$ satisfies Condition~\ref{eu.finitely.changing.2} up to the index $j-1$, it is not hard to check that $\zeta^\new$ satisfies Condition~\ref{eu.finitely.changing.2} up to the index $j$.
Move on to checking Condition~\ref{eu.finitely.changing.1}.
Recall that $\zeta^\new$ differs from $\zeta^\old$ only at the $j$th cluster.
Further, the difference occurs only on $O''$.
%To see why, note that $\zeta^\new_j|_{O'} = \zeta^\old_j|_{O'}$  by construction and that $\zeta^\old_{j}|_{[0,\termts]\setminus O}=\zeta^\old_{j-1}|_{[0,\termts]\setminus O}=\zeta^\new_j|_{[0,\termts]\setminus O}$ by the definition of $O$ and by the construction of $\zeta^\new$.
On $O''$, where $\zeta^\new_j$ and $\zeta^\old_j$ differ, we have $\zeta^\new_j|_{O''} = \zeta^\old_{j-1}|_{O''}$.
These properties together with \eqref{e.clustering.rateq} and $|\Sgn[\Cdot]|\leq 1/2$ give 
\begin{align}
	\label{e.finitelychanging.2}
	\sup_{s\in[0,\termts]}\frac{1}{n}\sum_{\cc=1}^{n}\big|\zeta^\new_\cc(s)-\zeta^\old_\cc(s)\big|
	&\leq
	\int_{O''} \d s\, \frac{1}{n}\big|\dot{\zeta}^\old_{j-1}-\dot{\zeta}^\old_{j} \big|\,,
\\
	\label{e.finitelychanging.3}
	|\rateq(\zeta^\new)-\rateq(\zeta^\old)|
	&\leq 
	\int_{O''} \d s \, \frac{1}{2n} \sum_{\cc=j-1,j} \big( |\dot\zeta^\old_{\cc}| + \tfrac12 \big)^2.
\end{align}
By \eqref{e.finitelychanging.1}, the right hand sides of \eqref{e.finitelychanging.2}--\eqref{e.finitelychanging.3} are bounded by $\e/n$.
By \eqref{e.wass.coupling}, the left hand side of \eqref{e.finitelychanging.2} bounds $\dist_{[0,\termts]}(\zeta^\new,\zeta^\old)$ from above.

Finally, we well approximate $\zeta$ by PL-clustering deviations.
Let $0=s_0<s_1<\ldots<s_\ell=\termt$ prescribe the intervals on which the clusters of $\zeta$ either never touch or completely coincide.
Further partition each $[s_{k-1},s_k]$ into smaller intervals of equal length.
Linearly interpolate the trajectories of the clusters of $\zeta$ with respect to the smaller intervals.
Use the resulting piecewise linear trajectories to build a PL-clustering $\xi$.
Within each $(s_{k-1},s_{k})$, since the clusters of $\zeta$ either never touch or completely coincide, the same property holds for $\xi$.
By using this property, \eqref{e.wass.coupling}, and \eqref{e.clustering.rateq}, it is not hard to check that, as the mesh of the sub partitions tends to zero, $\dist_{[0,\termts]}(\xi,\zeta)\to 0$ and $\rateq(\xi)\to\rateq(\zeta)$.

\section{Girsanov's transform}
\label{s.a.girsanov}
Here, we pack Girsanov's transform in ways convenient for our applications.
Fix $[s',s'']\subset[0,\termts]$, let $\P$ be the law of $(\X_{\ii}(s))_{s\in[Ts',Ts''],\ii=1,\ldots,N}$ under \eqref{e.aBP} given $(\X_\ii(Ts'))_{\ii=1,\ldots,N}$, take deterministic $v_\ii$ and $J_\ii\subset\{1,\ldots,N\}$, $\ii=1,\ldots,N$, and consider the law $\law$ such that
\begin{align}
	\text{under } \law, 
	\quad
	\d \X_\ii = \sum_{\jj\in\{1,\ldots,N\}\setminus J_\ii} \frac{1}{2} \sgn(\X_\jj-\X_\ii) \,\d s + N v_{\ii} \,\d s + \d\bm_{\ii},
	\
	s\in[Ts',Ts''],
\end{align}
with the same given $(\X_\ii(Ts'))_{\ii=1,\ldots,N}$.

\begin{lem}
\label{l.girsanov}
\begin{enumerate}[leftmargin=20pt, label=(\alph*)]
\item[]
\item \label{l.girsanov.1}
Assume $J_\ii=\emptyset$ for all $\ii$ and take any $r'\in(0,\infty)$.
For any measurable $\calE\subset\Csp([s',s''],\R)^N$ with $\calE\subset\{ |\bm_{\ii}(Ts)-\bm_{\ii}(Ts')| < NTr', \forall s\in[s',s''],\ii=1,\ldots,N \}$,
\begin{align}
	\label{e.girsanov.1}
	\frac{1}{N^3 T}\log\P\big[\calE\big]
	\geq
	\frac{1}{N^3 T}\log\law\big[\calE\big]
	-
	\frac{(s'-s)}{2N}\sum_{\ii=1}^N v_\ii^2
	-
	\frac{r'}{N}\sum_{\ii=1}^N |v_\ii|.
\end{align}
\item \label{l.girsanov.2}
For any $p>1$ and any measurable $\calE\subset\Csp([s',s''],\R)^N$,
\begin{align}
	\label{e.girsanov.2}
	\frac{1}{N^3T}\log\P\big[\calE\big]
	\geq
	\frac{p}{N^3T}\log\law\big[\calE\big]
	-
	\frac{p\,(s''-s')}{p-1}\frac{1}{2N}\sum_{\ii=1}^N \Big( |v_\ii| + \frac{|J_{\ii}|^2}{2N} \Big)^2.
\end{align}
\end{enumerate}
\end{lem}
\begin{proof}
Set 
$
	A:= \int_{Ts'}^{Ts''}\sum_{\ii=1}^N \d \bm_\ii\,( Nv_\ii - \sum_{\jj\in J_\ii} \sgn(\X_{\jj}-\X_{\ii})/2 )
$
and let
$
	\ip{A}
	=
	\int_{Ts'}^{Ts''}\sum_{\ii=1}^N \d s\,( Nv_\ii - \sum_{\jj\in J_\ii} \sgn(\X_{\jj}-\X_{\ii})/2 )^2
$
denote its quadratic variation.
Slightly abusing notation, we write the \emph{expectation} under $\law$ also as $\law$.
Girsanov's transform gives
$ 
	\P[\calE] = \law[\ind_\calE\exp(A-\ip{A}/2)].
$
For Part~\ref{l.girsanov.1}, using
$
	\ind_{\calE}A\geq-N^2Tr'\sum_{\ii}|v_{\ii}|
$
and
$
	\ip{A} = N^{2}T(s''-s')\sum_{\ii} v_\ii^2
$
gives the desired result.
For Part~\ref{l.girsanov.2}, use H\"{o}lder's inequality
\begin{align}
	\label{e.holder}
	\law[\calE]
	\leq
	\law[\ind_\calE \exp(A-\tfrac{1}{2}\ip{A})]^{1/p}
	\law[\exp(\tfrac{1}{(p-1)}(-A+\tfrac{1}{2}\ip{A}))]^{(p-1)/p},
\end{align}
evaluate the last expectation
\begin{align}
	\law\Big[e^{\frac{1}{p-1}(-A+\frac12\ip{A})}\Big]
	=
	\law\Big[\exp\Big(\frac{p}{2(p-1)^2} \int_{Ts'}^{Ts''}\d s \sum_{\ii=1}^{N} \Big(Nv_\ii + \sum_{\jj\in J_\ii} \frac{1}{2} \sgn(\X_{\jj}-\X_{\ii}) \Big)^2\, \Big],
\end{align}
bound the summand in $\sum_{\ii}$ by $(N|v_{\ii}|+|J_{\ii}|/2)^2$, and
replace $\law[\ind_\calE \exp(A-\ip{A}/2)]$ with $\P[\calE]$ in \eqref{e.holder}.
Doing so gives
\begin{align}
	\law[\calE]
	\leq 
	\P[\calE]^{1/p} \,\exp\Big(\frac{(s''-s')T}{2(p-1)} \sum_{\ii=1}^{N} \Big( N|v_{\ii}|+\frac{|J_{\ii}|}{2N}\Big)^2\Big).
\end{align}
Applying $\frac{p}{N^3T}\log(\,\Cdot\,)$ to both sides and simplifying the result give \eqref{e.girsanov.2}.
\end{proof}

\bibliographystyle{alpha}
\bibliography{kpz-ldp,moments,rank}

\end{document}